\documentclass{amsart}
\date{\today}
\usepackage{latexsym,amsmath,amsthm,amsfonts,amscd,amssymb,mathdots}
\usepackage{stmaryrd}
\usepackage[mathscr]{eucal} 
\usepackage{mathrsfs}
\usepackage[utf8]{inputenc}
\usepackage[all,cmtip]{xy}
\usepackage{hyperref}
\usepackage{graphics}
\usepackage{lscape}
\usepackage{enumerate}
\usepackage{array}
\usepackage{framed}
\usepackage{mathtools}
\usepackage{epsfig,color}
\usepackage{hyperref}
\hypersetup{colorlinks}

\usepackage{color} 

\definecolor{darkred}{rgb}{1,0,0} 
\definecolor{darkgreen}{rgb}{0,1,0}
\definecolor{darkblue}{rgb}{0,0,1}

\hypersetup{colorlinks,
linkcolor=darkblue,
filecolor=darkgreen,
urlcolor=darkred,
citecolor=darkgreen}

\theoremstyle{plain}  
\newtheorem{theorem}{Theorem}[section]

\newtheorem*{theorem*}{Theorem}

\newtheorem{corollary}[theorem]{Corollary}
\newtheorem{lemma}[theorem]{Lemma}
\newtheorem{proposition}[theorem]{Proposition}
\newtheorem{conjecture}[theorem]{Conjecture}
\newtheorem{tech-lemma}[theorem]{Technical Lemma}
\newtheorem{definition}[theorem]{Definition}
\theoremstyle{remark}
\newtheorem{example}[theorem]{Example}
\newtheorem{notation}[theorem]{Notation}
\newtheorem{remark}[theorem]{Remark}
\newtheorem*{remark*}{Remark}

\newtheorem*{claim*}{Claim}

\newtheoremstyle{TheoremForIntro} 
        {.6em}{.6em}              
        {\itshape}                      
        {}                              
        {\bfseries}                     
        {.}                             
        { }                             
        {\thmname{#1}\thmnote{ \bfseries #3}}
    \theoremstyle{TheoremForIntro}
    \newtheorem{TheoremIntro}[theorem]{Theorem}
    \newtheorem{PropositionIntro}[theorem]{Proposition}
    \newtheorem{CorollaryIntro}[theorem]{Corollary}

\numberwithin{equation}{section}
\renewcommand{\leq}{\leqslant}

\renewcommand{\geq}{\geqslant}

\newcommand{\R}{\mathbb{R}}

\newcommand{\Z}{\mathbb{Z}}
\newcommand{\C}{\mathbb{C}}
\newcommand{\HH}{\mathbb{H}}
\newcommand{\V}{\mathbb{V}}
\newcommand{\W}{\mathbb{W}}

\newcommand{\cE}{{\mathcal E}}

\newcommand{\cG}{{\mathcal G}}
\newcommand{\cH}{{\mathcal H}}

\newcommand{\cK}{{\mathcal K}}

\newcommand{\cM}{{\mathcal M}}

\newcommand{\cO}{{\mathcal O}}

\newcommand{\cR}{{\mathcal R}}
\newcommand{\cS}{{\mathcal S}}

\newcommand{\cZ}{{\mathcal Z}}

\newcommand{\dbar}{\bar{\partial}}

\newcommand{\suchthat}{\;\;|\;\;}

\newcommand{\PSL}{\mathrm{PSL}}

\newcommand{\rS}{\mathrm{S}}

\newcommand{\rP}{\mathrm{P}}
\newcommand{\U}{\mathrm{U}}
\newcommand{\GL}{\mathrm{GL}}
\newcommand{\rL}{\mathrm{L}}
\newcommand{\SL}{\mathrm{SL}}
\newcommand{\SO}{\mathrm{SO}}
\newcommand{\rO}{\mathrm{O}}
\newcommand{\Sp}{\mathrm{Sp}}

\newcommand{\Bun}{\mathrm{Bun}}
\newcommand{\Fuch}{\mathrm{Fuch}}
\newcommand{\Hit}{\mathrm{Hit}}

\newcommand{\Sym}{\mathrm{Sym}}
\newcommand{\sym}{\mathrm{sym}}

\newcommand{\fgl}{\mathfrak{gl}}

\newcommand{\fso}{\mathfrak{so}}

\newcommand{\fg}{\mathfrak{g}}
\newcommand{\fh}{\mathfrak{h}}

\newcommand{\fl}{\mathfrak{l}}
\newcommand{\fm}{\mathfrak{m}}

\newcommand{\fo}{\mathfrak{o}}
\newcommand{\fp}{\mathfrak{p}}

\newcommand{\ft}{\mathfrak{t}}
\newcommand{\fu}{\mathfrak{u}}

\DeclareMathOperator{\ad}{ad}
\DeclareMathOperator{\Ad}{Ad}
\DeclareMathOperator{\Aut}{Aut}

\DeclareMathOperator{\tr}{tr}
\DeclareMathOperator{\rk}{rk}

\DeclareMathOperator{\Hom}{Hom}
\DeclareMathOperator{\End}{End}

\DeclareMathOperator{\Id}{Id}

\DeclareMathOperator{\Rep}{\mathcal{R}}
\newcommand{\G}{\mathrm{G}}
\newcommand{\rH}{\mathrm{H}}

\newcommand{\noi}{\noindent}

\newcommand{\smtrx}[1]{\left (\begin{smallmatrix}#1\end{smallmatrix}\right)}
\newcommand{\floor}[1]{\left\lfloor #1\right\rfloor}

\let\oldmarginpar\marginpar
\renewcommand\marginpar[1]{\oldmarginpar{\tiny\bf\begin{flushleft} #1
\end{flushleft}}}

\begin{document}

%
%

\title[$\SO(p,q)$-Higgs bundles and Higher Teichmüller components]
{$\SO(p,q)$-Higgs bundles and Higher Teichmüller components}
%
%


\author[M. Aparicio]{Marta Aparicio-Arroyo}
\address{M. Aparicio,
\newline\indent Raet | HR software and services\newline\indent
Avenida de Bruselas 7, 28108 Alcobendas, Madrid, Spain}
\email{Marta.Aparicio@raet.com}

\author[S. Bradlow]{Steven Bradlow}
\address{S. Bradlow,
\newline\indent
Department of Mathematics, University of Illinois at Urbana-Champaign\newline\indent
Urbana, IL 61801, USA}
\email{bradlow@math.uiuc.edu}

\author[B. Collier]{Brian Collier}
\address{B. Collier,
\newline\indent Department of Mathematics, University of Maryland\newline\indent College Park, MD 20742, USA}
\email{briancollier01@gmail.com}

\author[O. Garc\'{i}a-Prada]{Oscar Garc\'{i}a-Prada}
\address{O. Garc\'{i}a-Prada,
\newline\indent Instituto de Ciencias Matem\'aticas,  CSIC-UAM-UC3M-UCM, \newline\indent 
Nicol\'as Cabrera, 13--15, 28049 Madrid, Spain}
\email{oscar.garcia-prada@icmat.es}

\author[P.~B.\ Gothen]{Peter B.\ Gothen}
\address{P.~B. Gothen,
\newline\indent Centro de Matemática da Universidade do Porto, 
    \newline\indent Faculdade de Ci\^encias da Universidade do Porto, 
    \newline\indent Rua do Campo Alegre s/n, 4169-007 Porto, Portugal}
\email{pbgothen@fc.up.pt}

\author[A. Oliveira]{Andr\'e Oliveira}
\address{A. Oliveira, 
\newline\indent 
Centro de Matem\'atica da Universidade do Porto,
     \newline\indent Faculdade de Ci\^encias da Universidade do Porto, 
     \newline\indent
     Rua do Campo Alegre s/n, 4169-007 Porto, Portugal \newline\indent \textsl{On leave from:}\newline\indent Departamento de Matem\'atica, Universidade de Tr\'as-os-Montes e Alto Douro, UTAD,\newline\indent
Quinta dos Prados, 5000-911 Vila Real, Portugal}
\email{andre.oliveira@fc.up.pt\newline\indent agoliv@utad.pt}

\thanks{The authors acknowledge support from U.S. National Science Foundation grants DMS 1107452, 1107263, 1107367 ``RNMS: GEometric structures And Representation varieties" (the GEAR Network). The third author is funded by a National Science Foundation Mathematical Sciences Postdoctoral Fellowship, NSF MSPRF no. 1604263. The fourth author was partially supported by the Spanish MINECO under ICMAT Severo Ochoa project No. SEV-2015-0554, and under grant No. MTM2013-43963-P. The fifth and sixth authors were partially supported by CMUP (UID/MAT/00144/2013) and the project PTDC/MAT-GEO/2823/2014 funded by FCT (Portugal) with national funds. The sixth author was also partially supported by the Post-Doctoral fellowship SFRH/BPD/100996/2014 funded by FCT (Portugal) with national funds.
}

\begin{abstract}

Some connected components of a moduli space are mundane in the sense that they are distinguished only by obvious topological invariants or have no special characteristics. Others are more alluring and unusual either because they are not detected by primary invariants, or because they have special geometric significance, or both.  In this paper we describe new examples of such `exotic' components in moduli spaces of $\SO(p,q)$-Higgs bundles on closed Riemann surfaces or, equivalently, moduli spaces of surface group representations into the Lie group $\SO(p,q)$. Furthermore, we discuss how these exotic components are related to the notion of positive Anosov representations recently developed by Guichard and Wienhard. We also provide a complete count of the connected components of these moduli spaces (except for $\SO(2,q)$, with $q\geq 4$). 

\end{abstract}

\maketitle

\markleft{{Aparicio-Arroyo, Bradlow, Collier, Garc\'{i}a-Prada, Gothen, Oliveira}}

\tableofcontents


\section{Introduction}

For a fixed closed orientable surface $S$ and a semisimple Lie group $\G$, the representation variety $\mathcal{R}(S,\G)$, i.e. the moduli space of group homomorphisms from the fundamental group of $S$ into $\G$, has multiple connected components. Some of the components are mundane in the sense that they are distinguished by obvious topological invariants and have no known special characteristics. Others are more alluring and unusual, either because they are not detected by the primary invariants or because they parametrize objects of special significance, or both. 

Instances of such `exotic' components are well understood in two situations. 
The first is the case where $\G$ is the split real form of a complex semisimple Lie group, in which case the exotic components are known as Hitchin components (see \cite{liegroupsteichmuller}).  The second occurs when $\G$ is the isometry group of a non-compact Hermitian symmetric space, in which case the subspace with so-called maximal Toledo invariant has exotic components (see \cite{HermitianTypeHiggsBGG}). 
In \cite{BurgerIozziWienhardSurvey}, both of these classes of exotic components of representation varieties have been called {\em higher Teichm\"uller components} since they enjoy many of the geometric features of Teichm\"uller space.  

One common feature to all higher Teichm\"uller components is that the representations which they parametrize are all Anosov, a concept introduced by Labourie \cite{AnosovFlowsLabourie}. 
Anosov representations have many interesting dynamical and geometric properties, generalizing convex cocompact representations into rank one Lie groups. In particular, higher Teichm\"uller components consist entirely of discrete and faithful representations \cite{AnosovFlowsLabourie} which are holonomies of geometric structures on certain closed manifolds \cite{guichard_wienhard_2012}. 
In general, the Anosov condition is open in the representation variety and so does not by itself distinguish connected components.
More recently, in \cite{PosRepsGWPROCEEDINGS}, Guichard and Wienhard defined a notion of positivity which refines the Anosov property and is still an open condition. 
They conjecture that such positivity for Anosov representations is also a closed condition, and hence should detect connected components of a representation variety. 
They showed, moreover, that apart from the split real forms and the real forms of Hermitian type, the only other non-exceptional groups which allow positive representations are the groups  locally isomorphic to $\SO(p,q)$ for $1<p<q$, i.e. to the special orthogonal groups with signature $(p,q)$.  This leads directly to the conjecture that $\Rep(S,\SO(p,q))$ should have `exotic' connected components, fitting in the higher Teichm\"uller components framework in the above sense.

In this paper we establish the existence of such exotic components, count them, and show that each exotic component contains positive Anosov representations.  Our methods exploit the Non-Abelian Hodge (NAH) correspondence which defines a homeomorphism between $\mathcal{R}(S,\G)$ and the moduli space of polystable $\G$-Higgs bundles on a Riemann surface, say $X$, homeomorphic to $S$.  Denoting these moduli spaces by $\cM(X,\G)$ or simply $\cM(\G)$ (where we drop the $X$ from the notation unless explicitly needed for clarity or emphasis) our results thus actually address the connected components of $\cM(\SO(p,q))$.  Our main theorem\footnote{This result was announced, without details, in \cite{SOpqAnnouncement}. We now provide the details of the proof.} 
has two parts --- one is an existence result and one is a non-existence result. Namely we prove

\begin{enumerate}
\item the existence of a class of explicitly described exotic components of $\cM(\SO(p,q))$ for $1<p\leq q$, and
\item the non-existence of any other exotic components of $\cM(\SO(p,q))$ for both $p=1$ and $2<p\leq q$.
\end{enumerate}

\noindent Combining these two results, and including the $2^{2g+2}$ `mundane' components, yields a complete count of the connected components for the moduli spaces of $\SO(p,q)$-Higgs bundles $\cM(X,\SO(p,q))$ or, equivalently, the representation varieties $\Rep(S,\SO(p,q))$, for $2<p\leq q$. 

\begin{TheoremIntro}[\ref{Theorem: component count p>2}] 
    Let $X$ be a compact Riemann surface of genus $g\geq 2$ and denote the moduli space of $\SO(p,q)$-Higgs bundles on $X$ by $\cM(\SO(p,q)).$ For $2< p\leq q$, we have 
 \[|\pi_0(\cM(\SO(p,q)))|=2^{2g+2}+\begin{dcases}
        2^{2g+1}+2p(g-1)-1&\text{if\ $q=p+1$}\\
        2^{2g+1}&\text{otherwise}.
    \end{dcases}\]
\end{TheoremIntro}
\begin{remark}
Our methods also show that $\cM(\SO(1,q))$ does not have exotic components for $q>2$, yielding $2^{2g+1}$ connected components. 
For  $q\geq 4$ our techniques fall short of a component count of $\cM(\SO(2,q))$. However, we expect no new exotic components to exist (see Section \ref{Section SO2q} for details). 
\end{remark}

Except for the special cases $p=2$, $q=p$ or $q=p+1$, the group $\SO(p,q)$ is neither split nor of Hermitian type, so the relation between topological invariants and connected components in the representation varieties or related Higgs bundle moduli spaces  cannot be inferred from previously known mechanisms.    

The primary topological invariants are apparent from the structure of the Higgs bundles. 
In the case of $\SO(p,q)$-Higgs bundles on $X$, the objects are described by a triple $(V,W,\eta)$, where $V$ and $W$ are holomorphic orthogonal bundles of rank $p$ and $q$ respectively, such that $\Lambda^pV\cong\Lambda^qW$, and $\eta$ is a holomorphic section of the bundle $\Hom(W,V)\otimes K$, where $K$ is the canonical bundle of $X$.
The topological invariants are then the first and second Stiefel-Whitney classes of $V$ and $W$, subject to the constraint that $sw_1(V)=sw_1(W)$. 
These invariants provide a primary decomposition of the moduli space $\cM(\SO(p,q))$ into (not necessarily connected) components labeled by triples $(a,b,c)\in H^1(S,\Z_2)\times H^2(S,\Z_2)\times H^2(S,\Z_2)$. 
Using the notation $\cM^{a,b,c}(\SO(p,q))$ to denote the union of components labeled by $(a,b,c)$, we can thus write
\begin{equation}\label{Mpq-abc}
\cM(\SO(p,q))=\coprod_{(a,b,c)\in \Z_2^{2g}\times\Z_2\times\Z_2}\cM^{a,b,c}(\SO(p,q)).
\end{equation}
Each space $\cM^{a,b,c}(\SO(p,q))$ has one connected component characterized entirely by the topological invariants $(a,b,c)$. 
This is the connected component which contains the moduli space of polystable orthogonal bundles with these invariants, corresponding to Higgs bundles for the maximal compact subgroup of $\SO(p,q)$. 
Denoted by $\cM^{a,b,c}(\SO(p,q))_{\mathrm{top}}$, these comprise the $2^{2g+2}$ `mundane' components for $2< p\leq q$.  Our existence result identifies additional components disjoint from the  $\cM^{a,b,c}(\SO(p,q))_{\mathrm{top}}$ components. Identifying the topological invariants of each component of Theorem \ref{Theorem: component count p>2} gives the following precise component count.

\begin{CorollaryIntro}[\ref{corollary component count Mabc}]
For $2<p< q-1$ and $(a,b,c)\in H^1(S,\Z_2)\times H^2(S,\Z_2)\times H^2(S,\Z_2)$, we have
\[|\pi_0(\cM^{a,b,c}(\SO(p,q)))|=\begin{dcases}
    2&\text{if $p$ is odd and $b=0$}\\
    2^{2g}+1&\text{if $p$ is even, $a=0$ and $b=0$}\\
    1&\text{otherwise~.}
\end{dcases}\]

 \end{CorollaryIntro}
\begin{remark}
    For $p=1$ and $p=2$, the primary topological invariants are slightly different. 
    For $p=q$ and $p=q-1,$ the connected component count of $\cM^{a,b,c}(\SO(p,q))$ is different (see Corollaries \ref{corollary connected components of SO(p,p+1)} and \ref{corollary connected components of SO(p,p)}). For $p=q$ and $p=q-1,$ all components had been previously detected in \cite{liegroupsteichmuller} and \cite{CollierSOnn+1components} respectively. Nevertheless, the nonexistence of additional components is new.
\end{remark}

One advantage of working on the Higgs bundle side of the NAH correspondence is that Higgs bundles and their moduli spaces possess a rich structure that provides tools which are not readily available in the representation varieties.  Two of these tools, which we exploit, are a real-valued proper function defined by the $\mathrm{L}^2$-norm of the Higgs field, called the Hitchin function, and a natural holomorphic $\C^*$-action. 
These are related since the critical points of the Hitchin function occur at fixed points of the $\C^*$-action. 
When the moduli space is smooth the Hitchin function is a perfect Morse-Bott function. 
While this is not the case in general, the properness of the Hitchin function implies that it attains its minimum on each connected component.  This allows useful information about $\pi_0$ to be extracted from the loci of local minima which, in turn, can be described using information about the corresponding $\C^*$-fixed points.

For many groups $\G$ the Hitchin function has no local minima on $\mathcal{M}(\G)$ other than those defining the mundane components (see for example \cite{Sp(2p2q)modulispaceconnected,Oliveira_GarciaPrada_2016}). In such cases these local minima yield enough information to completely count the components of $\cM(\G)$. 
The group $\SO(p,q)$ is not of this type.  While we are able to classify all the local minima on $\cM(\SO(p,q))$, the singularities in the space render this insufficient for completely determining the number of connected components.  The classification of local minima nevertheless plays a crucial role in the non-existence part of our main result.  In the proof of the  main existence theorem, the $\C^*$-fixed points are helpful but the new exotic components are detected by a more direct approach. 

To show that the components exist, we first describe a model for the supposed components.  We then construct a map from the model to $\cM(\SO(p,q))$ and show that the map has open and closed image. The description of the model invokes a variant of Higgs bundles in which the canonical bundle $K$ is replaced by the $p^\text{th}$ power of $K$.

\begin{TheoremIntro}[\ref{Thm Psi open and closed}]
Let $X$ be a compact Riemann surface with genus $g\geq2$ and canonical bundle $K$. Denote the moduli space of $K^p$-twisted $\SO(1,q-p+1)$-Higgs bundles on $X$ by $\cM_{K^p}(\SO(1,q-p+1))$ and the moduli space of $K$-twisted $\SO(p,q)$-Higgs bundles on $X$ by $\cM(\SO(p,q)).$ For $1\leq p\leq q,$ there is a well defined map
\begin{equation}\label{Eq Psi.1}
\Psi:\xymatrix{\cM_{K^p}(\SO(1,q-p+1))\times\displaystyle\bigoplus\limits_{j=1}^{p-1}H^0(X,K^{2j})\ar[r]&\cM(\SO(p,q))}
\end{equation}
which is an isomorphism onto its image and has an open and closed image. Furthermore, if $p>1,$ then every Higgs bundle in the image of $\Psi$ has a nowhere vanishing Higgs field.  
 \end{TheoremIntro}
In the case $p=2$,  the model described in this theorem coincides exactly with the description of 
the `exotic' maximal components of $\cM(\SO(2,q))$ (see \cite{HermitianTypeHiggsBGG,BGRmaximalToledo}), where the objects parametrized by the components are described by $K^2$-twisted Higgs bundles referred to as Cayley partners. 
In that setting, the emergence of the Cayley partners is a consequence of the fact that $\SO(2,q)$ is a group of Hermitian type; our new results for $\SO(p,q)$ with $p\geq 2$ show that the phenomenon has a more fundamental origin. 
In this regard, we note that our new components generalize both the afore-mentioned Cayley partners in the Hermitian case (i.e. for $p=2$) and also the Hitchin components for the split real forms $\SO(p,p)$ and $\SO(p,p+1)$ (see Section \ref{section Cayley partner} for more details). 

A key technical detail required to show that the map \eqref{Eq Psi.1} has open image, is the fact that the spaces (both the model and its image under the map) are essentially smooth.  
This means that all points are either smooth points or mildly singular, thus allowing the use of Kuranishi's methods to describe open neighborhoods of all points. 
The proof of this key technical detail uses the relation between the tangent spaces for points in $\cM(\SO(p,q))$ and hypercohomology spaces computed from a deformation complex.  
This complex has three terms, with the first term coming from infinitesimal automorphisms and the third term encoding integrability obstructions.  
The crucial lemma establishes the vanishing of the second hypercohomology, i.e.\ of integrability obstructions for infinitesimal deformations.  This is the first place where we exploit the natural $\C^*$-action on the moduli space. 
More precisely, it is the special structure of the fixed points of the action which allows us to prove the vanishing results for the deformation complexes at those points. 
We then use an upper-semicontinuity argument to extend the result to all  points where it is needed. 
To show that the image of the map \eqref{Eq Psi.1} is closed, the properness of the Hitchin fibration is exploited.

The non-existence part of the main theorem follows from a careful analysis of all the $\C^*$-fixed points, most of which is devoted to identifying which fixed points correspond to local minima of the Hitchin function.  We show that these are of two types, namely those where the Higgs field is identically zero, and those which lie in the new exotic components.  Since the former label the known `mundane' components, this proves that we have not missed any components.  



We now discuss a few consequences of our work for the $\SO(p,q)$-representation variety $\cR(S,\SO(p,q))$. Recall that a representation $\rho:\pi_1(S)\to \SO_0(2,1)$ is called {\em Fuchsian} if it is discrete and faithful. Recall also that
there is a unique (up to conjugation) principal embedding 
\begin{equation}
\label{eq princ embedd intro}
\iota:\SO_0(2,1)\to\SO_0(p,p-1)~.
\end{equation}
One consequence of our techniques is a dichotomy for polystable $\SO(p,q)$-Higgs bundles (see Corollary \ref{cor: Higgs bundle dichotomy}).
Translating this statement across the NAH correspondence leads to the following dichotomy for surface group representations into $\SO(p,q)$.
\begin{TheoremIntro} [\ref{THM: rep var dichotomy}]
Let $S$ be a closed surface of genus $g\geq 2$. For $2<p<q-1$, the representation variety $\cR(S,\SO(p,q))$ is a disjoint union of two sets,
\begin{equation}
    \label{Eq intro rep dichotomy}\cR(S,\SO(p,q))~=~\cR^{cpt}(S,\SO(p,q))~\sqcup~ \cR^{ex}(S,S,\SO(p,q))~,
\end{equation}
where
\begin{itemize}
    \item  $[\rho]\in\cR^{cpt}(S,\SO(p,q))$ if and only if $\rho$ can be  deformed to a compact representation,
    \item  $[\rho]\in\cR^{ex}(S,\SO(p,q))$ if and only if $\rho$ can be  deformed to a representation 
    \begin{equation}\label{eq: model rep intro}
    \rho'=\alpha\oplus (\iota\circ\rho_{\mathrm{Fuch}})\otimes \det(\alpha)~,
    \end{equation}
where $\alpha$ is a representation of $\pi_1(S)$ into the compact group $\rO(q-p+1)$, $\rho_{\mathrm{Fuch}}$ is a Fuchsian representation of $\pi_1(S)$ into $\SO_0(2,1)$, and $\iota$ is the principal embedding from \eqref{eq princ embedd intro}.
 \end{itemize} 
\end{TheoremIntro}
\begin{remark}
The above theorem still holds for $p=q>2$, with $\cR^{ex}(S,\SO(p,p))$ being the union of the Hitchin components, but it does not hold when $2<p=q-1$. Namely, there are exactly $2p(g-1)$ exotic components of $\cR(S,\SO(p,p+1))$ for which the result fails. With the exception of the Hitchin component, in \cite{CollierSOnn+1components} it is conjectured that all representations in these components are Zariski dense. 
\end{remark}
The model representations \eqref{eq: model rep intro} connect our work on the Higgs bundle side of the NAH correspondence to the theory of positive Anosov representations. For a parabolic subgroup $\rP<\G,$ the set of P-Anosov representations (see Definition \ref{DEF: Anosov rep}) defines an open set in the representation variety consisting of representations with desirable dynamic and geometric properties \cite{AnosovFlowsLabourie}.
In \cite{PosRepsGWPROCEEDINGS}, Guichard and Wienhard show that for certain pairs $(\G,\rP)$ Anosov representations can satisfy an additional positivity property.  The set of positive Anosov representations is open in $\cR(S,\G)$ and also conjectured by Guichard, Labourie and Wienhard to be closed \cite{PosRepsGLW,PosRepsGWPROCEEDINGS}, and hence to define connected components. Moreover the connected components of this set carry natural labels determined by the topological types of principal $\rP$-bundles (see \cite{TopInvariantsAnosov}).
For the classical groups, the pairs $(\G,\rP)$ which admit a notion of positivity come in three families: one with $\G$ a split real form, one with $\G$ a Hermitian group of tube type, and a third in which $\G$ is locally isomorphic to $\SO(p,q)$.  In the first two families the set of positive Anosov representations corresponds exactly to the connected components of Hitchin representations and maximal representations respectively; thus, for these families, positivity is indeed a closed condition. 
In the case of $\SO(p,q)$ the conjecture is open. However, it follows from the work of Guichard and Wienhard that the model representations \eqref{eq: model rep intro} are positive Anosov representations with respect to a parabolic subgroup $\rP$ which stabilizes a partial isotropic flag in $\R^{p+q}$. Hence as a corollary to Theorem \ref{THM: rep var dichotomy} we have:

\begin{PropositionIntro}[\ref{Prop existence of positive reps}]
Let $\rP\subset\SO(p,q)$ be the stabilizer of the partial flag $V_1\subset V_2\subset\cdots \subset V_{p-1},$
where $V_j\subset\R^{p+q}$ is an isotropic $j$-plane.   If $2<p<q-1$, then each connected component of $\cR^{ex}(S,\SO(p,q))$ from \eqref{Eq intro rep dichotomy} contains a nonempty open set of positive $\rP$-Anosov representations.
\end{PropositionIntro}

Assuming the conjecture of Guichard and Wienhard, it would follow from Proposition  \ref{Prop existence of positive reps} that the exotic components described in this paper correspond exactly to the components in $\cR(S,\SO(p,q))$ containing positive Anosov $\SO(p,q)$-representations. As further evidence for this conclusion it is noteworthy that in the cases where positivity is known to be a closed condition, the representations all satisfy a certain irreducibility condition, namely they do not factor through any proper parabolic subgroup of $\G$.  For the components of $\cR(S,\SO(p,q))$ which do not contain representations into compact groups, we can establish this irreducibility property as a corollary to Theorem \ref{Thm Psi open and closed}. 
In particular, it holds for all representations in the components $\cR^{ex}(S,\SO(p,q))$ from Theorem \ref{THM: rep var dichotomy}.

\begin{PropositionIntro}[\ref{Prop: irreducibility of exotic reps}] Let $\cR^{cpt}(S,\SO(p,q))$ be the union of the connected components of $\cR(S,\SO(p,q))$ containing compact representations.
Let $2<p\leq q$ and $\rho\in \cR(S,\SO(p,q))\setminus\cR^{cpt}(S,\SO(p,q))$. Then $\rho$ does not factor through any proper parabolic subgroup of $\SO(p,q)$.
\end{PropositionIntro}


Though our main results are the first to prove the existence of exotic components outside the realm of higher Teichm\"uller theory for groups of split and Hermitian type, evidence for such components has been building for some time.   
As mentioned above, considerations based on the Guichard-Wienhard positivity property had placed $\cR(S,\SO(p,q))$ among the representation varieties expected to have such components. Even earlier indications had come from a study of the local minima of the Hitchin function on $\cM(\SO(p,q))$. While the absolute minimum, i.e.\ the zero level, of the function is attained on the components $\cM^{a,b,c}(\SO(p,q))_{\mathrm{top}}$, in \cite{MartaThesis} the first author described additional smooth local minima at non-zero values, thus opening up the possibility that further components exist.

The special case $q=p+1$ provided the first confirmation of this possibility. 
Hitchin components were known to exist in $\cM(\SO(p,p+1))$ by virtue of the fact that the group $\SO(p,p+1)$ is the split real form of $\SO(2p+1,\C)$.  The results in \cite{CollierSOnn+1components} show that these are not the only exotic components. With the luxury of hindsight, we now see that the additional components in $\cM(\SO(p,p+1))$ coincide exactly with the exotic components described by our main results for the case $q=p+1$.

We note finally that additional features of the connected components of $\cM(\SO(p,q))$ have been detected by Baraglia and Schaposnik (in \cite{DavidLauraCayleyLanglands}) by examining spectral data on generic fibers of the Hitchin fibration for $\cM(\SO(p+q,\C))$.
Their methods cannot distinguish connected components because of the genericity assumption on the fibers, but, where they apply, their methods provide an intriguing alternative perspective. 

\subsubsection*{Acknowledgments}
The authors are grateful to Olivier Guichard, Beatrice Pozzetti, Carlos Simpson, Richard Wentworth and Anna Wienhard for useful conversations and to the referee for a careful reading and for a number of helpful remarks and corrections.

\section{Higgs bundle background}

In this section we recall the necessary background on $\G$-Higgs bundles on a compact Riemann surface and their deformation theory. Special attention is then placed on the group $\SO(p,q).$ Higgs bundles were introduced by Hitchin in \cite{selfduality} and Simpson in \cite{SimpsonVHS}, and have been studied extensively by many authors. For real groups we will mostly follow \cite{HiggsPairsSTABILITY}. For the rest of the paper, let $X$ be a compact Riemann surface of genus $g\geq 2$ and with canonical bundle $K\to X$.

\subsection{General Definitions} Let $\G$ be a real reductive Lie group with Lie algebra $\fg$ and choose a maximal compact subgroup $\rH\subset\G$ with Lie algebra $\fh\subset\fg$. Fix a Cartan splitting $\fg\cong\fh\oplus \fm$, where $\fm$ is the orthogonal complement of $\fh\subset \fg$ with respect to a nondegenerate $\Ad(\G)$-invariant bilinear form. In particular, $[\fh,\fm]\subset\fm$ and $[\fm,\fm]\subset\fh$, thus such a splitting is preserved by the adjoint action of $\rH$ on $\fg$, giving a linear representation $\rH\to\GL(\fm).$ 
Complexifying everything yields an $\Ad(\rH^\C)$-invariant splitting $\fg^\C\cong\fh^\C\oplus \fm^\C$. 

For any group $\G$, if $P$ is a principal $\G$-bundle and $\alpha:\G\to\GL(V)$ is a linear representation, denote the associated vector bundle $P\times_{\G} V$ by $P[V].$
\begin{definition}\label{Def Ltwisted GHiggsBundle}
  Fix a holomorphic line bundle $L\to X$. An {\em $L$-twisted $\G$-Higgs bundle} is a pair $(\cE,\varphi)$ where $\cE$ is a holomorphic principal $\rH^\C$-bundle  and $\varphi\in H^0(X,\cE[\fm^\C]\otimes L)$ is a holomorphic section of the associated $\fm^\C$-bundle twisted by $L.$ The section $\varphi$ is called the {\em Higgs field}.
\end{definition}

 \begin{remark}
As usual, when the line bundle $L$ is the canonical bundle $K$ of the Riemann surface, we refer to a $K$-twisted Higgs bundle as a \emph{Higgs bundle}. We are mainly interested in the case $L=K,$ however, taking $L=K^p$ will also play an important role. 
 \end{remark}

 \begin{example}\label{Ex compact complex Higgs}
When $\G$ is a compact group, we have $\G^\C=\rH^\C$ and $\fm^\C=0$, so a $\G$-Higgs bundle is just a holomorphic $\G^\C$-bundle on $X.$ When $\G$ is a complex group, we have $\G=\rH^\C$ and $\fm^\C\cong\fg.$ In this case, the Higgs field is just an $L$-twisted section of the adjoint bundle. 
 \end{example}



Rather than dealing with principal bundles, we will use a linear representation $\alpha:\rH^\C\to\GL(V)$ and work with vector bundles and sections of associated bundles. A holomorphic principal $\GL(n,\C)$-bundle is equivalent to a rank $n$ holomorphic vector bundle $E$. 
For $\SL(n,\C)$ we obtain an oriented vector bundle $(E,\omega)$, where $\omega\in H^0(\Lambda^nE)$ is nowhere vanishing. For $\rO(n,\C)$ we get an orthogonal vector bundle $(E,Q)$, where $Q\in H^0(\Sym^2E^*)$ such that $\det(Q)$ is nowhere vanishing. Finally, for $\SO(n,\C)$ we obtain an oriented orthogonal vector bundle $(E,Q,\omega),$ where $\det(Q)(\omega,\omega)=1.$
 
 The standard representations give the following vector bundle definitions, which are equivalent to the principal bundle formulations given by Definition \ref{Def Ltwisted GHiggsBundle}.
\begin{definition}\label{def L twisted GL and SL Higgs}\label{def L twisted On and SO Higgs}

     An \emph{$L$-twisted $\GL(n,\C)$-Higgs bundle} over $X$ is a pair $(E,\Phi)$, where $E\to X$ is a rank $n$ holomorphic vector bundle and $\Phi\in H^0(\End(E)\otimes L)$.
    
     An \emph{$L$-twisted $\SL(n,\C)$-Higgs bundle} is a triple $(E,\omega,\Phi)$, where $(E,\omega)$ is a rank $n$ holomorphic oriented vector bundle and $\Phi\in H^0(\End(E)\otimes L)$ satisfies $\tr(\Phi)=0$.
    
     An \emph{$L$-twisted $\rO(n,\C)$-Higgs bundle} is a triple $(E,Q,\Phi)$, where $(E,Q)$ is a rank $n$ holomorphic orthogonal vector bundle and $\Phi\in H^0(\End(E)\otimes L)$ satisfies $\Phi^TQ+Q\Phi=0.$
     
     An \emph{$L$-twisted $\SO(n,\C)$-Higgs bundle} is a quadruple $(E,Q,\omega,\Phi)$, where $(E,Q,\omega)$ is a rank $n$ holomorphic oriented orthogonal vector bundle and $\Phi\in H^0(\End(E)\otimes L)$ satisfies $\Phi^TQ+Q\Phi=0.$
\end{definition}

Two $\GL(n,\C)$-Higgs bundles $(E_1,\Phi_1)$ and $(E_2,\Phi_2)$ are isomorphic if there exists a holomorphic bundle isomorphism $f:E_1\to E_2$ so that $f^*\Phi_2=\Phi_1.$ For $\SL(n,\C),$ $\rO(n,\C)$ and $\SO(n,\C)$-Higgs bundles we require that the isomorphism $f$ pulls back the additional structure. 

The group $\rO(p,q)$ is the group of linear automorphisms of $\R^{p+q}$ which preserve a nondegenerate symmetric quadratic form of signature $(p,q)$.  Note that $\rO(p,q)$ and $\rO(q,p)$ are isomorphic groups, so we can assume that $p\leq q$ without loosing any generality. We are mainly interested in the subgroup $\G=\SO(p,q)$ of $\rO(p,q)$ which also preserves an orientation of $\R^{p+q}$. This group has two connected components provided $p$ and $q$ are both positive, and the connected component of the identity is denoted by $\SO_0(p,q)$. We shall assume throughout the paper that $0<p\leq q$.

If $Q_p$ and $Q_q$ are positive definite symmetric $p\times p$ and $q\times q$ matrices, then the Lie algebra $\fso(p,q)$ is defined by the matrices
\[\fso(p,q)\cong\left\{\smtrx{A&B\\C&D}\ \left|\ \smtrx{A&B\\C&D}^T\right.\smtrx{Q_p&\\&-Q_q}+\smtrx{Q_p&\\&-Q_q}\smtrx{A&B\\C&D}=0\right\},\]
where $A$ is a $p\times p$ matrix, $B$ is a $p\times q$ matrix, $C$ is a $q\times p$ matrix and $D$ is a $q\times q$ matrix.
Thus, 
\begin{equation}\label{EQ so(p,q) Lie algebra}
  \xymatrix@C=.5em{A^TQ_p+Q_pA=0,& D^TQ_q+Q_qD=0&\text{and}&C=-Q_q^{-1}B^TQ_p~.}
\end{equation}

The maximal compact subgroup of $\rO(p,q)$ is $\rO(p)\times \rO(q)$ and the maximal compact subgroup of $\SO(p,q)$ is $\rS(\rO(p)\times\rO(q))$. Using \eqref{EQ so(p,q) Lie algebra}, the complexified Cartan decomposition of $\fso(p,q)$ is 
\[\fso(p+q,\C)\cong(\fso(p,\C)\oplus\fso(q,\C))\oplus \Hom(W,V),\]
 where $V$ and $W$ are the standard representations of $\rO(p,\C)$ and $\rO(q,\C).$ 
 Using these representations, we have the following vector bundle definition of an $\SO(p,q)$-Higgs bundle. 

\begin{definition}\label{DEF SO(p,q) Higgs bundles}
An \emph{$L$-twisted $\rO(p,q)$-Higgs bundle} is a tuple $(V,Q_V,W,Q_W,\eta)$, where $(V,Q_V),$ $(W,Q_W)$ are rank $p$, $q$ holomorphic orthogonal vector bundles respectively and $\eta\in H^0(\Hom(W,V)\otimes L).$

An \emph{$L$-twisted $\SO(p,q)$-Higgs bundle} is a tuple $(V,Q_V,W,Q_W,\omega,\eta),$ where $(V,Q_V,W,Q_W,\eta)$ is an $L$-twisted $\rO(p,q)$-Higgs bundle and $\left(V\oplus W,\smtrx{Q_V\\&-Q_W},\omega\right)$ is an oriented orthogonal vector bundle.
\end{definition}

\begin{remark}
We will usually interpret the orthogonal structures and the orientation as isomorphisms:
\[Q_V:V\xrightarrow{\ \ \cong\ \ } V^*, \ \ \ \  \ \ \   Q_W:W\xrightarrow{\ \ \cong\ \ } W^*\ \ \ \ \  \ \ \ \text{and}\ \ \ \ \ \ \ \ \ \omega:\Lambda^p V\xrightarrow{\ \ \cong\ \ }\Lambda^qW.\]
Moreover, we will usually suppress the orthogonal structures and orientation from the notation.
\end{remark}
Two $\SO(p,q)$-Higgs bundles $(V_1,Q_{V_1},W_1,Q_{W_1},\omega_1,\eta_1)$ and $(V_2,Q_{V_2},W_2,Q_{W_2},\omega_2,\eta_2)$ are isomorphic if there exists bundle isomorphisms $g_V:V_1\to V_2$ and $g_W:W_1\to W_2$ so that 
\[\xymatrix@=1em{Q_{V_1}=g_V^TQ_{V_2}g_V,&Q_{W_1}=g_W^TQ_{W_2}g_W,&\omega_1=\det(g_V)\det(g_W)\omega_2&\text{and}&\eta_1=g_{V}^{-1}\eta_2g_{W}}.\]

Given an $L$-twisted $\SO(p,q)$-Higgs bundle $(V,Q_V,W,Q_W,\eta)$, let 
\[\eta^*=(Q_W^{-1}\otimes\Id_L)(\eta^T\otimes\Id_L)Q_V,\] where
$\eta^T:V^*\otimes L^{-1}\to W^*$ is the dual map. The $L$-twisted $\SO(p+q,\C)$-Higgs bundle associated to $(V,Q_V,W,Q_W,\eta)$ is given by 
\begin{equation}\label{Eq SO(p+q,C) Higgs bundle}
  (E,Q,\omega,\Phi)=\Big(V\oplus W,\smtrx{Q_V&0\\0&-Q_W},\omega,\smtrx{0&\eta\\\eta^*&0}\Big)~.
\end{equation}

In subsequent sections, we will also need the notions of $\U(p,q)$-Higgs bundles and $\GL(n,\R)$-Higgs bundles. The complexified Cartan decompositions for these groups are given by 
\[\fu(p,q)^\C\cong(\fgl(p,\C)\oplus\fgl(q,\C)) \oplus (\Hom(E,F)\oplus \Hom(F,E)),\]
\[\fgl(n,\R)^\C\cong \fo(n,\C)\oplus \sym(\C^n),\]
where $E$ and $F$ are respectively the standard representations of $\GL(p,\C)$ and $\GL(q,\C)$ and $\sym(\C^n)$ denotes the set of symmetric endomorphisms of $\C^n.$ 
\begin{definition}\label{Def GLnR and Upq Higgs def}
  An \emph{$L$-twisted $\U(p,q)$-Higgs bundle} is a tuple $(E,F,\beta,\gamma)$, where $E$, $F$ are holomorphic vector bundles of rank $p$, $q$ respectively, $\beta\in H^0(\Hom(F,E)\otimes L)$ and $\gamma\in H^0(\Hom(E,F)\otimes L)$.
 
  An \emph{$L$-twisted $\GL(n,\R)$-Higgs bundle} is a tuple $(E,Q,\Phi)$ where $(E,Q)$ is a holomorphic rank $n$ orthogonal vector bundle and $\Phi\in H^0(\End(E)\otimes L)$ such that $\Phi^TQ=Q\Phi$.
 \end{definition} 

\subsection{The Higgs bundle moduli space and deformation theory}

To form a moduli space of $\G$-Higgs bundles we need a notion of stability for these objects. 
In general, these stability notions involve the interaction of the Higgs field with 
certain parabolic reductions of structure group.
For the above groups, stability can be simplified and expressed in vector bundle terms (see \cite{HiggsPairsSTABILITY}). 

\begin{proposition}\label{Def Stability SL(n,C)}
An $L$-twisted $\SL(n,\C)$-Higgs bundle $(E,\Phi)$ is 
\begin{itemize}
  \item semistable if for every holomorphic subbundle $F\subset E$ with $\Phi(F)\subset F\otimes L$ we have $\deg(F)\leq0,$
  \item stable if for every proper holomorphic subbundle $F\subset E$ with $\Phi(F)\subset F\otimes L$ we have $\deg(F)<0,$
  \item polystable if it is semistable and for every degree zero subbundle $F\subset E$ with $\Phi(F)\subset F\otimes L$, there is a subbundle $F'$ with $\Phi(F')\subset F'\otimes L$ so that $E\cong F\oplus F'.$  That is, 
  \[(E,\Phi)=\Big(F\oplus F',\smtrx{\Phi_F&0\\0&\Phi_{F'}}\Big)~.\]
  \end{itemize} 
\end{proposition}

\begin{remark}\label{Remark stability O(n,C)}
  For the notions of stability, semistability and polystability for an $L$-twisted $\rO(n,\C)$-Higgs bundles $(E,Q,\Phi),$ one only needs to consider {\em isotropic} subbundles $F\subset E$ with $\Phi(F)\subset F\otimes K$. Here a subbundle $F\subset E$ is isotropic if $F\subset F^\perp$, where $F^\perp$ is the perpendicular subbundle defined by $Q.$ For a polystable $L$-twisted $\rO(n,\C)$-Higgs bundle, if $F\subset E$ is a degree zero isotropic subbundle with $\Phi(F)\subset F\otimes L,$ then $E\cong F\oplus F'$ where $F'$ is a degree zero coisotropic subbundle satisfying $\Phi(F')\subset F'\otimes L.$ We note also that the polystability of $(E,Q)$ as an orthogonal vector bundle is equivalent to the polystability of $E$ as a vector bundle \cite{Ramanan1981}.
 \end{remark}

 For real groups, the notions of semistability, stability and
 polystability are a bit more involved.  However, for the purpose of defining the moduli spaces we are interested in, it is sufficient to use the following result of
 \cite{HiggsPairsSTABILITY}.
 

\begin{proposition}\label{Prop G polystable iff SLn Polystable}
  Let $\G$ be a real form of a simple subgroup of $\SL(n,\C).$ An
  $L$-twisted $\G$-Higgs bundle $(\cE,\varphi)$ is polystable if and
  only if the induced $\SL(n,\C)$-Higgs bundle is polystable in the
  sense of Proposition \ref{Def Stability SL(n,C)}. The analogous statement for semistability also holds.
\end{proposition}

\begin{definition}
The \emph{moduli space of $L$-twisted $\G$-Higgs bundles on $X$} is the set $\cM_L(\G)$ of isomorphism classes of polystable $L$-twisted $\G$-Higgs bundles. The subset where $\mathcal{E}$ has fixed topological type $a$ is denoted by $\cM^a_L(\G)\subset\cM_L(\G)$. In the case $L=K$, we shall denote the corresponding moduli spaces just by $\cM^ a(\G)\subset \cM(\G)$.
\end{definition}
\begin{remark}\label{rmk:dimension}
The above defines the moduli space as a set. It can be given the structure of a complex analytic variety using standard methods, as we briefly outline in Section~\ref{sec:complex-analytic} below.
Alternatively, the moduli space can be constructed algebraically as the set of $\cS$-equivalence classes of semistable $\G$-Higgs bundles
as a particular case of a construction of Schmitt \cite{schmitt_2005} using geometric invariant theory. When the maximal compact subgroup of $\G$ is semisimple and $\deg(L)\geq 2g-2$, the expected dimension of $\cM_L(\G)$ is $\dim(\fh)(g-1)+\dim(\fm)(\deg(L)+1-g)$.
\end{remark}

The automorphism group $\Aut(\cE,\varphi)$ of a $\G$-Higgs bundle $(\cE,\varphi)$ consists of holomorphic bundle automorphisms which act trivially on the Higgs field: 
\begin{equation}
  \label{eq: automorphism group def}\Aut(\cE,\varphi)=\{f:\cE\to\cE|\ \Ad_f\varphi=\varphi\}.
\end{equation}
The center $\cZ(\G^\C)$ of $\G^\C$ is the intersection of the center of $\rH^\C$ and the kernel of the representation $\Ad:\rH^\C\to\GL(\fm^\C)$. 
Thus, we always have $\cZ(\G^\C)\subset\Aut(\cE,\varphi)$. 

\begin{remark}\label{remark stable open in polystable}
  If $\G^\C$ is semisimple, then a $\G$-Higgs bundle is {\em stable}
  if and only if it is polystable with finite automorphism group. 
\end{remark}

The deformation theory of a $\G$-Higgs bundle $(\cE,\varphi)$ is
governed by the complex of sheaves 
\begin{equation}\label{eq complex of sheaves}
  C^\bullet(\cE,\varphi):
  \xymatrix{\cE[\fh^\C]\ar[r]^-{\ad_\varphi}&\cE[\fm^\C]\otimes L;}
\end{equation}
indeed $\HH^0(C^\bullet(\cE,\varphi))$ can be identified with the Lie
algebra of $\Aut(\cE,\varphi)$ and
$\HH^1(C^\bullet(\cE,\varphi))$ is the infinitesimal deformation space
(see \cite{biswas-ramanan}). There is a long exact sequence in hypercohomology:
\begin{equation}\label{EQ deformation complex DEF}
\xymatrix@R=1em@C=1.3em{0\ar[r]&\HH^0(C^\bullet(\cE,\varphi))\ar[r]&H^0(\cE[\fh^\C])\ar[r]^{ \ad_\varphi\ \ \ \ }&H^0(\cE[\fm^\C]\otimes L)\ar[r]&\HH^1(C^\bullet(\cE,\varphi))\ar@{->}`r/3pt [d] `/10pt[l] `^d[llll] `^r/3pt[d][dlll]\\&H^1(\cE[\fh^\C])\ar[r]^{\ad_\varphi \ \ }&H^1(\cE[\fm^\C]\otimes L)\ar[r]&\HH^2(C^\bullet(\cE,\varphi))\ar[r]&0~.}
\end{equation}
\begin{remark}
    \label{remark duality of H0 and H2 for complex groups} When the group $\G$ is complex, Serre duality implies that the second hypercohomology group in this deformation complex is isomorphic to the dual of the zeroth hypercohomology group \cite[Proposition 3.17]{HiggsPairsSTABILITY}. In particular, this implies that for complex semisimple groups $\HH^2(C^\bullet(\cE,\varphi))$ vanishes if and only if the Higgs bundle $(\cE,\varphi)$ is stable.
\end{remark}

\subsection{The complex analytic point of view on the moduli space}
\label{sec:complex-analytic}

Fix a $C^\infty$ principal $\rH^\C$-bundle $P$.  The set of holomorphic structures on $P$ is an affine space modeled on $\Omega^{0,1}(X,P[\fh^\C])$: indeed, since $\dim(X)=1$, any partial connection on $P$ is integrable and thus defines a holomorphic structure on $P$. A partial connection induces a Dolbeault operator $\dbar_P$ on any vector bundle associated to $P$ and we shall by a slight abuse of notation denote the partial connection itself by the same symbol. We can now introduce the space of $L$-twisted Higgs bundle structures on $P$ by
\begin{equation}
  \label{eq configuration space}
  \cH_L(\G,P)=\{(\dbar_P,\varphi)\ |\ \dbar_P\varphi=0\},
\end{equation}
where $\varphi\in\Omega^{0}(X,P[\fm^\C]\otimes L)$ is the Higgs field.

The complex gauge group $\cG_{\rH^\C}$ of $C^\infty$ bundle automorphisms of $P$ acts on the space $\cH_L(\G,P)$. Moreover, this action preserves the subspace $\cH_L(\G,P)^{ps}\subset\cH_L(\G,P)$ of polystable $L$-twisted Higgs bundles. If we denote the topological type of $P$ by $a$, we thus have an identification
\begin{displaymath}
  \cM^a_L(\G) = \cH_L(P,\G)^{ps}/\cG_{\rH^\C}~.
\end{displaymath}
In order to give the moduli space a topology, suitable Sobolev
completions must be used in standard fashion; see
\cite{atiyah-bott:1982}, and also \cite[Sec.~8]{hausel-thaddeus:2004}
where the straightforward adaptation to Higgs bundles is discussed in
the case $\G=\GL(n,\C)$. Then the orbits of the $\cG_{\rH^\C}$-action
on $\cH_L(\G,P)^{ps}$ are closed in the space of semistable $\G$-Higgs bundles and the moduli space $\mathcal{M}_L(\G)$
becomes a Hausdorff topological space.

\begin{remark} If $\cH^s_L(P,\G)\subset \cH^{ps}_L(P,\G)$ denotes the subset of stable Higgs bundle structures, then $\cH^s_L(P,\G)$ is open in $\cH^{ps}_L(P,\G)$. The stable objects thus define open subsets of  $  \cM^a_L(\G) $.
\end{remark}

The moduli space can be given the structure of a complex analytic
variety using the Kuranishi model in a way analogous to the
case of vector bundles on algebraic surfaces; see, e.g.,
\cite[Sec~6.4.1]{donaldson-kronheimer:1990} or
\cite[Chap.~4]{friedman-morgan}. We briefly recall this for a point
represented by a Higgs bundle $(\cE,\varphi)$ with
vanishing $\HH^2(C^\bullet(\cE,\varphi))$, this being the only case we
shall need. For any such $(\cE,\varphi)$, there is an open
neighborhood $U$ of zero in $\HH^1(C^\bullet(\cE,\varphi))$ and a
local versal family
of $\G$-Higgs bundles parametrized
by $U$ which restricts to $(\cE,\varphi)$ over $\{0\}\times X$.
Moreover, if $(\cE,\varphi)$ is polystable then (semistability being
an open condition) $U$ can be taken to consist only of semistable
$\G$-Higgs bundles, and the map taking a semistable $\G$-Higgs bundle
in $U$ to the polystable representative of its $\cS$-equivalence class
projects $U$ onto an open neighborhood of $(\cE,\varphi)$ in the
moduli space.


Though we shall not need this, we note that the neighborhood $U$ can
be taken to be $\Aut(\cE,\varphi)$-invariant and then an open
neighborhood of $(\cE,\varphi)$ in the moduli space is modeled on
the GIT quotient
$U\sslash \Aut(\cE,\varphi).$
When the automorphism group $\Aut(\cE,\varphi)$ is finite, the GIT quotient  simplifies to a regular quotient, and the isomorphism class $(\cE,\varphi)$ defines (at worst) an orbifold point of $\cM_L(\G)$.

\subsection{Stability and deformation complex for $\G=\SO(p,q)$}
We shall need the precise notion of stability for $\SO(p,q)$-Higgs bundles. The derivation of the following simplification of the stability notion for $\SO(p,q)$-Higgs bundles is very similar to many cases treated in the literature. For example, see \cite{Sp(2p2q)modulispaceconnected} for the case $\G=\Sp(2p,2q).$ 


\begin{proposition}\label{Def SO(p,q) stability and polystability}
Let $(V,Q_V,W,Q_W,\omega,\eta)$ be an $L$-twisted $\SO(p,q)$-Higgs bundle and let $\eta^*=Q_W^{-1}\eta^TQ_V$. Then it is
\begin{itemize}
  \item semistable if and only if for any pair of isotropic subbundles $V_1\subset V$ and $W_1\subset W$ such that $\eta(W_1)\subset V_1\otimes L$ and $\eta^*(V_1)\subset W_1\otimes L$, we have $\deg(V_1)+\deg(W_1)\leq 0,$
  \item stable if and only if for any pair of isotropic subbundles $V_1\subset V$ and $W_1\subset W$, at least one of which is a proper\footnote{We note that for a rank two orthogonal bundle of the form ($L\oplus L^*,\smtrx{0&\Id\\ \Id&0})$, the isotropic subbundle $L$ is not considered to be proper. This is because $\SO(2,\C)\cong\C^*$, so $L$ does not define a proper reduction of structure group.} subbundle, and such that $\eta(W_1)\subset V_1\otimes L$ and $\eta^*(V_1)\subset W_1\otimes L$, we have $\deg(V_1)+\deg(W_1)< 0,$
  \item polystable if and only if it is semistable and whenever $V_1\subset V$ and $W_1\subset W$ are isotropic subbundles with $\eta(W_1)\subset V_1\otimes L$, $\eta^*(V_1)\subset W_1\otimes L$ and $\deg(V_1)+\deg(W_1)=0,$ there are coisotropic bundles $V_2\subset V$ and $W_2\subset W$, complementary to $V_1$ and $W_1$ respectively, so that $\eta(W_2)\subset V_2\otimes L$ and $\eta^*(V_2)\subset W_2\otimes L$. That is, 
  \[(V,W,\eta)=\Big(V_1\oplus V_2, W_1\oplus W_2, \smtrx{\eta_1&0\\0&\eta_2}\Big)~.\]
\end{itemize}
\end{proposition}

We now give a recursive classification of strictly polystable $\SO(p,q)$-Higgs bundles, which will be important in the following sections of the paper.

Given a $\U(p,q)$-Higgs bundle $(E,F,\beta,\gamma)$ with $\deg(E\oplus F)=0$, consider the associated $\SO(2p,2q)$-Higgs bundle 
\[(V,Q_V,W,Q_W,\eta)=\Big(E\oplus E^*,\smtrx{0&\Id\\ \Id&0},F\oplus F^*,\smtrx{0&\Id\\ \Id&0},\smtrx{\beta&0\\0&\gamma^T}\Big)~.\]
If $(E,F,\beta,\gamma)$ is a polystable $\U(p,q)$-Higgs bundle, then this $\SO(2p,2q)$-Higgs bundle is strictly polystable.  Indeed, $E$, $E^*,$ $F$ and $F^*$ are all isotropic subbundles with $\deg(E)+\deg(F)=0$ and 
\[\xymatrix@=1em{\eta(F)\subset E\otimes K,&\eta(F^*)\subset E^*\otimes K,&\eta^*(E)\subset F\otimes K,&\text{and}&\eta^*(E^*)\subset F^*\otimes K.}\]

\begin{proposition}\label{Prop: strictly polystable SOpq}
  An $\SO(p,q)$-Higgs bundle $(V,Q_V,W,Q_W,\eta)$ is polystable if and only if it is isomorphic to
\begin{equation}\label{eq:SOpq polystable}
 \bigg(E\oplus E^*\oplus  V_0, \smtrx{0&\Id&0\\ \Id&0&0\\0&0&Q_{V_0}}, F\oplus F^*\oplus W_0, \smtrx{0&\Id&0\\ \Id&0&0\\0&0&Q_{W_0}},\smtrx{\beta&0&0\\0&\gamma^T&0\\0&0&\eta_0}\bigg),
\end{equation}  
where $(E,F,\beta,\gamma)$ is a polystable $\U(p_1,q_1)$-Higgs bundle with $\deg(E)+\deg(F)=0$, and $(V_0,Q_{V_0},W_0,Q_{W_0},\eta_0)$ is a stable $\SO(p-2p_1,q-2q_1)$-Higgs bundle.
 \end{proposition} 
\begin{proof}
If $(V,Q_V,W,Q_V,\eta)$ is stable, take $p_1=q_1=0$. Suppose that $(V,Q_V,W,Q_W,\eta)$ is strictly polystable and that $E\subset V$ and $F\subset W$ are isotropic subbundles of rank $p_1$ and $q_1$ respectively, such that $\deg(E)+\deg(F)=0$ and 
 \[\xymatrix{\eta(F)\subset E\otimes K&\text{and}&\eta^*(E)\subset F\otimes K}.\]
Since $(V,W,\eta)$ is polystable, the bundles $V$ and $W$ split as $V=E\oplus V'$ and $W=F\oplus W'$, 
where $V'$ and $W'$ are both coisotropic subbundles with the property
\[\xymatrix{\eta(W')\subset V'\otimes K&\text{and}&\eta^*(V')\subset W'\otimes K}.\]
Since the bundles $E$ and $F$ are isotropic, the bundles $V'$ and $W'$ are extensions of the form:
\[\xymatrix@=1em{0\ar[r]&E^\perp/E\ar[r]&V'\ar[r]&E^*\ar[r]&0&\text{and}&0\ar[r]&F^\perp/F\ar[r]&W'\ar[r]&F^*\ar[r]&0}.\]
We claim that the above extension classes vanish. For the bundle $V$ we have a holomorphic splitting $E\oplus V'$ and a smooth splitting $E\oplus E^\perp/E\oplus E^*$. In this smooth splitting, the orthogonal structure $Q_V$ and the $\dbar$-operator on $V$ are isomorphic to 
\[\xymatrix{Q_V\cong\smtrx{0&0&\Id\\0&Q_{E^\perp/E}&0\\ \Id&0&0}&\text{and}&\dbar_V\cong \smtrx{\dbar_{E}&0&0\\0&\dbar_{E^\perp/E}&\alpha\\0&0&\dbar_{E^*}}},\]  
where $\alpha\in\Omega^{0,1}(\Hom(E^*,E^\perp/E)).$ However, since the orthogonal structure $Q_V$ is holomorphic, we have $\alpha=0.$ By applying the same argument to the bundle $W$, we have the following holomorphic splitting
\[(W,Q_W)\cong \bigg(F\oplus F^\perp/F\oplus F^*,\smtrx{0&0&\Id\\0&Q_{F^\perp/F}&0\\\Id&0&0}\bigg).\]
The conditions $\eta(F)\subset E\otimes K$, $\eta^*(E)\subset F\otimes K$ and $\eta(W')\subset V'\otimes K$ imply that $\eta$ is given by 
\[\eta=\smtrx{\beta&0&0\\0&\eta_0&0\\0&0&\gamma^T}:F\oplus F^\perp/F\oplus F^*\longrightarrow E\oplus E^\perp/E\oplus E^*.\]
The tuple $(E,F,\beta,\gamma)$ defines a polystable $\U(p_1,q_1)$-Higgs bundle and 
\[(V_0,Q_{V_0},W_0,Q_{W_0},\eta_0)=(E^\perp/E,Q_{E^\perp/E},F^\perp/F,Q_{F^\perp/F},\eta_0)\]
defines a polystable $\SO(p-2p_1,q-2q_1)$-Higgs bundle. By iterating this process if necessary, we may assume $(V_0,W_0,\eta_0)$ is stable. 

The converse statement is clear.
 \end{proof}

For an $L$-twisted $\SO(p,q)$-Higgs bundle $(V,Q_V,W,Q_W,\omega,\eta)$, write
\[\xymatrix@=.8em{\fso(V)=\{\alpha\in \End(V)\ |\ \alpha^TQ_V+Q_V\alpha=0\}&\text{and}&\fso(W)=\{\beta\in \End(W)\ |\ \beta^TQ_W+Q_W\beta=0\}.}\] 
Then the Lie algebra bundles $\cE[\fh^\C]$ and $\cE[\fm^\C]\otimes L$ are given by
\[\xymatrix{\cE[\fso(p,\C)\oplus \fso(q,\C)]\cong \fso(V)\oplus \fso(W)&\text{and}&\cE[\fm^\C]\otimes L\cong\Hom(W,V)\otimes L.}\]


The deformation complex \eqref{eq complex of sheaves} becomes 
\begin{equation}\label{EQ ad eta definition}
  C^\bullet(V,W,\eta): \xymatrix@R=0em{\fso(V)\oplus\fso(W)\ar[r]^-{\ad_\eta}&\Hom(W,V)\otimes L~~,\\(\alpha,\beta)\ar@{|->}[r]&\eta\otimes \beta-(\alpha\otimes \Id_L)\otimes \eta}
\end{equation}
and the long exact sequence \eqref{EQ deformation complex DEF} is given by 
\begin{equation}\label{EQ SO(p,q) deformation complex DEF}
\xymatrix@R=1em@C=.8em{0\ar[r]&\HH^0(C^\bullet(V,W,\eta))\ar[r]&H^0(\fso(V)\oplus\fso(W))\ar[r]^-{\ad_\eta}&H^0(\Hom(W,V)\otimes L)\ar[r]&\HH^1(C^\bullet(V,W,\eta))\ar@{->}`r/3pt [d] `/10pt[l] `^d[llll] `^r/3pt[d][dlll]\\&H^1(\fso(V)\oplus\fso(W))\ar[r]^-{\ad_\eta}&H^1(\Hom(W,V)\otimes L)\ar[r]&\HH^2(C^\bullet(V,W,\eta))\ar[r]&0~.}
\end{equation}
We will use the above complex and long exact sequence extensively throughout the paper.

Finally, we make explicit the gauge theoretic perspective for $\SO(p,q)$-Higgs bundles. 
Fix $C^\infty$ rank $p$ and $q$ orthogonal vector bundles $(\underline V,Q_V)$ and $(\underline W,Q_W)$ respectively, and a smooth nowhere vanishing section
$\omega$ of $\Lambda^{p+q}\underline (V\oplus W)$ so that $\det(Q_V\oplus -Q_W)(\omega,\omega)=1$. An
$\SO(p,q)$-Higgs bundle structure on $(\underline V,Q_V,\underline W,Q_W,\omega)$
consists of a triple $(\dbar_V,\dbar_W,\eta)$, where $\eta\in\Omega^{0,1}(\Hom(W,V)\otimes K)$ and $\dbar_V$ and $\dbar_W$ are Dolbeault operators on $\underline V$ and $\underline W$ with respect to which $Q_V,$ $Q_W,$ $\omega$ and $\eta$ are each holomorphic.
An isomorphism between two such Higgs bundle structures $(\dbar_V,\dbar_W,\eta)$ and
$(\dbar_V',\dbar_W',\eta')$ is given by an element of the $\rS(\rO(p,\C)\times\rO(q,\C))$-gauge
group. That is, a pair of $C^\infty$ bundle automorphism $f_V:\underline {V}\to \underline{ V} $ and $f_W:\underline W\to\underline W$ so that
\[\xymatrix{f_V^TQ_Vf_V=Q_V,&f_W^TQ_Wf_W=Q_W,& \text{and}& \det(f_V)\otimes \det(f_W)=1,}\] with the property that
$(f_V^*\dbar_V',f_W^*\dbar_W',f_W^{-1}\eta'f_V)=(\dbar_V,\dbar_W,\eta).$
\subsection{The Hitchin fibration and Hitchin component}
Let $\G^\C$ be a complex semisimple Lie group of rank $\ell$ and let $p_1,\ldots,p_{\ell}$ be a basis of $\G^\C$-invariant homogeneous polynomials on $\fg^\C$ with $\deg(p_j)=m_j+1.$ 
Given an $L$-twisted $\G^\C$-Higgs bundle $(\cE,\varphi),$ the tensor $p_j(\varphi)$ is a holomorphic section of $L^{m_j+1}.$ The map 
$(\cE,\varphi)\mapsto(p_1(\varphi),\ldots,p_{\ell}(\varphi))$ descends to a map 
\begin{equation}\label{EQ Hitchin Fibration}
  h:\xymatrix{\cM_L(\G^\C)\ar[r]&\displaystyle\bigoplus\limits_{j=1}^{\ell}H^0(L^{m_j+1})}
\end{equation}
known as the Hitchin fibration. In \cite{selfduality}, Hitchin showed that $h$ is a {\em proper} map for $L=K$, and for general $L$ properness was shown by Nitsure in \cite{NitsureHiggs}. 

 Another important aspect of the Hitchin fibration for this paper is the Hitchin section. 
 \begin{theorem}\label{THM Hitchin component}
  (Hitchin \cite{liegroupsteichmuller}) Let $\G$ be the split real form of a complex semisimple Lie group $\G^\C$ of rank $\ell$. There is a section of the fibration \eqref{EQ Hitchin Fibration} with $L=K$ such that the image consists of $\G$-Higgs bundles and
  defines a connected component of $\cM(\G).$
 \end{theorem}
 \begin{remark}\label{Remark auto of Hitchin component}
  For a split real group $\G$, a connected component of $\cM(\G)$ described by Theorem \ref{THM Hitchin component} is called a \emph{Hitchin component}.  Since the Hitchin components are smooth, the automorphism group of a Higgs bundle in such a component is as small as possible. For $\rO(p,p-1)$, it is given by $\pm(\Id_V,\Id_W).$
 \end{remark}

We now describe an explicit construction of a section of \eqref{EQ Hitchin Fibration} for the group $\G^\C=\rO(2p-1,\C)$.  This construction will be used in Section \ref{sec:exoticcomp}. We will construct one such section $s_H^I$ for each choice of a holomorphic line bundle $I$ with $I^2\cong \cO$. In this case, the rank is $p-1$, the integers $m_j+1$ equal to $2j$ and the split real form is isomorphic to $\rO(p,p-1)$.  Therefore the Hitchin section is given by 
\[  
  s_H^I:\bigoplus\limits_{j=1}^{p-1}H^0(K^{2j})\to \cM(\rO(2p-1,\C)).
\]
For each $n$, consider the holomorphic orthogonal bundle
\begin{equation}\label{eq:KnQn}
(\cK_n,Q_n)=\bigg(K^{n}\oplus K^{n-2}\oplus\cdots\oplus K^{2-n}\oplus K^{-n},\smtrx{&&1\\&\iddots &\\1}\bigg).
\end{equation}
For $(q_2,\ldots,q_{2p-2})\in \bigoplus\limits_{j=1}^{p-1}H^0(K^{2j})$, the $\rO(p,p-1)$-Higgs bundle $(V,Q_V,W,Q_W,\eta)$ in the image of a Hitchin section $s^I_H$ is given by 
\begin{equation}
    \label{EQ Hitchin section O(p,C)}s^I_H(q_2,\ldots,q_{2p-2})=(I\otimes\cK_{p-1},Q_{p-1},I\otimes \cK_{p-2},Q_{p-2},\eta(q_2,\ldots,q_{2p-2})),
\end{equation}
where $\eta(q_2,\ldots,q_{2p-2})$ depends on a choice of the basis of invariant polynomials. Notice that, in particular, the holomorphic structures on $V=I\otimes\cK_{p-1}$ and $W=I\otimes\cK_{p-2}$ are fixed. One choice for $\eta(q_2,\ldots,q_{2p-2})$ is given by 
\begin{equation}
  \label{eq choice of Hitchin section O(p-1,p)}
  \eta(q_2,\ldots,q_{2p-2})=\smtrx{q_2&q_4&\cdots& q_{2p-2}\\1&q_2&\cdots&q_{2p-4}\\&\ddots&\ddots&\\&&1&q_2\\&&&1}:\xymatrix{I\otimes\cK_{p-2}\ar[r]&I\otimes \cK_{p-1}\otimes K}.
\end{equation}
For example, when $p=3$ we have 
\[(V,Q_V,W,Q_W,\eta(q_2,q_4))=\Big(IK^2\oplus I\oplus IK^{-2},\smtrx{&&1\\&1&\\1}, IK\oplus IK^{-1},\smtrx{&1\\1},\smtrx{q_2&q_4\\1&q_2\\0&1}\Big)~.\]
If $(E,Q,\Phi)$ is the associated $\rO(5,\C)$-Higgs bundle from \eqref{Eq SO(p+q,C) Higgs bundle}, 
then $\tr(\Phi^2)=8q_2$ and $\tr(\Phi^4)=20q_2^2+8q_4.$ So the above description describes the Hitchin section for the basis $p_1(\Phi)=\frac{1}{8}\tr(\Phi^2)$ and $p_2=\frac{1}{8}\tr(\Phi^4)-\frac{20}{64}(\tr(\Phi^2))^2$.

\subsection{Topological invariants}
\label{sec:top-invariants}
 
Since $\rH^\C$ and $\G$ are both homotopy equivalent to $\rH,$ the set of equivalence classes of topological $\rH^\C$-bundles on $X$ is the same as the set of equivalence classes of topological $\G$-bundles on $X$. Denote this set by $\mathrm{Bun}_X(\G)$. 
This gives a decomposition of the Higgs bundle moduli space,
\[\cM_L(\G)=\coprod\limits_{a\in\mathrm{Bun}_X(\G)}\cM_L^a(\G)~,\]
where $a\in\mathrm{Bun}_X(\G)$ is the topological type of the underlying $\rH^\C$-bundle of the Higgs bundles in $\cM_L^a(\G).$ 

In general, the number of connected components of the moduli space of $\G$-Higgs has not been established. However, when $\G$ is {\em compact} and semisimple, the spaces $\cM^a(\G)$ are connected and nonempty \cite{ramanathan_1975}. Using Example \ref{Ex compact complex Higgs}, this implies the following proposition.
\begin{proposition}\label{prop pi0 H semisimple}
If $\G$ is a connected real semisimple Lie group such that the maximal compact subgroup $\rH$ is semisimple, then, for each $a\in\Bun_X(\G),$ the space $\cM^{a}(\G)$ is nonempty. Moreover, each space $\cM^a(\G)$ contains a unique connected component with the property that every Higgs bundle in it can be  deformed to a Higgs bundle with zero Higgs field. 
\end{proposition}

The above proposition implies that, when $\G$ is a semisimple {\em complex} Lie group, the space $\cM^a(\G)$ is nonempty for each $a\in\Bun_X(\G)$. In fact, each of the spaces $\cM^a(\G)$ is connected. This was proven for connected groups by Li \cite{JunLiConnectedness} and in general in \cite{Oliveira_GarciaPrada_2016}. In particular, we have the following:
\begin{corollary}
      If $\G$ is a semisimple complex Lie group, then every Higgs bundle  $(\cE,\varphi)\in\cM(\G)$ can be  deformed to a Higgs bundle with vanishing Higgs field. In particular,
      \[|\pi_0(\cM(\G))|=|\Bun_X(\G)|.\]
     \end{corollary}     
A semisimple Lie group $\G$ whose maximal compact subgroup is not semisimple but only reductive is called {\em a group of Hermitian type.} We will discuss this case in more detail in Section \ref{Section SO2q}.

We have $\rO(1)\cong \Z_2$ and $\rO(1)$-bundles are classified by their first Stiefel-Whitney class $sw_1\in H^1(X,\Z_2).$ For $p\geq 2$, topological $\rO(p)$-bundles have two characteristic classes, a first Stiefel-Whitney class and a second Stiefel-Whitney class $sw_2\in H^2(X,\Z_2).$ 
When the first Stiefel-Whitney class vanishes, the structure group can be reduced to $\SO(p).$ Since $\SO(2)$ is a circle, the second Stiefel-Whitney class of an $\rO(2)$-bundle lifts to the degree of a circle bundle when $sw_1=0.$ However, as an $\rO(2)$-bundle, it is only the absolute value of the degree which is a topological invariant. 
For $p>2,$ the first and second Steifel-Whitney classes classify topological $\rO(p)$-bundles over $X$, while the $\SO(p)$-bundles are classified topologically just by $sw_2$.

We will be particularly interested in the case of $K^p$-twisted $\SO(1,n)$-Higgs bundles and $K$-twisted $\SO(p,q)$-Higgs bundles. 
Since the maximal compact subgroup of $\SO(p,q)$ is $\rS(\rO(p)\times \rO(q))$, the Higgs bundles are determined by two orthogonal bundles which have the same first Stiefel-Whitney class. Let $\cM^{a,b,c}_L(\SO(p,q))$ denote the subset of $\SO(p,q)$-Higgs bundles $(V,Q_V,W,Q_V,\eta)$ so that 
\[\xymatrix{a=sw_1(V,Q_V)=sw_1(W,Q_W)&b=sw_2(V,Q_V)&\text{and}&c=sw_2(W,Q_W).}\]
These invariants are constant on connected components, thus we have a decomposition
\begin{equation}
  \label{eq sopq moduli top inv}\cM_L(\SO(p,q)) = \coprod\cM_L^{a,b,c}(\SO(p,q))~.
\end{equation}
Note that when $p=1$ the invariant $b$ is zero, while when $q=1$ then $c=0$.

The case of $\SO(2,q)$ with vanishing first Stiefel-Whitney class behaves differently.
Let $(V,W,\eta)$ be a polystable $K^p$-twisted $\SO(2,q)$-Higgs bundle with $sw_1(V)=0$. Then there is a line bundle $N$ so that the $\SO(2,\C)$-bundle $(V,Q_V)$ is isomorphic to 
\begin{equation}
  \label{eq:SO_0-2q-Higgs}
  (V,Q_V)\cong (N\oplus N^{-1}, \smtrx{0&1\\1&0})~.
\end{equation}
With respect to this splitting, the Higgs field $\eta:W\to V\otimes K^p$ decomposes as
\[\eta=\smtrx{\gamma\\\beta}:W\to (N\oplus N^{-1})\otimes K^p.\]
\section{The $\C^*$-action and its fixed points}

In this section we recall the definition of the $\C^*$-action on the
Higgs bundle moduli space and discuss its importance for the study of
the connected components of the moduli space of $\G$-Higgs
bundles. This method was pioneered by Hitchin
\cite{selfduality,liegroupsteichmuller} using gauge theoretic methods.  
%
%
For completeness we have included in
Appendix~\ref{sec:review-gauge} a brief review of some essential facts
coming from the gauge theoretic approach and how they translate into
the language of holomorphic geometry used in the main body of the paper.

\subsection{Definition and basic properties of the action}
\label{sec:basic-properties-action}

The action of $\C^*$ on the $L$-twisted Higgs bundle moduli space is defined by
scaling the Higgs field. Namely, $\lambda\cdot(\cE,\varphi)=(\cE,\lambda\varphi)$ for $\lambda\in\C^*$.
 Since this preserves the notions of
(poly)stability, it induces a holomorphic action on the moduli space.
By properness of the Hitchin fibration, if
$(\cE,\varphi)$ is the isomorphism class of a polystable $L$-twisted
$\G$-Higgs bundle, then the limit
$\lim\limits_{\lambda\to 0}(\cE,\lambda\varphi)$ exists and is a
polystable fixed point of the $\C^*$-action \cite{SimpsonVHS}.

\begin{notation}
  \label{not:equivalence}
  Note that we have denoted the isomorphism class of a Higgs bundle
  and the Higgs bundle itself with the same symbol. The context will
  always clarify which object we are referring to.
\end{notation}

Consider the function on the moduli space of $\G$-Higgs
bundles which assigns the $\mathrm{L}^2$-norm of the Higgs field with
respect to the harmonic metric solving the self-duality equations (cf.~\eqref{EQ Hitchin Function gauge}):
\begin{equation}\label{EQ Hitchin Function}
    f:\cM(\G)\to\R,\ \ \ (\cE,\varphi)\mapsto\int_X||\varphi||^2.
\end{equation}
We will refer to the function $f$ as the \emph{Hitchin function}.
Note that $f$ is non-negative and zero if and only if
$\varphi=0$. Using Uhlenbeck compactness, Hitchin showed that the map
$f$ is proper  and hence it attains local minima on
each closed subset of $\cM(\G)$ \cite{selfduality}. In particular, we have
\[|\pi_0(\cM(\G))|\leq|\pi_0(\mathrm{Min}(\cM(\G)))|,\]
where 
$\mathrm{Min}(\cM(\G))\subset\cM(\G)$ denotes the subset where $f$
attains a local minimum.

The starting point for determining the local minima of $f$ is the
following result (Lemma~\ref{lem:minima-fixed-HH-vanishing}):
\begin{proposition}
  \label{prop:minima-are-fixed}
  Let $(\cE,\varphi)$ be a $\G$-Higgs bundle such that 
  $\HH^0(C^\bullet(\cE,\varphi))=0$ and
  $\HH^2(C^\bullet(\cE,\varphi))=0$. If $(\cE,\varphi)$ is a local
  minimum of $f$ then it is a fixed point of the $\C^*$-action. 
\end{proposition}

In the situation of Proposition~\ref{prop:minima-are-fixed}
there is a weight space splitting (see Proposition \ref{prop:weight-decomposition-fixed-H}, \eqref{eq:weight-decomposition-adE} and also Section \ref{section SO(p,q) fixed points} for $\G=\SO(p,q)$) of the
Lie algebra bundle $\cE[\fg^\C] = \cE[\fh^\C] \oplus \cE[\fm^\C]$ as
\[\xymatrix{\cE[\fh^\C]=\bigoplus\cE[\fh^\C]_k&\text{and}&\cE[\fm^\C]=\bigoplus\limits\cE[\fm^\C]_k}\]
with $\varphi\in H^0(\cE[\fm^\C]_1\otimes K)$.
Thus, the complex $C^\bullet=C^\bullet(\cE,\varphi)$ defined in
\eqref{eq complex of sheaves} splits (see \eqref{eq:Ck-complex}) as $C^\bullet=\bigoplus
C^\bullet_k$, where
\begin{equation}
    \label{Eq fixedpoint deformation sheaf splitting}
C^\bullet_k=C^\bullet_k(\cE,\varphi):\xymatrix{\cE[\fh^\C]_{k}\ar[r]^-{\ad_\varphi}&\cE[\fm^\C]_{k+1}\otimes K,}
\end{equation}
yielding corresponding splittings
\begin{math}
  \HH^i(C^\bullet(\cE,\varphi)) = \bigoplus_k \HH^i(C^\bullet_{k}(\cE,\varphi)). 
\end{math}
There is also a corresponding splitting of the long exact sequence in cohomology from \eqref{EQ deformation complex DEF}:
\begin{equation}\label{EQ deformation complex splitting}    \xymatrix@R=1em{0\ar[r]&\HH^0(C^\bullet_{k})\ar[r]&H^0(\cE[\fh^\C]_k)\ar[r]^-{\ad_\varphi}&H^0(\cE[\fm^\C]_{k+1}\otimes K)\ar[r]&\HH^1(C^\bullet_k)\ar@{->}`r/3pt [d] `/10pt[l] `^d[llll] `^r/3pt[d][dlll]\\&H^1(\cE[\fh^\C]_{k})\ar[r]^-{\ad_\varphi}&H^1(\cE[\fm^\C]_{k+1}\otimes K)\ar[r]&\HH^2(C^\bullet_k)\ar[r]&0.}
\end{equation}

We have the following criterion for local minima of $f$ (see Lemma~\ref{lem:local-minima-HH-plus-vanishes}).

\begin{proposition}\label{prop:mincrit1}
 Let $(\cE,\varphi)$ be a $\G$-Higgs bundle which
    is a fixed point of the $\C^*$-action such that
    $\HH^0(C^\bullet(\cE,\varphi))=0$ and
    $\HH^2(C^\bullet(\cE,\varphi))=0$. Then $(\cE,\varphi)$ is a local
    minimum of the Hitchin function $f$ if and only if
    $\HH^1(C^\bullet(\cE,\varphi))_k=0$ for all $k>0$.
\end{proposition}

The following criterion for the vanishing in Proposition~\ref{prop:mincrit1} will be useful (see \cite[Section 3.4]{BGGHomotopyGroups}).

\begin{proposition}
    \label{prop: minima criteria}
    If $(\cE,\varphi)$ is a $\G$-Higgs bundle which is a
    fixed point of the $\C^*$-action such that $\HH^0(C^\bullet)=0$
    and $\HH^2(C^\bullet)=0$, then $(\cE,\varphi)$ is a local minimum
    of the Hitchin function $f$ if and only if either $\varphi=0$ or
    the map \eqref{Eq fixedpoint deformation sheaf splitting} is an
    isomorphism of sheaves for every $k>0.$
\end{proposition}


To classify the local minima of $f$, the following two results are needed
(with proofs given in the Appendix, where they appear as Lemmas~\ref{lem:reducing-minima} and \ref{lem:non-minimum-infinity-limit} respectively).

\begin{proposition}
  \label{prop:reducing-minima}
  Let $\G'\subset\G$ be a reductive subgroup. Suppose
  $(\cE,\varphi)$ is a $\G$-Higgs bundle which reduces to
  a $\G'$-Higgs bundle. If $(\cE,\varphi)$ is a minimum of the Hitchin
  function on $\cM(\G)$ then it is a minimum of the Hitchin function
  on $\cM(\G')$.
\end{proposition}

\begin{proposition}
\label{prop:notS-equiv-liminfty-notmin}
  Let $(\cE_0,\varphi_0)\in \mathcal{M}(\G)$ be a fixed point of the
  $\C^*$-action. Suppose there exists a semistable $\G$-Higgs bundle
  $(\cE,\varphi)$, which is not $\cS$-equivalent to $(\cE_0,\varphi_0)$,
  and such that $\lim_{t\to\infty}(\cE,t\varphi)=(\cE_0,\varphi_0)$ in
  $\mathcal{M}(\G)$. Then $(\cE_0,\varphi_0)$ is not a local minimum of
  $f$.  
\end{proposition}

The following result will help us show the vanishing
of $\HH^2(C^\bullet)$ for relevant Higgs bundles.  

\begin{lemma}\label{Lemma: reduction to fixed points H2=0}
If $(\cE,\varphi)$ is a polystable $L$-twisted Higgs bundle and $(\cE',\varphi')=\lim\limits_{\lambda\to0}(\cE,\lambda\varphi)$, then 
\[\dim\big(\HH^2(C^\bullet(\cE,\varphi))\big)\leq \dim\big(\HH^2(C^\bullet(\cE',\varphi'))\big).\]
\end{lemma}

\begin{proof}
    If $(\cE,\varphi)$ is fixed by the $\C^*$-action then we are done. If $(\cE,\varphi)$ is not fixed by $\C^*$, then consider the $\C^*$-family $(\cE,\lambda\varphi).$ Since $\lim\limits_{\lambda\to0}(\cE,\lambda\eta)$ exists, we can extend this to a family over $\mathbb A^1$, hence the result follows by semi-continuity of $\HH^2$.
\end{proof}

\begin{example}\label{ex: minima for GL(n,R) and Upq}
    The above minima criterion was used in \cite{BGGGLn} to classify all local minima for the group $\GL(n,\R)$, with $n\geq 2$, and in \cite{UpqHiggs} for the group $\U(p,q)$ (cf.\ Definition \ref{Def GLnR and Upq Higgs def}). For $\U(p,q)$, all minima $(E,F,\beta,\gamma)$ have either $\beta=0$ or $\gamma=0$.
    For $\GL(n,\R)$, and $n\geq 3$, the only local minima $(E,Q,\Phi)$ with non-zero Higgs field are the ones defining the Hitchin components. More precisely, they are given by 
\begin{equation}\label{eq:SLHitchin}
E=IK^{(n-1)/2}\oplus\cdots\oplus IK^{(1-n)/2},\ \ Q=\smtrx{&  &1\\ &\iddots& \\ 1&&}\text{ and }\ \Phi=\smtrx{0&&&\\1&0&&\\&\ddots&\ddots&\\&&1&0},
\end{equation} 
with $I$ a $2$-torsion line bundle. If $n=2$, the non-zero local minima are of the form
\begin{equation}\label{eq:SL2}
E=L\oplus L^{-1}\ \ Q=\smtrx{0& 1\\ 1 &0}\text{ and }\ \Phi=\smtrx{0&0\\ \Phi_1&0},
\end{equation}
 with $\Phi_1:L\to L^{-1}K$ non-zero and $0<\deg(L)\leq g-1$. 
\end{example}

\subsection{$\SO(p,q)$-fixed points}\label{section SO(p,q) fixed points}
We now focus on the details of fixed points of the $\C^*$-action on
the $L$-twisted $\SO(p,q)$-Higgs bundle moduli space. In order to get
a precise picture, the simplest approach is to analyze
these directly, following Simpson's procedure for usual Higgs (vector)
bundles \cite{localsystems}.

Let $(V,W,\eta)$ be a polystable $\SO(p,q)$-Higgs bundle with $(V,W,\eta)\cong(V,W,\lambda\eta)$ for all $\lambda\in\C^*$. If $\eta\neq0$, then for each $\lambda$ there are holomorphic orthogonal automorphisms $g_\lambda^V$ and $g^W_\lambda$ of $V$ and $W$ such that 
$(g^V_\lambda)^{-1}\cdot\eta\cdot g^W_\lambda=\lambda\eta$. Following Simpson, we take a $\lambda$ which is not a root of
unity. If we additionally take $\lambda\in S^1$ we may, using the gauge theoretic machinery of
  Appendix~\ref{sec:review-gauge}, take the automorphisms in
  the maximal compact subgroup, thus avoiding the 
  generalized eigenspaces considered by Simpson.

Let $V=\bigoplus_{\nu\in\R} V_\nu$ and
$W=\bigoplus_{\mu\in\R} W_\mu$ denote the
eigenbundle decompositions of $g^V_\lambda$ and $g^W_\lambda$ respectively, so that $g^V_\lambda|_{V_\nu}=\lambda^\nu\cdot \Id_{V_\nu}$ and $g_\lambda^W|_{W_\mu}=\lambda^\mu\cdot \Id_{W_\mu}.$
Since the gauge transformations $g_\lambda^V$ and $g_\lambda^W$ are
orthogonal, two eigenbundles $V_{\nu}$ and $V_{\nu'}$ or $W_{\mu}$ and
$W_{\mu'}$ are orthogonal if $\nu+\nu'\neq 0$ or
$\mu+\mu'\neq0$. Moreover, the quadratic forms define isomorphisms
$V_{\nu}\cong V_{-\nu}^*$ and $W_{\mu}\cong W_{-\mu}^*$.

For all weights $\mu$ and $\nu$, we have $\eta(W_\mu)\subset V_{\mu+1}\otimes L$ and $\eta^*(V_\nu)\subset W_{\nu+1}\otimes L$. Thus, $\eta=\sum\eta_\mu$ and $\eta^*=\sum\eta^*_\nu$, where
\begin{equation}\label{eq:Hodge decomposition of varphi}
\xymatrix{\eta_\mu=\eta|_{W_\mu}:W_\mu\longrightarrow V_{\mu+1}\otimes L&\text{and}&\eta^*_{-1-\nu}=\eta^*|_{V_\nu}:V_\nu\longrightarrow W_{\nu+1}\otimes L}.
\end{equation}
We may decompose $V\oplus W$ into a direct sum of minimal unbroken
chains of $V_{\nu}$'s and $W_{\mu}$'s connected by non-zero Higgs
fields. Consider such a chain
\begin{displaymath}
  V_{a} \xrightarrow{\eta^*_{-a-1}} W_{a+1} \xrightarrow{\eta_{a+1}} \dots 
\end{displaymath}
For simplicity of notation, we have suppressed the twisting by $L$ from the Higgs field. This will be done every time we use these chain representations.
We now consider two cases. (Of course similar arguments will apply for
chains starting with a $W_\mu$.)

\emph{Case (1).}
Suppose $V_{-a} \cong V_{a}^*$ is among the bundles of the chain. Then
$W_{-a-1} \cong W_{a+1}^*$ is also among the bundles of the chain,
because the non-zero map $V_{a} \to W_{a+1}$ is dual to
$W_{-a-1} \to V_{-a}$. Moreover, $V_{-a}$ is evidently the last bundle
of the chain. Thus, the weights must be integers and the restriction of the quadratic forms on
$V$ and $W$ to the
chain is non-degenerate.

\emph{Case (2).} Suppose now that
$V_{-a} \cong V_{a}^*$ is not among the bundles of the chain. Then,
arguing in a similar way to case (1), we see that $W_{-a-1}$ cannot be in the
chain either. 
In this case the chain is isotropic for the quadratic forms on
$V$ and $W$. Note that the weights are only well defined up to
overall translation on such a chain.

We summarize the above characterization of $\C^*$-fixed points in the following proposition.

\begin{proposition}\label{prop: fixed point characterization} 
  If $(V,W,\eta)$ is a polystable $L$-twisted $\SO(p,q)$-Higgs bundle which is a fixed point of the $\C^*$-action with $\eta\neq 0$, then it is a direct sum
of holomorphic chains with non-zero Higgs fields of the following two types:
\begin{equation}\label{eq: holomorphic chain int} 
\xymatrix@R=0em{\cdots\ar[r]^-{\eta_{-3}}&V_{-2}\ar[r]^-{\eta_1^*}&W_{-1}\ar[r]^-{\eta_{-1}}&V_0\ar[r]^-{\eta_{-1}^*}&W_1\ar[r]^-{\eta_1}&V_2\ar[r]^-{\eta_{-3}^*}&\cdots\\&&&\oplus&&&  \\
\cdots\ar[r]^-{\eta_2^*}&W_{-2}\ar[r]^-{\eta_{-2}}&V_{-1}\ar[r]^-{\eta^*_{0}}&W_0\ar[r]^-{\eta_0}&V_1\ar[r]^-{\eta_{-2}^*}&W_2\ar[r]^-{\eta_2}&\cdots
}
\end{equation}
or 
\begin{equation}
\label{eq: holomorphic chain half int} 
\xymatrix@R=0em{\cdots\ar[r]^-{\eta_{a-1}}&V_{a}\ar[r]^-{\eta_{-a-1}^*}&W_{a+1}\ar[r]^-{\eta_{a+1}}&V_{a+2}\ar[r]^-{\eta_{-a-3}^*}&W_{a+3}\ar[r]^-{\eta_{a+3}}&\cdots\\&&&\hspace{-1.75cm}\oplus&&&  \\
\cdots\ar[r]^-{\eta_{a+3}^*}&W_{-a-3}\ar[r]^-{\eta_{-a-3}}&V_{-a-2}\ar[r]^-{\eta^*_{a+1}}&W_{-a-1}\ar[r]^-{\eta_{-a-1}}&V_{-a}\ar[r]^-{\eta_{a-1}^*}&\cdots
}
\end{equation}
where the corresponding quadratic forms define isomorphisms $V_j\cong (V_{-j})^*$ and $W_j\cong (W_{-j})^*$.
The two chains in \eqref{eq: holomorphic chain half int} are dual to each other.
\end{proposition}

Proposition \ref{prop: fixed point characterization} provides a characterization of polystable $\C^*$-fixed points with non-vanishing Higgs field. The next result shows that stability imposes further conditions on such fixed points. 

\begin{proposition}\label{prop: stable fixed points}
  Suppose $(p,q)\neq (2,2)$. If $(V,W,\eta)$ is a stable $L$-twisted $\SO(p,q)$-Higgs bundle which is a $\C^*$-fixed point, then it is represented by a chain of type \eqref{eq: holomorphic chain int}.
\end{proposition}
\begin{proof}
  Suppose $(V,W,\eta)$ is represented by \eqref{eq: holomorphic chain half int}. Consider the subbundles $V'\subset V$ and $W'\subset W$ formed by the summands of the first chain. This is a pair of isotropic $\eta$-invariant subbundles (at least one of which is proper because $(p,q)\neq (2,2)$), and the same is true for the pair $V'^*\subset V$ and $W'^*\subset W$ formed by the summands of the second chain. 
Since $\deg(V')+\deg(W')=-\deg(V'^*)-\deg(W'^*)$, such an $\SO(p,q)$-Higgs bundle is not stable. This argument also shows that if $(V,W,\eta)$ has a summand given by \eqref{eq: holomorphic chain half int}, then it is not stable.
\end{proof}

\subsection{Special fixed points on $\cM(\SO(2,q))$}
When $p=2$, we have 
special fixed points of the form
\begin{equation}
  \label{eq:SO2q-absolute-minimum}
  V_{-1}\xrightarrow{\eta_0^*}W_0\xrightarrow{\eta_0}V_{1}~,
\end{equation}
where $V_{-1}\cong V_1^*$ and $\eta_0\neq 0$.
Note that $\deg(V_1)<0$ by polystability. Also, such a Higgs bundle is of the form \eqref{eq:SO_0-2q-Higgs} with \emph{either} $N=V_1$,
$\gamma=\eta_0$ and $\beta=0$, \emph{or}  $N^{-1}=V_1$,
$\beta=\eta_0$ and $\gamma=0$. Conversely, an $\SO(2,q)$-Higgs bundle of the form \eqref{eq:SO_0-2q-Higgs} with exactly one of  $\beta$ or $\gamma$ zero is such a fixed point.

\begin{proposition}
  \label{prop:SO2q-absolute-minimum}
  Any $\SO(2,q)$-Higgs bundle $(V,W,\eta)$ which is a fixed point of the
  $\C^*$-action of the form \eqref{eq:SO2q-absolute-minimum} has $sw_1(V)=sw_1(W)=0$ and represents a
  local minimum of the Hitchin function.
\end{proposition}
\begin{proof}
  The vanishing of the first Stiefel-Whitney class is immediate from
  $V=V_1^*\oplus V_1$. To see that such a fixed point is a minimum,
  associate to it the $\U(1,q)$-Higgs bundle
  $(V_1,W_0,\eta_0,0)$. Since a $\U(1,q)$-Higgs bundle with $\gamma=0$
  is a minimum of the Hitchin function on its respective moduli space
  \cite{UpqHiggs} the conclusion follows by
  Proposition~\ref{prop:reducing-minima}.
\end{proof}


Fixed points of the $\C^*$-action in $\cM(\SO(2,2))$ are particularly
easy to describe using \eqref{eq: holomorphic chain int} and
\eqref{eq: holomorphic chain half int}. Let $(V,W,\eta)$ be an $\SO(2,2)$-Higgs bundle. If $sw_1(V)=sw_1(W)\neq 0$, then neither $V$ nor $W$ have holomorphic isotropic subbundles, thus $(V,W,\eta)$ is a fixed point if and only if $\eta=0.$ If $sw_1(V)=sw_1(W)=0,$ then $V=N\oplus N^{-1}$ and $W=M\oplus M^{-1}$ where $N$ and $M$ are isotropic line bundles. Up to switching the roles of $N$, $M$, $N^{-1}$ and $M^{-1}$, the holomorphic chains are given by 
\begin{equation}
  \label{EQ SO(2,2) fixed points}
  \xymatrix@R=0em{M\ar[r]^-{\smtrx{a\\b}}&N\oplus N^{-1}\ar[r]^-{\smtrx{b&a}}&M^{-1}},
\end{equation}
which are of the form \eqref{eq:SO2q-absolute-minimum}. Hence, in view
of Proposition~\ref{prop:SO2q-absolute-minimum}, we have the following result.

\begin{proposition}\label{prop SO(2,2) fixed points are minima}
Every fixed point in $\cM(\SO(2,2))$ is a local minimum.
\end{proposition}

\subsection{$\SO(1,n)$-fixed points and local structure of $\cM_{K^p}(\SO(1,n))$}

Recall from Definition \ref{DEF SO(p,q) Higgs bundles} that a $K^p$-twisted $\SO(1,n)$-Higgs bundle is a tuple $(I,Q_I,W,Q_W,\omega,\eta)$.
 Note that the isomorphism $(-\Id_I\oplus \Id_W):I\oplus W\to I\oplus W$ acts on such a tuple by $(I,Q_I,W,Q_W,\omega,\eta)\mapsto (I,Q_I,W,Q_W,-\omega,-\eta)$. In particular, for $\C^*$-fixed points, the isomorphism class is independent of the choice of $\omega.$ This implies that the two choices of orientation define $\SO(1,n)$-Higgs bundles which are in the same connected component. For this reason, we ignore the orientation in this section.

\begin{lemma}\label{lemma: fixed point SO(1,n)}
    If $(I,W,\eta)$ is a polystable $K^p$-twisted $\SO(1,n)$-Higgs bundle which is a $\C^*$-fixed point with $\eta\neq 0$, then it decomposes as
    \[(I,W,\eta)\cong \Big(I,W_{-1}\oplus W_0\oplus W_{1},\ \big(\eta_{-1}\ \ \ 0\ \ \ 0\big)\Big)~,\]
    where $(W_0,Q_0)$ is a polystable orthogonal bundle and $W_{1}\cong W_{-1}^*.$ 
    Furthermore, $\big(I,W_{-1}\oplus W_1, \big(\eta_{-1}\ \ 0\big)\big)$ is a stable $K^p$-twisted $\rO(1,n')$-Higgs bundle which is stable as a $K^p$-twisted $\rO(n'+1,\C)$-Higgs bundle.
    In the notation of \eqref{eq: holomorphic chain int}, such an $(I,W,\eta)$ is given by the chain
    \[\xymatrix@R=0em{W_{-1}\ar[r]^{\eta_{-1}}&I\ar[r]^{\eta_{-1}^*}&W_1\\&\oplus&\\&W_0&}.\]
 \end{lemma}
 \begin{proof}
The first part of the statement follows directly from Proposition \ref{prop: fixed point characterization}.
Since the bundles $W_1$ and $W_{-1}$ are isotropic, if $W_1$ has a degree zero subbundle $U$, then $W_{-1}$ has $U^*$ as a subbundle contained in the kernel of $\eta_{-1}$ by polystability. 
We may thus assume that the invariant polystable orthogonal subbundle $U^*\oplus U$ is a summand of $W_0$. 
 Now since $(W_{-1}\oplus W_{1},I,\smtrx{\eta_{-1}&0})$ is a stable $\rO(1,n')$-Higgs bundle, the associated $\rO(n'+1,\C)$-Higgs bundle is stable by \cite[Proposition 2.7]{MartaSO(1n)}.
\end{proof}

As in \eqref{Eq fixedpoint deformation sheaf splitting} with $K$ replaced with $K^p$, at a $\C^*$-fixed point $(I,W,\eta)\cong(I,W_{-1}\oplus W_0\oplus W_{1},\ (\eta_{-1}\ \ \ 0\ \ \ 0))$ in $\cM_{K^p}(\SO(1,n))$ the deformation complex \eqref{EQ ad eta definition} splits as $C^\bullet(I,W,\eta) = \bigoplus C^\bullet_k,$ where
\begin{displaymath}
  C^\bullet_k : \xymatrix@R=0em{\fso_k(I)\oplus\fso_k(W)\ar[r]^-{\ad_\eta}&\Hom_{k+1}(W,I)\otimes K^p~~.}
\end{displaymath}
We have $\fso(I)=0$ and $\End(W_{-1}\oplus W_0\oplus W_1)=\bigoplus\limits_{j=-2}^2\End_{j}(W)$,
where
\[\xymatrix@=.2em{\End_{2}(W)^*=\End_{-2}(W)=\Hom(W_{1},W_{-1}),\\ \End_1(W)^*=\End_{-1}(W)=\Hom(W_{1},W_0)\oplus \Hom(W_0,W_{-1}),\\\End_0(W)=\End(W_{-1})\oplus \End(W_0)\oplus \End(W_{1}).}\]
This gives the grading on $\fso(W)=\bigoplus\limits_{j=-2}^{2}\fso_j(W)$, where 
\[\xymatrix@=0em{\fso_2(W)^*=\fso_{-2}(W)=\{\beta\in\Hom(W_1,W_{-1})\ |\ \beta+\beta^*=0\},\\
\fso_1(W)^*=\fso_{-1}(W)=\{(\beta,-\beta^*)\in\End_{-1}(W)\},\\
\fso_0(W)=\{(\beta_{-1},\beta_0,-\beta_{-1}^*)\in \End_0(W)\ |\ \beta_0+\beta_0^*=0\}.}\]
Notice that $\fso_0(W)\cong\fso(W_0)\oplus \End(W_{-1})$, where $\fso(W_0)$ is the bundle of skew-symmetric endomorphisms of $W_0$ with respect to $Q_0$.
Also, $\Hom(W,I)\otimes K^p=\Hom_{-1}(W,I)\otimes K^p\oplus \Hom_0(W,I)\otimes K^p\oplus \Hom_1(W,I)\otimes K^p$, where 
\[\xymatrix{\Hom_{\pm1}(W,I)\otimes K^p=\Hom(W_{\mp 1},I)\otimes K^p&\text{and}&\Hom_{0}(W,I)\otimes K^p=\Hom(W_0,I)\otimes K^p.}\]

For each $k=-2,\ldots, 2$, the above splittings give
 $\ad_\eta:\fso_k(W)\to\Hom_{k+1}(W,I)\otimes K^p$, where $\ad_\eta$ is defined by composing with $\eta_{-1}$. This yields long exact sequences in cohomology
\begin{equation}
    \label{eq graded hypercohomology SO(1,n) sequence}\xymatrix@R=1em@C=1.5em{0\ar[r]&\HH^0(C^\bullet_k)\ar[r]&H^0(\fso_k(W))\ar[r]^-{\eta_{-1}}&H^0(\Hom_{k+1}(W,IK^p))\ar[r]&\HH^1(C^\bullet_k)\ar@{->}`r/3pt [d] `/10pt[l] `^d[llll] `^r/3pt[d][dlll]\\&H^1(\fso_k(W))\ar[r]^-{\eta_{-1}}&H^1(\Hom_{k+1}(W,IK^p))\ar[r]&\HH^2(C^\bullet_k)\ar[r]&0~.}
\end{equation}

\begin{lemma}\label{lemma vanishing H2 SO1n}
    For $p>1$, if $(I,W,\eta)$ is a polystable $K^p$-twisted $\SO(1,n)$-Higgs bundle, then the second hypercohomology group $\HH^2(C^\bullet(I,W,\eta))$ vanishes. 
\end{lemma}
\begin{proof}
By Lemma \ref{Lemma: reduction to fixed points H2=0}, to show that $\HH^2(C^\bullet(I,W,\eta))$ vanishes it suffices to show the vanishing of each graded piece of \eqref{eq graded hypercohomology SO(1,n) sequence} at a fixed point of the $\C^*$-action. Such fixed points are given by Lemma \ref{lemma: fixed point SO(1,n)}.
    
    First note that $\HH^2(C^\bullet_k)=0$ for $k\geq 1$ since $\Hom_{k+1}(W,I)=0$ for $k\geq1.$ 
    Stability implies $W_1$ and $W_0$ have no positive degree subbundles, and, by Serre duality, we have 
\[H^1(\Hom_{k+1}(W,IK^p))\cong \begin{dcases}
        H^0(\Hom(IK^{p-1},W_1))^* &k=-2\\
        H^0(\Hom(IK^{p-1}, W_0))^* &k=-1~.
    \end{dcases}\]
    Thus, since $p>1,$ $H^1(\Hom_{k+1}(W,IK^p))=0$ for $k\leq -1.$
    
Finally, the form of the Higgs field implies the kernel of $\ad_\eta:\fso_0(W)\to \Hom_1(W,I)\otimes K^p$ is $\fso( W_0)$. Hence, $\HH^2(C^\bullet_0)$ injects into the second hypercohomology group of the stable $\rO(1,n')$-Higgs bundle $\big(I,W_{-1}\oplus W_1, \big(\eta_{-1}\ \ 0\big)\big).$
The associated $\rO(n'+1,\C)$-Higgs bundle is stable by Lemma \ref{lemma: fixed point SO(1,n)}, so this hypercohomology group vanishes 
by Remark \ref{remark duality of H0 and H2 for complex groups}.  
\end{proof}

\begin{lemma}\label{Lemma: H1 of SO(1,n) fixed point}
    If $p>1$ and $(I,W,\eta)=\big(I,W_{-1}\oplus W_0\oplus W_1,\big(\eta_{-1}\ \ 0\ \ 0\big)\big)$ is a polystable $K^p$-twisted $\SO(1,n)$-Higgs bundle which is a $\C^*$-fixed point, then  
  \[\HH^0(C^\bullet)\cong H^0(\fso(W_0))\ \ \text{ and }\ \ \HH^1(C^\bullet)=\bigoplus\limits_{k=-2}^2 \HH^1(C^\bullet_k).\] Moreover,
     \begin{itemize}
         \item $\HH^1(C^\bullet_2)\cong H^1(\fso_{2}(W))$,
         \item $\HH^1(C^\bullet_1)\cong H^1(\Hom(W_{-1},W_0))$,
            \item $\HH^1(C^\bullet_0)\cong H^1(\fso(W_0))\oplus \HH^1_0$, where $\HH^1_0$ is defined by the sequence
    \[\xymatrix@R=1em{0\ar[r]& H^0(\End(W_{-1}))\ar[r]^-{\eta_{-1}}&H^0(\Hom(W_{-1},I K^p))\ar[r]&\HH^1_0\ar@{->}`r/3pt [d] `/10pt[l] `^d[lll] `^r/3pt[d][dll]\\& H^1(\End(W_{-1}))\ar[r]^-{\eta_{-1}}& H^1(\Hom(W_{-1},IK^p))\ar[r]&0~,}\]
    \item $\HH^1(C^\bullet_{-1})$ is defined by the sequence
    \[\xymatrix@C=1.5em{0\ar[r]& H^0(\Hom(W_0,W_{-1}))\ar[r]^-{\eta_{-1}}&H^0(\Hom(W_0,I K^p))\ar[r]&\HH^1(C^\bullet_{-1})\ar[r]& H^1(\Hom(W_0,W_{-1}))\ar[r]&0~,}\]
    \item $\HH^1(C^\bullet_{-2})$ is defined by the sequence
    \[\xymatrix@C=1.5em{0\ar[r]& H^0(\fso_{-2}(W))\ar[r]^-{\eta_{-1}}&H^0(\Hom(W_1,IK^p))\ar[r]&\HH^1(C^\bullet_{-2})\ar[r]& H^1(\fso_{-2}(W))\ar[r]& 0~.}\]
     \end{itemize}
\end{lemma}
\begin{proof}
By Lemma \ref{lemma: fixed point SO(1,n)}, a $\C^*$-fixed point is given by $(I,W,\eta)=\big(I,W_{-1}\oplus W_0\oplus W_1,\big(\eta_{-1}\ \ 0\ \ 0\big)\big)$, where  $W_0$ is a polystable orthogonal bundle and $(I,W_{-1}\oplus W_1,\smtrx{\eta_{-1}&0})$ is a stable $\rO(1,n')$-Higgs bundle such that the associated $\rO(n'+1,\C)$-Higgs bundle is also stable. In particular, $W_1$ has no non-negative degree subbundles and $W_0$ has no positive degree subbundles. Recall that in the proof of Lemma \ref{lemma vanishing H2 SO1n} it was shown that $H^1(\Hom_{k+1}(W,IK^p))=0$ for $k\leq -1.$

For $k=2,$ we have $C_2^\bullet: \fso_2(W)\to 0$, thus, $\HH^0(C_2^\bullet)=H^0(\fso_2(W))$ and $\HH^1(C_2^\bullet)=H^1(\fso_2(W))$. In particular, $\HH^0(C_2^\bullet)$ injects into the zeroth hypercohomology group of the deformation complex of the $\rO(1,n')$-Higgs bundle $(I,W_{-1}\oplus W_1,\smtrx{\eta_{-1}&0}),$ which vanishes by stability. 

For $k=1,$ $\fso_1(W)\cong\Hom(W_{-1},W_0)$ and $C_1^\bullet:\fso_1(W)\to 0$ imply $\HH^0(C_1^\bullet)=H^0(\Hom(W_{-1},W_0))$ and $\HH^1(C_1^\bullet)=H^1(\Hom(W_{-1},W_0))$. The vanishing of $H^0(\Hom(W_{-1},W_0))\cong H^0(\Hom(W_0,W_1))$ follows from stability. Namely, any non-zero homomorphism $f:W_0\to W_1$ defines a non-negative degree subbundle of $W_1,$ contradicting the stability of $(I,W_{-1}\oplus W_1,\smtrx{\eta_{-1}&0})$. 

For  $k=0$, $C_0^\bullet:\fso_0(W)\to\Hom_1(W,I)\otimes K^p$ is given by
\[C_0^\bullet:\End(W_{-1})\oplus\fso(W_0)\to\Hom(W_{-1},I)\otimes K^p,\ \ \ \ (\beta_{-1},\beta_0)\mapsto\eta_{-1}\beta_{-1}.\]
Thus, we can split $C_0^\bullet$ as $C_0^\bullet=C_0^{\bullet,\prime}\oplus C_0^{\bullet,\prime\prime}$ with $C_0^{\bullet,\prime}:\End(W_{-1})\xrightarrow{\eta_{-1}}\Hom(W_{-1},I)\otimes K^p$ and $C_0^{\bullet,\prime\prime}:\fso(W_0)\to 0$. 
The hypercohomology groups split accordingly, hence
\[\xymatrix{\HH^0(C_0^{\bullet,\prime\prime})=H^0(\fso(W_0))&\text{and}&\HH^1(C_0^{\bullet,\prime\prime})\cong H^1(\fso(W_0))}~.\] 
For $C_0^{\bullet,\prime}$, $\HH^0(C_0^{\bullet,\prime})=0$ by stability of $\big(I,W_{-1}\oplus W_1, \big(\eta_{-1}\ \ 0\big)\big)$. Thus, if $\HH^1_0=\HH^1(C_0^{\bullet,\prime})$, we have
\[\xymatrix@R=1em{0\ar[r]& H^0(\End(W_{-1}))\ar[r]^-{\eta_{-1}}&H^0(\Hom(W_{-1},I K^p))\ar[r]&\HH^1_0\ar@{->}`r/3pt [d] `/10pt[l] `^d[lll] `^r/3pt[d][dll]\\& H^1(\End(W_{-1}))\ar[r]^-{\eta_{-1}}& H^1(\Hom(W_{-1},IK^p))\ar[r]&0~.}\]

For $k=-1,$ we have $H^1(\Hom_{0}(W,IK^p))=0$ and $C_{-1}^\bullet: \Hom(W_0,W_{-1})\xrightarrow{\eta_{-1}}\Hom(W_0,I)\otimes K^p$. Thus, \[\xymatrix@R=1em{0\ar[r]&\HH^0(C_{-1}^\bullet)\ar[r]& H^0(\Hom(W_0,W_{-1}))\ar[r]^-{\eta_{-1}}&H^0(\Hom(W_0,I K^p))\ar@{->}`r/3pt [d] `/10pt[l] `^d[lll] `^r/3pt[d][dll]\\ &\HH^1(C^\bullet_{-1})\ar[r]& H^1(\Hom(W_0,W_{-1}))\ar[r]&0~.}\]
It remains to show that $\HH^0(C_{-1}^\bullet)=0$. If $N$ is the kernel of $\eta_{-1}:W_{-1}\to IK^p$, then $\HH^0(C_{-1}^\bullet)\cong H^0(\Hom(W_0,N))$. If $N=0$ we are done so suppose $N\neq 0$. Stability of $\big(I,W_{-1}\oplus W_1, \big(\eta_{-1}\ \ 0\big)\big)$ implies $\deg(N)<0$ and moreover $N$ has no non-negative degree subbundles. A non-zero section $\beta\in H^0(\Hom(W_0,N))$ must have a non-trivial kernel since otherwise $\beta(W_0)\subset N$ would define a non-negative degree subbundle. However, this implies that $\deg(\ker(\beta))>0$, contradicting the polystability of $W_0$. We conclude that $H^0(\Hom(W_0,N))=0,$ and thus $\HH^0(C_{-1}^\bullet)=0$.

Finally consider the case of $C_{-2}^\bullet: \fso_{-2}(W)\xrightarrow{\ad_\eta}\Hom(W_1,I)\otimes K^p$. As in the case $k=2,$ stability of the $\rO(1,n')$-Higgs bundle $(I,W_{-1}\oplus W_1,\smtrx{\eta_{-1}&0})$ implies $\HH^0(C_{-2}^\bullet)=0$. The group $\HH^1(C_{-2}^\bullet)$ is defined by the exact sequence in the statement of the lemma since $H^1(\Hom(W_1,IK^p))=0$. 
\end{proof}

\section{Existence of exotic components of $\cM(\SO(p,q))$}\label{sec:exoticcomp}

In this section we will prove the following theorem exhibiting connected components of $\cM(\SO(p,q))$ which are not distinguished by primary characteristic classes for $p\geq 2$.  
\begin{theorem}\label{Thm Psi open and closed}
Let $X$ be a compact Riemann surface with genus $g\geq2$ and canonical bundle $K$. Denote the moduli space of $K^p$-twisted $\SO(1,q-p+1)$-Higgs bundles on $X$ by $\cM_{K^p}(\SO(1,q-p+1))$ and the moduli space of $K$-twisted $\SO(p,q)$-Higgs bundles on $X$ by $\cM(\SO(p,q)).$ For $1\leq p\leq q,$ there is a well defined map
\begin{equation}\label{Eq Psi}
\Psi:\xymatrix{\cM_{K^p}(\SO(1,q-p+1))\times\displaystyle\bigoplus\limits_{j=1}^{p-1}H^0(K^{2j})\ar[r]&\cM(\SO(p,q))}
\end{equation}
which is an isomorphism onto its image and has an open and closed image. Furthermore, if $p\geq 2$, then every Higgs bundle in the image of $\Psi$ has a nowhere vanishing Higgs field. 
 \end{theorem}

\begin{remark}\label{remark:lowerboundconncomp}
As a direct corollary of the above theorem, we have that, for $p>2$,
    \[\left| \pi_0\big(\cM(\SO(p,q))\big)\right|\ \geq\  2^{2g+2}+\left|\pi_0\big(\cM_{K^p}(\SO(1,q-p+1))\big)\right|.\]
    In Theorem \ref{Theorem: component count p>2}  we will show that the above inequality is in fact an equality.
 \end{remark}
\begin{remark}
The space of holomorphic differentials $H^0(K^{2j})$ can be identified with the moduli space $\cM_{K^{2j}}(\R^+)$. This will be used in Section \ref{section Cayley partner}, to interpret Theorem \ref{Thm Psi open and closed} as a generalized Cayley correspondence.
\end{remark}

\subsection{Defining the map $\Psi$}

Recall that a $K^p$-twisted $\SO(1,n)$-Higgs bundle is a triple $(I,\widehat W,\hat\eta)$, where $\widehat W$ is a rank $n$ vector bundle with an orthogonal structure $Q_{\widehat W}$, $I\cong\Lambda^n\widehat W$ and $\hat\eta\in H^0(\Hom(\widehat W,I)\otimes K^p)$.

Let $\cH_{K^p}(\SO(1,q-p+1))$ denote the configuration space of all $K^p$-twisted $\SO(1,q-p+1)$-Higgs bundles and let $\cH(\SO(p,q))$ denote the configuration space of all $\SO(p,q)$-Higgs bundles. That is, $\cH_{K^p}(\SO(1,q-p+1))$ consists of pairs $(\dbar_{\widehat W},\hat\eta)$ where $\dbar_{\widehat W}$ is a Dolbeault operator on $\widehat W$, $\hat\eta\in \Omega^{1,0}(\Hom(\widehat W,\Lambda^{q-p+1}\widehat W))$ such that $\dbar_{\widehat W}\hat\eta=0$ and $\dbar_{\widehat W}Q_{\widehat W}=0$. The space $\cH(\SO(p,q))$ is defined analogously.

Recall that the Hitchin section $s_H^I:\bigoplus\limits_{j=1}^{p-1}H^0(K^{2j})\to\cM(\SO(p,p-1))$ is given by \eqref{EQ Hitchin section O(p,C)}, and that 
\[(I\otimes \cK_n,Q_n)=\left(I\otimes(K^{n}\oplus K^{n-2}\oplus\cdots\oplus K^{2-n}\oplus K^{-n}),\smtrx{&&1\\&\iddots&\\1&&}\right).\]
Recall that the Higgs field in the image of $s_H^I$ is given by $\eta(q_2,\ldots,q_{2p-2}):I\otimes \cK_{p-2}\to I\otimes \cK_{p-1}\otimes K$, as in \eqref{eq choice of Hitchin section O(p-1,p)}.

Define the map 
\begin{equation}\label{EQ tildePsi-domain}
\widetilde\Psi:\xymatrix{\cH_{K^p}(\SO(1,q-p+1))\times\displaystyle\bigoplus\limits_{j=1}^{p-1}H^0(K^{2j})\ar[r]&\cH(\SO(p,q))}
\end{equation}
by
\begin{equation}
    \label{EQ tildePsi}
    \widetilde\Psi((I, \widehat W,\hat\eta),q_2,\ldots,q_{2p-2})=\left(I\otimes \cK_{p-1},\widehat W\oplus I\otimes \cK_{p-2},\Big(\eta_{\widehat W}\ \ \eta(q_2,\ldots,q_{2p-2})\Big)\right)
\end{equation}
where   
\[\eta_{\widehat W}=\smtrx{\hat\eta\\0\\\vdots\\0}:\xymatrix{\widehat W \ar[r]&I\otimes(K^{p}\oplus K^{p-2}\oplus\cdots\oplus K^{2-p})=I\otimes \cK_{p-1}\otimes K}.\]
It is clear that the map $\widetilde \Psi$ is continuous.
\begin{remark}
When defining the map $\widetilde \Psi,$ we have ignored the orientations of the $\SO(1,n)$ and $\SO(p,q)$-Higgs bundles. An orientation $\hat\omega:I\otimes \Lambda^{q-p+1}\widehat W\to\cO$ clearly induces an orientation $\omega:I^{p}\otimes I^{p-1}\otimes \Lambda^{q-p+1}\widehat W\to \cO$ on the image. Moreover, the two choices of orientation will not define different components of the moduli space (see Remark \ref{rem: orientation doesn't change component count}). 
\end{remark}

\begin{lemma}
    For $(I,\widehat W,\hat\eta,q_2,\ldots,q_{2p-2})\in \cH_{K^p}(\SO(1,q-p+1))\times\bigoplus\limits_{j=1}^{p-1}H^0(K^{2j})$, the $\SO(p,q)$-Higgs bundle $\widetilde\Psi(I,\widehat W,\hat\eta,q_2,\ldots,q_{2p-2})$ is (poly)stable  if and only if the $K^p$-twisted $\SO(1,q-p+1)$-Higgs bundle $(I,\widehat W,\hat\eta)$ is (poly)stable.  
\end{lemma}
\begin{proof}
Fix $(I,\widehat W,\hat\eta,q_2,\ldots,q_{2p-2})\in \cH_{K^p}(\SO(1,q-p+1))\times\bigoplus\limits_{j=1}^{p-1}H^0(K^{2j})$.
Recall that an $\SO(p,q)$-Higgs bundle is polystable if and only if the associated $\SL(p+q,\C)$-Higgs bundle is polystable. 
Suppose first that $q_{2j}=0$ for all $j$. Then the $\SL(p+q,\C)$-Higgs bundle associated to the image of $\widetilde\Psi(I,\widehat W,\hat\eta,0,\ldots,0)$ is represented by
    \[\xymatrix@R=1em{IK^{p-1}\ar[r]^-1&IK^{p-2}\ar[r]^-1&\cdots\ar[r]^-1&IK^{2-p}\ar[r]^-1&IK^{1-p}\ar[dll]^-{\hat\eta^*}\\ &&\widehat W\ar[ull]^-{\hat\eta}&&}.\]
To check (poly)stability for such a ``cyclic'' Higgs bundle, it suffices to show that each of the bundles in the above cycle do not contain an invariant destabilizing subbundle (see Proposition 6.3 of \cite{KatzMiddleInvCyclicHiggs}). 
Thus $\widetilde\Psi(I,\widehat W,\hat\eta,0,\ldots,0)$ is polystable if and only if there are no destabilizing subbundles of $\widehat W$ in the kernel of $\hat\eta$, that is, if and only if $(I,\widehat W,\hat\eta)$ is polystable. 
Furthermore, since $\widetilde\Psi(I,\widehat W,\hat\eta,0,\ldots,0)$ is strictly polystable if and only if $\widehat W$ contains a degree zero isotropic subbundle in the kernel of $\hat\eta$, we conclude that $\widetilde\Psi(I,\widehat W,\hat\eta,0,\ldots,0)$ is stable if and only if $(I,\widehat W,\hat\eta)$ is stable.

Now suppose $(q_2,\ldots,q_{2p-2})\neq (0,\ldots,0)$ and let $(V,W,\eta)=\widetilde\Psi(I,\widehat W,\hat\eta,q_2,\ldots,q_{2p-2})$ be given by \eqref{EQ tildePsi}. For $\lambda\in\C^*,$ consider the following holomorphic orthogonal gauge transformations of $V$ and $W$
\[g_V=\smtrx{\lambda^{1-p}&&&\\&\lambda^{3-p}&&\\&&\ddots&\\&&&\lambda^{p-1}}\ \ \ \text{and}\ \ \ g_W=\smtrx{\Id_{\widehat W}&&&&\\&\lambda^{2-p}&&&\\&&\lambda^{4-p}&&\\&&&\ddots&\\&&&&\lambda^{p-2}}.\]
Using the description of $s_H^I$ from \eqref{EQ Hitchin section O(p,C)} and \eqref{eq choice of Hitchin section O(p-1,p)}, a straightforward computation shows that 
\begin{equation}\label{EQ scaling gauge action}
    (g_V,g_W)\cdot(V,W,\lambda\eta)=\widetilde\Psi(I,\widehat W,\lambda^{p}\hat\eta,\lambda^2 q_2,\lambda^4q_4,\ldots,\lambda^{2p-2} q_{2p-2}).
\end{equation}
 Assume $(I,\widehat W,\hat\eta)$ is stable. In particular,  $(I,\widehat W,\lambda^{p}\hat\eta)$ is a stable $K^p$-twisted $\SO(1,q-p+1)$-Higgs bundle for all $\lambda\in\C^*$. 
 By the above argument, $\widetilde\Psi(I,\widehat W,\lambda^{p}\hat\eta,0,\ldots,0)$ is also stable for all $\lambda\in\C^*.$ 
 Hence, by the continuity of $\widetilde\Psi$ and since stability is an open condition (cf.\ Remark \ref{remark stable open in polystable}), there is a neighborhood $U$ of $(0,\ldots,0)\in\bigoplus\limits_{j=1}^{p-1}H^0(K^{2j})$ such that $\widetilde\Psi(I,\widehat W,\lambda^p\hat\eta,\lambda^2 q_2,\lambda^4q_4,\ldots,\lambda^{2p-2} q_{2p-2})$ is stable for $(\lambda^2q_2,\ldots,\lambda^{2p-2}q_{2p-2})\in U$ i.e.\ for small $\lambda$.
 From \eqref{EQ scaling gauge action}, $(V,W,\lambda\eta)$ is stable, and thus, $(V,W,\eta)$ is also stable.
 This argument is reversible, so $(V,W,\eta)$ is stable if and only if $(I,\widehat W,\hat\eta)$ is stable. 
 
Assume now that $(I,\widehat W,\hat\eta)$ is strictly polystable. By Proposition \ref{Prop: strictly polystable SOpq}, there is $q'$ satisfying $p-1\leq q'<q$, such that 
 \[(I,\widehat W,\hat\eta)=\left(\widehat W'\oplus \widehat W'', \big(\hat\eta'\ \ 0\big)\right),\]
 where $(I,\widehat W',\hat\eta')$ is a stable $K^p$-twisted $\rO(1,q'-p+1)$-Higgs bundle and $\widehat W''$ is a polystable orthogonal bundle of rank $q-q'$. In this case, we have 
 \[\widetilde\Psi(I,\widehat W,\hat\eta,q_2,\ldots,q_{2p-2})=\left(V, \widehat W'\oplus \widehat W'',\big(\hat\eta' \ \ 0\big)\right)\]
 where
\begin{equation}\label{eq:Psi-moduli - -q'-p+1}
(V,W',\hat\eta')=\widetilde\Psi(I,\widehat W',\hat\eta',q_2,\ldots,q_{2p-2}), 
\end{equation}
and the map $\widetilde\Psi$ in \eqref{eq:Psi-moduli - -q'-p+1} is defined as in \eqref{EQ tildePsi-domain} and \eqref{EQ tildePsi}, but with $q$ replaced by $q'$.
 By the above argument, $\widetilde\Psi(I,\widehat W',\hat \eta',q_2,\ldots,q_{2p-2})$ is a stable $\rO(p,q')$-Higgs bundle. Since $\widehat W''$ is a polystable orthogonal bundle, we conclude that $\widetilde\Psi(I,\widehat W,\hat\eta,q_2,\ldots,q_{2p-2})$ is a strictly polystable $\SO(p,q)$-Higgs bundle. Again, the argument is reversible, hence the converse also holds.
\end{proof}

The next lemma shows that $\widetilde\Psi$ both respects isomorphism classes of the corresponding objects and is injective on such classes.

\begin{lemma}\label{Lemma Gauge group fixing im Psi}
Two $\SO(p,q)$-Higgs bundles $\widetilde\Psi(I,\widehat W,\hat\eta,q_2,\ldots,q_{2p-2})$ and $\widetilde\Psi(I',\widehat W',\hat\eta',q_2',\ldots,q_{2p-2}')$ are in the same $\rS(\rO(p,\C)\times\rO(q,\C))$-gauge orbit if and only if $(I,\widehat W,\hat\eta)$ and $(I',\widehat W',\hat\eta')$ are in the same $\rS(\rO(1,\C)\times\rO(q-p+1,\C))$-gauge orbit and $q_{2j}=q_{2j}'$ for all $1\leq j\leq p-1$. Furthermore, each $\rS(\rO(1,\C)\times\rO(q-p+1,\C))$-gauge transformation between $(I,\widehat W,\eta)$ and $(I',\widehat W',\hat\eta')$ uniquely determines an $\rS(\rO(p,\C)\times\rO(q,\C))$-gauge transformation between the Higgs bundles $\widetilde\Psi(I,\widehat W,\hat\eta,q_2,\ldots,q_{2p-2})$ and $\widetilde\Psi(I',\widehat W',\hat\eta',q_2,\ldots,q_{2p-2}).$
\end{lemma}
\begin{proof}
Let $(I,\widehat W,\hat\eta)$ and $(I',\widehat W',\hat\eta')$ be two points in $\cH_{K^p}(\SO(1,q-p+1))$, and $(q_2,\dots,q_{2p-2})$ and $(q_2',\dots,q_{2p-2}')$ be two points in $\bigoplus\limits_{j=1}^{p-1}H^0(K^{2j})$. Denote the associated points in the image of the map $\widetilde \Psi$ from \eqref{EQ tildePsi} by 
\[(V,W,\eta)=\xymatrix{\widetilde\Psi(I,\widehat W,\hat\eta,q_2,\ldots,q_{2p-2})& \text{and}&(V',W',\eta')=\widetilde\Psi(I',\widehat W',\hat\eta',q_2',\ldots,q_{2p-2}'),}\]
and recall that $V=I\otimes \cK_{p-1}$ and $W=\widehat W\oplus I\otimes \cK_{p-2}.$ 

First suppose $(\det(g_{\widehat W}),g_{\widehat W})$ is an $\rS(\rO(1,\C)\times \rO(q-p+1,\C))$-gauge transformation with 
\[(\det(g_{\widehat W}),g_{\widehat W})\cdot(I,\widehat W,\hat\eta)=(I',\widehat W',\hat\eta').\]
A straightforward computation shows that the $\rS(\rO(p,\C)\times\rO(q,\C))$-gauge transformation 
\begin{equation}
    \label{eq the gauge equiv}(g_V,g_W)=\Big(\det(g_{\widehat W})\Id_V,\smtrx{g_{\widehat W}&0\\0&\det(g_{\widehat W})\Id_{\cK_{p-2}})}\Big)
\end{equation}
acts on $(V,W,\eta)$ as 
\[(g_V,g_W)\cdot(V,W,\eta)=\widetilde\Psi(I',\widehat W',\hat\eta',q_2,\ldots,q_{2p-2}).\]
Thus, if $(I,W,\eta)$ and $(I',W',\eta')$ are in the same $\rS(\rO(1,\C)\times \rO(q-p+1,\C))$-gauge orbit, then $\widetilde \Psi(I,W,\eta,q_2,\dots,q_{2p-2})$ and $\widetilde \Psi(I',W',\eta',q_2,\dots,q_{2p-2})$ are in the same $\rS(\rO(p,\C)\times \rO(q,\C))$-gauge orbit.

Now suppose $(V,W,\eta)$ and $(V',W',\eta')$ are in the same $\rS(\rO(p,\C)\times \rO(q,\C))$-gauge orbit. 
The action of $(g_V,g_W)$ on $(V,W,\eta)$ is given by
\[(g_V,g_W)\cdot(\dbar_V,\dbar_W,\eta)=(g_V^{-1}\dbar_V g_V\ ,\  g_W^{-1}\dbar_W g_W\ ,\ g_V^{-1}\eta g_W)~.\]
With respect to the decompositions $W= \widehat W\oplus I\otimes\cK_{p-2}$ and $W'= \widehat W'\oplus I'\otimes\cK_{p-2}$, write
    \[g_W=\smtrx{g_{\widehat W}&A\\
                B&g_{\cK_{p-2}}}
                \ \ \ \ \ \ \ \ \text{and} \ \ \ \ \ \ \ \eta=\Big(\eta_{\widehat W}\ \ \ \eta(q_2,\ldots,q_{2p-2})\Big)~.\]
The gauge transformation $(g_V,g_W)$ acts on the Higgs field by 
\[g_V^{-1}\eta g_W=g_V^{-1}\cdot\Big(\eta_{\widehat W}g_{\widehat W}+\eta(q_2,\dots,q_{2p-2})B\ \ \ \eta_{\widehat W}A+\eta(q_2,\ldots,q_{2p-2})g_{\cK_{p-2}}\Big)~,
\]
and hence
\begin{equation}\label{eq:gauge Higgsfield}
\Big(\eta_{\widehat W}'\ \ \ \eta(q_2',\ldots,q_{2p-2}')\Big)=g_V^{-1}\cdot\Big(\eta_{\widehat W}g_{\widehat W}+\eta(q_2,\dots,q_{2p-2})B\ \ \ \eta_{\widehat W}A+\eta(q_2,\ldots,q_{2p-2})g_{\cK_{p-2}}\Big).
\end{equation}

We now use the description of $\eta(q_2,\dots,q_{2p-2})$ from \eqref{eq choice of Hitchin section O(p-1,p)}.
Since $g_V^{-1}$ is invertible and holomorphic, its matrix representation in the decompositions $V=I\otimes \cK_{p-1}$ and $V'=I'\otimes \cK_{p-1}$ is upper triangular with non-zero diagonal entries. 
A straightforward computation, using the form of $\eta(q_2',\ldots,q_{2p-2}')$ and the fact that $g_V^{-1}\eta_{\widehat W}g_{\widehat W}$ has the form $\smtrx{*\\0\vspace{-.1cm}\\ \vdots\\0}$, shows that $B=0$. 
By orthogonality of $g_W$ we conclude also that $A=0$, $g_{\widehat W}$ is an $Q_{\widehat W}$-orthogonal gauge transformation and $g_{\cK_{p-2}}$ is a $Q_{\cK_{p-2}}$-orthogonal gauge transformation. 

We now have $\eta(q_{2}',\ldots,q_{2p-2}')=g_V^{-1}\eta(q_{2},\dots,q_{2p-2})g_{\cK_{p-2}}$. Since $(I\otimes \cK_{p-1},I\otimes\cK_{p-2},\eta(q_2,\dots,q_{2p-2}))$ and $(I'\otimes \cK_{p-1},I'\otimes \cK_{p-2},\eta(q_2',\dots,q_{2p-2}'))$ define gauge equivalent Higgs bundle in an $\rO(p,p-1)$-Hitchin component, we have $(q_2,\dots,q_{2p-2})=(q_2',\dots,q_{2p-2}').$ 
By Remark \ref{Remark auto of Hitchin component}, this implies 
\[(g_V,g_{\cK_{p-2}})=\pm(\Id_V,\Id_{\cK_{p-2}})~.\] 
Finally, the determinant of $g_{\widehat W}$ determines the above sign since $\det(-\Id_V)\det(-\Id_{\cK_{p-2}})=-1$ and
\[1=\det(g_V)\det(g_W)=\det(g_V)\det(g_{\cK_{p-2}})\det(g_{\widehat W}).\] Thus, the gauge transformation $g_{\widehat W}$ uniquely determines $g_{\cK_{p-2}}$ and $g_V$. This shows that $(g_V,g_W)$ is given by \eqref{eq the gauge equiv}, completing the proof.
\end{proof}

As a consequence of the two previous lemmas, we have the following proposition.
\begin{proposition}\label{prop def of map Psi}
    The map $\widetilde\Psi$
        from \eqref{EQ tildePsi} descends to a continuous map of moduli spaces
    \begin{equation}
        \label{EQ PsiOnModuli}
        \Psi:\cM_{K^p}(\SO(1,q-p+1))\times \displaystyle\bigoplus\limits_{j=1}^{p-1}H^0(K^{2j})\longrightarrow\cM(\SO(p,q)),
     \end{equation} 
     which is a homeomorphism onto its image.
\end{proposition}

\begin{remark}
From Remark \ref{rmk:dimension}, one can check that the dimension of $\cM_{K^p}(\SO(1,q-p+1))\times\bigoplus\limits_{j=1}^{p-1}H^0(K^{2j})$ is the expected dimension of $\cM(\SO(p,q)).$ In particular, the map $\Psi$ is open on the smooth locus. Since the spaces $\cM(\SO(p,q))$ and $\cM_{K^p}(\SO(1,q-p+1))$ are singular, we have to examine the local structures of each space to prove openness of $\Psi$ at singular points.  
\end{remark}

\subsection{Local structure of fixed points in the image of $\Psi$}\label{sec: Local structure of fixed points} 
We will now analyze the local structure of fixed points of the $\C^*$-action in $\cM(\SO(p,q))$ which lie in the image of the map $\Psi$. The following lemma follows immediately from Lemma \ref{lemma: fixed point SO(1,n)} and Proposition \ref{prop def of map Psi}.
\begin{lemma}\label{Lem: Fixed points in the image of Psi}  
An $\SO(p,q)$-Higgs bundle $(V,W,\eta)$ in the image of $\Psi$ is a fixed point of the $\C^*$-action if and only if $(V,W,\eta)=\Psi(I,\widehat W,\hat\eta,0,\ldots,0)$, where $(I,\widehat W,\hat\eta)$ is a fixed point of the $\C^*$-action in $\cM_{K^p}(\SO(1,q-p+1))$. In particular, such a fixed point is given by
\footnote{The notation from Lemma \ref{lemma: fixed point SO(1,n)} has changed slightly, $(W_{-1},\eta_{-1},W_0)$ is now represented by $(W_{-p},\eta_{-p},W_0')$.}
\[(I,\widehat W,\hat\eta)=\big(I,W_{-p}\oplus W_0'\oplus W_p,\big(\eta_{-p}\ \ \ 0\ \ \ 0\big)\big)~,\]
where $W_0'$ is a polystable orthogonal bundle of rank $q-p+1-2\rk(W_p)$ and $\det(W_0')=I$, $W_p$ is either zero or a negative degree vector bundle with no non-negative degree subbundles, $W_{-p}\cong W_{p}^*$ and $\eta_{-p}$ is non-zero if $W_{-p}$ is non-zero.
The associated $\SO(p,q)$-Higgs bundle will be represented by 
\begin{equation}\label{EQ: Higgs field at singularity}
    \xymatrix@C=1.5em@R=-.3em{W_{-p}\ar[r]^-{\eta_{-p}}&IK^{p-1}\ar[r]^-1&IK^{p-2}\ar[r]^-1&\cdots\ar[r]^-1&I\ar[r]^-1&\cdots\ar[r]^-1&IK^{2-p}\ar[r]^-1&IK^{1-p}\ar[r]^-{\eta_{-p}^*}&W_{p}\\&&&&\oplus&&&&\\&&&&W_0'&&&&}.
\end{equation}
 \end{lemma}

Let $(V,W,\eta)$ be a polystable $\SO(p,q)$-Higgs bundle in the image of $\Psi$ of the form \eqref{EQ: Higgs field at singularity}.  This will be fixed until the end of Section \ref{sec: Local structure of fixed points}. If $W_p$ is zero, some of the  considerations below simplify.

We will repeatedly use the following bundle decompositions of $V$ and $W$ from \eqref{EQ: Higgs field at singularity}:
\begin{equation}\label{eq:bundle-decomp-fixed}
\begin{gathered}
V= V_{1-p}\oplus V_{3-p}\oplus\cdots\oplus V_{p-3}\oplus V_{p-1},\\
W=W_{-p}\oplus W_{2-p}\oplus\cdots\oplus W_0\oplus\cdots\oplus W_{p-2}\oplus W_{p},\\
V_j=IK^{-j}\ \text{for all }j,\ \ \ W_j=IK^{-j}\ \text{if}\ 0<|j|<p,\ \ \ \text{ and }\ \ \ W_0=\begin{dcases}W_0' &\text{if } p\text{ odd}\\ I\oplus W_0'&\text{if } p\text{ even.} \end{dcases}
\end{gathered}
\end{equation}
In particular, even though $(V,W,\eta)$ is not assumed to be stable, we get a weight decomposition like \eqref{Eq fixedpoint deformation sheaf splitting} of the deformation complex \eqref{EQ ad eta definition} as
\begin{math}
  C^\bullet(V,W,\eta) = \bigoplus C^\bullet_k,
\end{math}
where
\begin{displaymath}
  C^\bullet_k : \xymatrix@R=0em{\fso_k(V)\oplus\fso_k(W)\ar[r]^-{\ad_\eta}&\Hom_{k+1}(W,V)\otimes K~~.}
\end{displaymath}

In terms of the above splittings, we have $\End(V)=\displaystyle\bigoplus\limits_{k=2-2p}^{2p-2}\End_k(V)$, where $\End_{2k+1}(V)=0$ and
\begin{equation}
\label{eq EndV2k}
\End_{2k}(V)=\begin{dcases}
        \bigoplus\limits_{j=0}^{p-1-k}\Hom(V_{1-p+2j},V_{1-p+2j+2k})&k\geq0\\
        \bigoplus\limits_{j=0}^{p-1+k}\Hom(V_{p-1-2j},V_{p-1-2j+2k})&k<0.
    \end{dcases}
\end{equation}
Similarly, $\End(W)=\displaystyle\bigoplus\limits_{k=-2p}^{2p}\End_k(W)$, where
\begin{equation}\label{eq EndW2k}
    \End_{2k}(W)=\begin{dcases}
        \End(W_0)\oplus\bigoplus\limits_{j=0}^p \End(W_{p-2j}) &k=0 \text{\ and $p$ odd}\\
        \bigoplus\limits_{j=0}^{p-k}\Hom(W_{-p+2j},W_{-p+2j+2k})&k>0\text{ or }k=0 \text{ and $p$ even}\\
        \bigoplus\limits_{j=0}^{p+k}\Hom(W_{p-2j},W_{p-2j+2k})&k<0
    \end{dcases}
\end{equation}
and
\begin{equation}
    \label{eq EndW2k+1}
    \End_{2k+1}(W)=\begin{dcases}
        \Hom(W_{-2k-1},W_{0})\oplus \Hom(W_0,W_{2k+1})&\text{$2k+1\leq p$ and $p$ odd}\\
        0&\text{otherwise}.
    \end{dcases}
\end{equation}
Finally, $\Hom(W,V)=\displaystyle\bigoplus\limits_{k=1-2p}^{2p-1}\Hom_k(W,V)$, where 
\begin{equation}
    \label{eq HomWV2k+1}
\Hom_{2k+1}(W,V)=\begin{dcases}
        \bigoplus\limits_{j=0}^{p-1-k}\Hom(W_{-p+2j},V_{1-p+2j+2k})&2k+1\geq0\\
        \bigoplus\limits_{j=0}^{p+k}\Hom(W_{p-2j},V_{p-2j+1+2k})&2k+1<0,
    \end{dcases}
\end{equation}
and
\begin{equation}
    \label{eq HomWV2k}
\Hom_{2k}(W,V)=\begin{dcases}
        \Hom(W_{0},V_{2k})&\text{$1-p\leq 2k\leq p-1$ and $p$ odd}\\
        0&\text{otherwise}.
    \end{dcases}
\end{equation}
Note that the Higgs field $\eta$ is a holomorphic section of $\Hom_1(W,V)\otimes K$.

The Lie algebra bundle $\fso(V)\oplus\fso(W)\subset\End(V)\oplus\End(W)$ with fiber $\fso(p,\C)\oplus\fso(q,\C)$ consists of $Q_V$ and $-Q_W$ skew symmetric endomorphisms of $V$ and $W$ respectively.
The decompositions \eqref{eq EndV2k}, \eqref{eq EndW2k} and \eqref{eq EndW2k+1} induce the following decomposition of $\fso(V)\oplus\fso(W)\subset\End(V)\oplus\End(V)$:
\[\fso(V)=\bigoplus\limits_{k=2-2p}^{2p-2}\fso_k(V)\hspace{1cm} \text{ and }\hspace{1cm} \fso(W)=\bigoplus\limits_{k=-2p}^{2p}\fso_k(W).\]
Here $\fso_{2k+1}(V)=0$ and, using \eqref{eq EndV2k},
\begin{equation}
    \label{eq fsoV2k}
\fso_{2k}(V)=\begin{dcases}\{(\alpha_{1-p},\alpha_{3-p},\ldots,\alpha_{p-1-2k})\in\End_{2k}(V)\ |\ \alpha_i=-\alpha^*_{-2k-i}\}& k\geq0\\
\{(\alpha_{p-1},\alpha_{p-3},\ldots,\alpha_{1-p-2k})\in\End_{2k}(V)\ |\ \alpha_i=-\alpha^*_{-2k-i}\}& k<0,
\end{dcases}
\end{equation}
where the index of each homomorphism corresponds to the index of its domain, i.e., \[\alpha_i:V_i\to V_{i+2k}.\]
For $\fso(W)$,  using \eqref{eq EndW2k} we have
\begin{equation}\label{eq fsoW2k}
\fso_{2k}(W)=\begin{dcases}
\{(\beta',\beta_p,\beta_{p-2},\ldots,\beta_{-p})\in\End_{0}(W) |\ \beta'=-(\beta')^*,\ \beta_i=-\beta^*_{-i}\} &k=0 \text{\ and $p$ odd}\\
\{(\beta_{-p},\beta_{2-p},\ldots,\beta_{p-2k})\in\End_{2k}(W)\ |\ \beta_i=-\beta^*_{-2k-i}\}& k>0\text{ or }k=0 \text{ and $p$ even}\\
\{(\beta_p,\beta_{p-2},\ldots,\beta_{-p+2k})\in\End_{2k}(W)\ |\ \beta_i=-\beta^*_{-2k-i}\}& k<0,
\end{dcases}
\end{equation}
where $\beta':W_0\to W_0$ 
and, as above, 
$\beta_i:W_i\to W_{i+2k}.$ 
For odd weights, using \eqref{eq EndW2k+1} we have
\begin{equation}
    \label{eq fsoW2k+1}
    \fso_{2k+1}(W)=\begin{dcases}
        \{(\beta_{-2k-1},-\beta_{-2k-1}^*)\in\Hom(W_{-2k-1},W_{0})\oplus \Hom(W_0,W_{2k+1})\}&\text{$2k+1\leq p$ and $p$ odd}\\
        0&\text{otherwise}.
    \end{dcases}
\end{equation}

Since $\eta\in H^0(\Hom_1(W,V)\otimes K)$, the map $\ad_\eta$ restricts to $\fso_k(V)\oplus\fso_k(W)\to \Hom_{k+1}(W,V)\otimes K$, yielding the subcomplex $C^\bullet_k$ of $C^\bullet$ of weight $k$ as in \eqref{Eq fixedpoint deformation sheaf splitting}
\[C^\bullet_k=C^\bullet(V,W,\eta)_k:\fso_k(V)\oplus\fso_k(W)\xrightarrow{\ad_\eta}\Hom_{k+1}(W,V)\otimes K,\hspace{1.25cm} (\alpha,\beta)\mapsto \eta \circ \beta-\alpha\circ\eta.\] 
This gives rise to a splitting of the hypercohomology sequence associated to $C^\bullet$:
\begin{equation}
    \label{eq SOpq graded hypercohomoly sequence}
        \xymatrix@R=1em@C=1em{0\ar[r]&\HH^0(C^\bullet_{k})\ar[r]&H^0(\fso_k(V)\oplus\fso_k(W))\ar[r]^-{\ad_\eta}&H^0(\Hom_{k+1}(W,V)\otimes K)\ar[r]&\HH^1(C^\bullet_k)\ar@{->}`r/3pt [d] `/10pt[l] `^d[llll] `^r/3pt[d][dlll]\\&H^1(\fso(V)_{k+1}\oplus \fso_{k+1}(W))\ar[r]^-{\ad_\eta}&H^1(\Hom_{k+1}(W,V)\otimes K)\ar[r]&\HH^2(C^\bullet_k)\ar[r]&0.}  
\end{equation}

For all $k,$ we will compute $\HH^1(C^\bullet_k)$ and show $\HH^2(C^\bullet_k)$ vanishes in a series of lemmas.
Using \eqref{eq:bundle-decomp-fixed} and the decomposition of $\Hom_1(W,V)\otimes K$ from \eqref{eq HomWV2k+1}, we write
\begin{equation}\label{eq eta decomposition}
\eta=(\eta_{-p},\eta_{2-p},\ldots,\eta_{p-2})\in\bigoplus_{j=0}^{p-1}H^0(\Hom(W_{-p+2j},V_{1-p+2j})\otimes K),
\end{equation}
where 
\begin{equation}\label{eq eta decomposition2}
\begin{dcases}\eta_{-p}:W_{-p}\to V_{1-p}\otimes K &\text{ is defined in Lemma \ref{Lem: Fixed points in the image of Psi}},\\
\eta_0=\smtrx{1&0}:I\oplus W_0'\to V_1\otimes K &\text{ if }p\text{ even},\\
\eta_i=1:W_i\to V_{i+1}\otimes K &\text{ otherwise}.
\end{dcases}
\end{equation}
\begin{lemma}\label{lemma k>0 H1 vanishing H2 not p 2p}
     The map $\ad_\eta:\fso_k(V)\oplus\fso_k(W)\to\Hom_{k+1}(W,V)\otimes K$ is an isomorphism for each positive weight $k\notin\{p,2p\}$. In particular, 
    \[\xymatrix{\HH^0(C^\bullet_k)=0,&\HH^1(C^\bullet_k)=0&\text{and}&\HH^2(C^\bullet_k)=0}.\]
\end{lemma}
\begin{proof}
    We start by considering the case $C^\bullet_{2k+1}$ with $0<2k+1$ and $2k+1\neq p$. If $p$ is even or $p<2k+1,$ the result is immediate since $\fso_{2k+1}(V)$, $\fso_{2k+1}(W)$ and $\Hom_{2k+2}(W,V)\otimes K$ are all zero by \eqref{eq EndW2k+1} and \eqref{eq HomWV2k}.
    For $p$ odd and $2k+1< p$, we have $\fso_{2k+1}(V)=0$, $\fso_{2k+1}(W)=\{(\beta_{-2k-1},-\beta_{-2k-1}^*)\in\Hom(W_{-2k-1},W_0)\oplus \Hom(W_0,W_{2k+1})\}$ and $\Hom_{2k+2}(W,V)\otimes K=\Hom(W_0,V_{2k+2})\otimes K$. Using \eqref{eq eta decomposition}, the map $\ad_\eta$ is the isomorphism sending $\beta_{-2k-1}$ to the composition of $-\beta_{-2k-1}^*$ with $1=\eta_{2k+1}$: 
    \[\xymatrix@R=1em{W_0\ar[dr]_{-\beta_{-2k-1}^*}\ar[rr]&&V_{2k+2}\otimes K\\&\ \ W_{2k+1}\ar[ur]_{1}&}\]

    Now consider the case $C^\bullet_{2k}$ with $0<2k$ and $2k\notin\{p,2p\}$. We first show $\fso_{2k}(V)\oplus\fso_{2k}(W)$ and $\Hom_{2k+1}(W,V)\otimes K$ are isomorphic. 
    Using \eqref{eq:bundle-decomp-fixed} and  \eqref{eq HomWV2k+1}, we have 
    \begin{equation}
    \label{eq hom2k+1 Kpowers}\Hom_{2k+1}(W,V)\otimes K\cong \begin{dcases}
        \Hom(W_{-p},IK^{p-2k})\oplus \underbrace{K^{-2k}\oplus\cdots\oplus K^{-2k}}_{p-k-1\text{ times}} &\hspace{-2em}{\text{$2k> p$ or $p$ odd,}}\\
        \Hom(W_{-p},IK^{p-2k})\oplus \Hom(W_0',IK^{-2k})\oplus \underbrace{K^{-2k}\oplus\cdots\oplus K^{-2k}}_{p-k-1\text{ times}} &\text{otherwise}.
    \end{dcases}
    \end{equation}
    On the other hand, by \eqref{eq fsoV2k} and since the weight is positive, we have 
    \begin{equation}
        \label{eq powers of K soV}\fso_{2k}(V)\cong\bigoplus\limits_{j=0}^{\floor{\frac{p-k}{2}}-1} \Hom(V_{2j-p+1},V_{2j-p+1+2k})\cong \underbrace{K^{-2k}\oplus\cdots\oplus K^{-2k}}_{\floor{\frac{p-k}{2}}\text{ times}}~.
    \end{equation} 
    Similarly, by \eqref{eq fsoW2k}, $\fso_{2k}(W)\cong\bigoplus\limits_{j=0}^{\floor{\frac{p-k-1}{2}}} \Hom(W_{2j-p},W_{2j-p+2k})$, and thus,
\begin{equation}
    \label{eq powers of K soW}
    \fso_{2k}(W)\cong \begin{dcases}
        \Hom(W_{-p},IK^{p-2k})\oplus\underbrace{K^{-2k}\oplus\cdots\oplus K^{-2k}}_{\floor{\frac{p-k-1}{2}} \text{ times}} & \text{$2k> p$ or $p$ odd}\\
         \Hom(W_{-p}, IK^{p-2k})\oplus \Hom(W_0',IK^{-2k})\oplus\underbrace{K^{-2k}\oplus\cdots\oplus K^{-2k}}_{\floor{\frac{p-k-1}{2}} \text{ times}}&\text{otherwise}.
    \end{dcases}
\end{equation}
From \eqref{eq hom2k+1 Kpowers}, \eqref{eq powers of K soV} and \eqref{eq powers of K soW}, we see that $\fso_{2k}(V)\oplus \fso_{2k}(W)$ is isomorphic to $\Hom_{2k+1}(W,V)\otimes K$.

Now we will show  
\[C_{2k}^\bullet:\fso_{2k}(V)\oplus \fso_{2k}(W)\xrightarrow{\ad_\eta}\Hom_{2k+1}(W,V)\otimes K,\hspace{1cm}\ad_\eta(\alpha,\beta)=\eta\circ\beta-\alpha\circ\eta\] is an isomorphism.
Using the notations of \eqref{eq fsoV2k}, \eqref{eq fsoW2k} (for positive weight) and \eqref{eq eta decomposition}, if 
\[\xymatrix@=.5em{\alpha=(\alpha_{1-p},\alpha_{3-p},\ldots,\alpha_{p-1-2k}),& \beta=(\beta_{-p},\beta_{2-p},\ldots,\beta_{p-2k})&\text{and}&\eta=(\eta_{-p},\eta_{2-p},\ldots,\eta_{p-2})~,}\]
then
\[\ad_\eta(\alpha,\beta)=(\eta_{-p+2k}\beta_{-p}-\alpha_{1-p}\eta_{-p},\eta_{2-p+2k}\beta_{2-p}-\alpha_{3-p}\eta_{2-p},\ldots,\eta_{p-2}\beta_{p-2-2k}-\alpha_{p-1-2k}\eta_{p-2-2k}).\]

First assume $p-k$ is even. In this case we have
\begin{equation*}
\xymatrix@=.5em{\alpha=(\alpha_{1-p},\ldots,\alpha_{-k-1},-\alpha_{-k-1}^*,\ldots,-\alpha_{1-p}^*) &\text{and}&
\beta=(\beta_{-p},\ldots,\beta_{-k-2},0,-\beta_{-k-2}^*,\ldots,-\beta_{-p}^*)~.}
\end{equation*}
For $p$ odd or $2k>p$, we have $\eta_i=1$ for all $i\neq -p$ by \eqref{eq eta decomposition2}. Hence $\ad_\eta(\alpha,\beta)$ is given by 
\begin{equation}\label{eq ad eta 2kgeq0}
(\beta_{-p}-\alpha_{1-p}\eta_{-p},\beta_{2-p}-\alpha_{3-p},\ldots,\beta_{-k-2}-\alpha_{-k-1},\alpha_{-k-1}^*,-\beta_{-k-2}^*+\alpha_{-k-3}^*,\ldots,-\beta_{2-p}^*+\alpha_{1-p}^*).
\end{equation}
This vanishes if and only if $\alpha$ and $\beta$ are both identically zero, so $\ad_\eta$ is an isomorphism.
For $p$ even and $2k\leq p$, the only difference is that $W_0=I\oplus W_0'$. Therefore, if we write 
\[\beta_0=\Big(\beta_0^I\ \ \ \beta_0'\Big):I\oplus W_0'\to W_{2k},\] then the terms $W_0\to V_{2k+1}\otimes K$ and $W_{-2k}\to V_1\otimes K$ of $\ad_\eta$ are given by 
\begin{equation}
    \label{Eq p even p/2 map}
    \Big(\beta_0^I-\alpha_1\ \ \ \beta_0'\Big):I\oplus W_0'\to V_{2k+1}\otimes K\ \ \ \ \ \text{and}\ \ \ \ \ \ \ -\beta_0^{I\ast}+\alpha_1^*:W_{-2k}\to V_1\otimes K.
\end{equation}
Again, $\ad_\eta$ vanishes if and only if $\alpha$ and $\beta$ both vanish, and is therefore an isomorphism.

Now suppose $p-k$ is odd. In this case, $\eqref{eq fsoV2k}$ and \eqref{eq fsoW2k} imply that 
\[\xymatrix@=.5em{\alpha=(\alpha_{1-p},\ldots,\alpha_{-k-2},0,-\alpha_{-k-2}^*,\ldots,-\alpha_{1-p}^*)&\text{and}&
\beta&=(\beta_{-p},\ldots,\beta_{-k-1},-\beta_{-k-1}^*,\ldots,-\beta_{-p}^*).}\]
For $p$ odd or $2k>p$, $ad_\eta(\alpha,\beta)$ is given by
\[(\beta_{-p}-\alpha_{1-p}\eta_{-p},\beta_{2-p}-\alpha_{3-p},\ldots,\beta_{-k-3}-\alpha_{-k-2},\beta_{-k-1},-\beta_{-k-1}^*+\alpha_{-k-2}^*,\ldots,-\beta_{2-p}^*+\alpha_{1-p}^*).\]
Since this vanishes if and only if $\alpha$ and $\beta$ both vanish, $\ad_\eta$ is an isomorphism. The case of $p$ even and $2k\leq p$ follows from a similar calculation as the one done above.

Since $\fso_k(V)\oplus \fso_{k}(W)\xrightarrow{\ad_\eta} \Hom_{k+1}(W,V)\otimes K$ is an isomorphism for all positive weights $k$ different than $p$ and $2p$, we conclude that the hypercohomology groups $\HH^*(C^\bullet_k)$ all vanish for such $k$.
\end{proof}

Next we consider the subcomplexes of weight $p$ and $2p$.

\begin{lemma}\label{lemma p, 2p H1 vanishing H2}
The hypercohomology groups $\HH^*(C^\bullet_p)$ and $\HH^*(C^\bullet_{2p})$ are given by 
\[\xymatrix@R=0em{\HH^0(C^\bullet_p)=0,&\HH^1(C^\bullet_p)\cong H^1(\Hom(W_{-p},W_0'))&\text{and}&\HH^2(C^\bullet_p)=0,\\\HH^0(C^\bullet_{2p})=0,&\HH^1(C^\bullet_{2p})\cong H^1(\fso_{2p}(W))&\text{and}&\HH^2(C^\bullet_{2p})=0,}\]
where $\fso_{2p}(W)=\{\beta\in \Hom(W_{-p},W_p)| \beta+\beta^*=0\}.$
\end{lemma}
\begin{proof}
First note that $\fso_{2p}(V)=0$, $\fso_{2p}(W)=\{\beta\in \Hom(W_{-p},W_p)| \beta+\beta^*=0\}$ and $\Hom_{2p+1}(W,V)=0$, hence
\[\xymatrix{\HH^0(C^\bullet_{2p})\cong H^0(\fso_{2p}(W)), & \HH^1(C^\bullet_{2p})\cong H^1(\fso_{2p}(W))&\text{and}&\HH^2(C^\bullet_{2p})=0~.}\]
If $p$ is odd, then $W_0=W_0',$ $\fso_{p}(W)\cong\Hom(W_{-p},W_0'),$ $\fso_p(V)=0$ and $\Hom_{p+1}(W,V)=0$,  thus 
    \[\xymatrix{\HH^0(C^\bullet_p)\cong H^0(\Hom(W_{-p},W_0')), & \HH^1(C^\bullet_p)\cong H^1(\Hom(W_{-p},W_0'))&\text{and}&\HH^2(C^\bullet_p)=0.}\]
    Moreover, $H^0(\fso_{2p}(W))$ and $H^0(\Hom(W_{-p},W_0'))$ were shown to vanish in the proof of Lemma \ref{Lemma: H1 of SO(1,n) fixed point}, completing the proof for the case $2p$ and when $p$ is odd.
    
Now suppose $p$ is even, then $W_0=I\oplus W_0'$ and, from \eqref{eq EndV2k}, \eqref{eq EndW2k} and \eqref{eq HomWV2k+1}, we have
\[\fso_{p}(V)\cong\underbrace{K^{-p}\oplus\cdots\oplus K^{-p}}_{\floor{\frac{p}{4}} \text{ times}}~,\]
\[\fso_p(W)\cong\Hom(W_{-p},I)\oplus\Hom(W_{-p},W_0')\oplus\underbrace{K^{-p}\oplus\cdots\oplus K^{-p}}_{\floor{\frac{p-2}{4}} \text{ times}}\oplus \Hom(W_0',K^{-p})~\]
and
\[\Hom_{p+1}(W,V)\otimes K\cong \Hom(W_{-p},I)\oplus \underbrace{K^{-p}\oplus\cdots\oplus K^{-p}}_{\frac{p}{2}-1\text{ times}}~.\]
Thus, $\fso_p(V)\oplus\fso_p(W)\cong\Hom(W_{-p},W_0')\oplus \Hom_{p+1}(W,V)\otimes K.$

If $\frac{p}{2}$ is even and $(\alpha,\beta)\in\fso_p(V)\oplus\fso_p(W)$, then
\[\xymatrix@=.5em{\alpha=(\alpha_{1-p},\ldots,\alpha_{-\frac{p}{2}-1},-\alpha_{-\frac{p}{2}-1}^*,\ldots,-\alpha_{1-p}^*)&\text{and}&
\beta=(\beta_{-p},\ldots,\beta_{-\frac{p}{2}-2},0,-\beta_{-\frac{p}{2}-2}^*,\ldots,-\beta_{-p}^*)}.\]
Using the decomposition of $\eta$ from \eqref{eq eta decomposition} and \eqref{eq eta decomposition2}, we see that $\ad(\alpha,\beta)$ is given by
\[(\eta_0\beta_{-p}-\alpha_{1-p}\eta_{-p},\beta_{2-p}-\alpha_{3-p},\ldots,\beta_{-\frac{p}{2}-2}-\alpha_{-\frac{p}{2}-1},\alpha_{-\frac{p}{2}-1}^*,-\beta_{-\frac{p}{2}-2}^*+\alpha_{-\frac{p}{2}-3}^*,\ldots,-\beta_{2-p}^*+\alpha_{1-p}^*).\]
If we write $\beta_{-p}=\smtrx{\beta_{-p}^I\\ \beta_{-p}'}:W_{-p}\to I\oplus W_0'$, then $\eta_0\beta_{-p}=\smtrx{1&0}\smtrx{\beta_{-p}^I\\ \beta_{-p}'}=\beta_{-p}^I$. Hence $\Hom(W_{-p},W_0')$ is in the kernel of $\ad_\eta$ and $\eta_0\beta_{-p}-\alpha_{1-p}\eta_{-p}=\beta_{-p}^I-\alpha_{1-p}\eta_{-p}$. 
We conclude that the map induced by $\ad_\eta$ on $(\fso_p(V)\oplus\fso_p(W))/\Hom(W_{-p},W_0')\to\Hom_{p+1}(W,V)\otimes K$ is given by 
\[\ad_\eta:\Hom(W_{-p},W_0')\oplus(\fso_p(V)\oplus\fso_p(W))/\Hom(W_{-p},W_0')\xrightarrow{\smtrx{0&\delta}} \Hom_{p+1}(W,V)\otimes K\]
    with $\delta$ an isomorphism. In particular, this implies that 
    \[\xymatrix{\HH^0(C^\bullet_{p})\cong H^0(\Hom(W_{-p},W_0')), & \HH^1(C^\bullet_{p})\cong H^1(\Hom(W_{-p},W_0'))&\text{and}&\HH^2(C^\bullet_{p})=0.}\]
    Moreover, $H^0(\Hom(W_{-p},W_0'))$ was shown to vanish in the proof of Lemma \ref{Lemma: H1 of SO(1,n) fixed point}.
The proof for $\frac{p}{2}$ odd follows from similar arguments.  
\end{proof}

Now we consider negative odd weights different from $-p$.

\begin{lemma}\label{lemma 2k+1<0 H1 vanishing H2 not -p -2p}
     The map $\ad_\eta:\fso_{2k+1}(V)\oplus\fso_{2k+1}(W)\to\Hom_{2k+2}(W,V)\otimes K$ is an isomorphism for $2k+1<0$ and $2k+1\neq-p$. In particular,
    \[\xymatrix{\HH^0(C^\bullet_{2k+1})=0,&\HH^1(C^\bullet_{2k+1})=0&\text{and}&\HH^2(C^\bullet_{2k+1})=0.}\]
\end{lemma}
\begin{proof}
First, note that $\fso_{2k+1}(V)=0$. Also, if $p$ is even or $2k+1<-p$, then  $\fso_{2k+1}(W)=0$ and $\Hom_{2k+2}(W,V)=0$. For $p$ odd and $2k+1>-p$, 
\[\fso_{2k+1}(W)=\{(\beta_{-2k-1},-\beta_{-2k-1}^*)\in\Hom(W_{-2k-1},W_0)\oplus\Hom(W_0,W_{2k+1})\}\] and $\Hom_{2k+2}(W,V)\otimes K=\Hom(W_0,V_{2k+2})\otimes K$. 
Moreover, $\ad_\eta:\fso_{2k+1}(W)\to\Hom_{2k+2}(W,V)\otimes K$ is given by
\[\xymatrix@R=.5em{W_0\ar[rr]\ar[dr]_-{-\beta_{-2k-1}^*}&&V_{2k+2}\otimes K\\&W_{2k+1}\ar[ur]_1&}\]
which is an isomorphism.
\end{proof}

Next we deal with negative even weights different from $-p$ and $-2p$.

\begin{lemma}\label{lemma 2k<0 H1 vanishing H2 not -p -2p}
    For $2k<0$ and $2k\notin\{-p,-2p\}$, $\Hom_{2k+1}(W,V)\otimes K\cong \fso_{2k}(W)\oplus\fso_{2k}(V)\oplus K^{-2k}$ and $\ad_\eta$ decomposes as \[\ad_\eta=\smtrx{a\\b}:\fso_{2k}(W)\oplus\fso_{2k}(V)\to \big(\fso_{2k}(W)\oplus\fso_{2k}(V)\big)\oplus K^{-2k},\] where $a$ is an isomorphism.
In particular, 
    \[\xymatrix{\HH^0(C^\bullet_{2k})=0,&\HH^1(C^\bullet_{2k})\cong H^0(K^{-2k})&\text{and}&\HH^2(C^\bullet_{2k})=0.}\]
\end{lemma}
\begin{proof}
    Using \eqref{eq HomWV2k+1}, we have that
\[\Hom_{2k+1}(W,V)\otimes K\cong\begin{dcases}
\Hom(W_p, IK^{-p-2k})\oplus  \underbrace{K^{-2k}\oplus\cdots\oplus K^{-2k}}_{p+k \text{\ times}}&\hspace{-4em}{\text{if $p$ odd or }2k<-p}\\
\Hom(W_p, IK^{-p-2k})\oplus \Hom(W_0',IK^{-2k})\oplus \underbrace{K^{-2k}\oplus\cdots\oplus K^{-2k}}_{p+k \text{\ times}}&\text{otherwise~.}
\end{dcases}
\]
    If $p+k$ is even, then by \eqref{eq fsoV2k} and \eqref{eq fsoW2k} we have  
    \[\fso_{2k}(V)\cong\Big\{(\alpha_{p-1},\ldots,\alpha_{-k+1},-\alpha_{-k+1}^*,\ldots,-\alpha_{p-1}^*)\in\bigoplus\limits_{j=0}^{p-1+k}\Hom(V_{p-1-2j},V_{p-1-2j+2k})\Big\}\]
    \[\fso_{2k}(W)\cong\Big\{(\beta_p,\ldots,\beta_{-k+2},0,-\beta_{-k+2}^*,\ldots,-\beta_p^*)\in\bigoplus\limits_{j=0}^{p+k}\Hom(W_{p-2j},W_{p-2j+2k})\Big\}.\]
    Thus, 
\[\fso_{2k}(V)\cong \underbrace{K^{-2k}\oplus\cdots\oplus K^{-2k}}_{\frac{p+k}{2}\text{ times}}\]
and
\[\fso_{2k}(W)\cong \begin{dcases}
        \Hom(W_p,IK^{-p-2k})\oplus \underbrace{K^{-2k}\oplus\cdots\oplus K^{-2k}}_{\frac{p+k}{2}-1\text{ times}}&\text{if $p$ odd or }2k<-p\\
\Hom(W_p,IK^{-p-2k})\oplus\Hom(W_0',IK^{-2k})\oplus \underbrace{K^{-2k}\oplus\cdots\oplus K^{-2k}}_{\frac{p+k}{2}-1\text{ times}}&\text{otherwise~.}        
    \end{dcases}
\]
Hence we conclude that $\Hom_{2k+1}(W,V)\otimes K\cong\fso_{2k}(W)\oplus\fso_{2k}(V)\oplus K^{-2k}$. By a similar argument, the  conclusion also holds for the case $p+k$ odd.

For the form  of $\ad_\eta$ in this splitting, first assume $p$ is odd or $2k<-p.$ If $p+k$ is even, then the map $\ad_\eta:\fso_{2k}(V)\oplus\fso_{2k}(W)\to \Hom_{2k+1}(W,V)\otimes K$ is given by 
\begin{equation}\label{eq:ad_eta 2k k<0}
\begin{split}
\ad_\eta(\alpha,\beta)&=(\beta_p,\beta_{p-2}-\alpha_{p-1},\ldots,\beta_{-k+2}-\alpha_{-k+3},-\alpha_{-k+1},-\beta_{-k+2}^*+\alpha_{-k+1}^*,\ldots,\alpha_{p-1}^*-\eta_{-p}\beta_p^*).
\end{split}
\end{equation}
Consider the summand $K^{-2k}\cong\Hom(W_{-k},V_{k+1})\otimes K$ of $\Hom_{2k+1}(W,V)\otimes K$ and take the corresponding quotient $(\Hom_{2k+1}(W,V)\otimes K)/K^{-2k}$. Then $\Hom_{2k+1}(W,V)\otimes K=(\Hom_{2k+1}(W,V)\otimes K)/K^{-2k}\oplus K^{-2k}$ and, from \eqref{eq:ad_eta 2k k<0}, we conclude that $\ad_\eta$ can be written as 
\[\ad_\eta=\smtrx{a\\b}:\fso_{2k}(V)\oplus\fso_{2k}(W)\to \big(\Hom_{2k+1}(W,V)\otimes K)/K^{-2k}\oplus K^{-2k}\] where $a$ is an isomorphism.
If $p+k$ is odd, a similar conclusion holds.

If $p$ is even and $-p<2k$, the only difference is that we have the following decompositions 
\[\beta_0=\Big(\beta_0^I\ \ \ \beta_0'\Big):I\oplus W_0'\to W_{2k}\ \ \ \ \  \text{and}\ \ \ \ \ \ \beta_0^*=\smtrx{(\beta_0^I)^*\\(\beta_0')^*}:W_{-2k}\to I\oplus W_0'.\] 
With these decompositions, the terms of $\ad_\eta$ which involve $\beta_0$ and $\beta_0^*$ are given by 
\begin{equation}\label{eq p even case of map}\xymatrix@=1em{&V_1\otimes K\ar[dr]^-{\alpha_1\otimes\Id_K}&\\ I\oplus W_0'\ar[ur]^-{\smtrx{1& 0}}\ar[dr]_-{\beta_0}&&V_{2k+1}\otimes K\\&W_{2k}\ar[ur]_-1&}\ \ \ \ \ \ \ \ \xymatrix@=1em{\\ \text{and}\\}\ \ \ \ \ \ \ \ \ \ \ \xymatrix@=1em{&V_{-2k+1}\ar[dr]^-{\ \ \ -\alpha_{-1}^*\otimes\Id_K}\otimes K&\\W_{-2k}\ar[ur]^-1\ar[dr]_-{-\beta_0^*}&&V_1\otimes K.\\& I\oplus W_0'\ar[ur]_-{\smtrx{1&0}}&}\end{equation}
The map $I\oplus W_0'\to V_{2k+1}\otimes K$ is given by $\smtrx{\beta_0^I-\alpha_1\,&\,\beta_0'}$ and the map $W_{-2k}\to W_1\otimes K$ is given by $-(\beta_0^I)^*+\alpha_{-1}$. In particular, we have $\ad_\eta=\smtrx{a\\b}:\fso_{2k}(W)\oplus\fso_{2k}(V)\to \big(\fso_{2k}(W)\oplus\fso_{2k}(V)\big)\oplus K^{-2k}$ with $a$ an isomorphism.
 
This implies that in the long exact sequence \eqref{eq SOpq graded hypercohomoly sequence}, for $2k<0$ and $2k\notin\{-p,-2p\},$ we have 
\[\xymatrix{\HH^0(C^\bullet_{2k})=0,&\HH^1(C^\bullet_{2k})\cong H^0(K^{-2k})& \text{and} &\HH^2(C^\bullet_{2k})=H^1(K^{-2k})=0,}\]
completing the proof. 
\end{proof}

The next lemma deals with $\HH^*(C^\bullet_{-p})$ and $\HH^*(C^\bullet_{-2p})$.

\begin{lemma}
\label{lemma -p  H1 vanishing H2}
In weight $-2p$ we have $\HH^0(C^\bullet_{-2p})=0,$ $\HH^2(C^\bullet_{2k})=0$ and
\begin{equation}
    \label{eq H1-2p}\xymatrix@=1.5em{0\ar[r]& H^0(\fso_{-2p}(W))\ar[r]^-{\eta_{-p}}&H^0(\Hom(W_p,K^p))\ar[r]&\HH^1(C^\bullet_{-2p})\ar[r]&H^1(\fso_{-2p}(W))\ar[r]&0,}
\end{equation}
where $\fso_{-2p}(W)=\{\beta\in\Hom(W_p,W_{-p})|\beta+\beta^*=0\}.$ 
For $p$ odd, we have 
\[\HH^0(C^\bullet_{-p})=0,\hspace{1cm}\HH^1(C^\bullet_{-p})\cong \HH^1_{-p}\hspace{1cm}\text{and}\hspace{1cm}\HH^2(C^\bullet_{-p})=0,\] 
where
\begin{equation}\label{eq H1-p}
\xymatrix@=1.5em{0\ar[r]& H^0(\Hom(W_p,W_0'))\ar[r]^-{\eta_{-p}}&H^0(\Hom(W_0',K^p))\ar[r]&\HH^1_{-p}\ar[r]&H^1(\Hom(W_p,W_0'))\ar[r]&0.}
\end{equation}
For $p$ even, 
\[\xymatrix@=.5em{\fso_{-p}(V)\oplus \fso_{-p}(W)\cong\Hom(W_p,W_0')\oplus A&\text{and}&\Hom_{1-p}(W,V)\otimes K= K^p\oplus \Hom(W_0', K^p)\oplus A},\] and with respect to this splitting $\ad_\eta=\smtrx{0&b\\\eta_{-p}&0\\0&a}$, where $a:A\to A$ is an isomorphism. In particular, for $p$ even, 
\[\HH^0(C^\bullet_{-p})=0,\hspace{1cm}\HH^1(C^\bullet_{-p})\cong H^0(K^p)\oplus \HH^1_{-p}\hspace{1cm}\text{and}\hspace{1cm}\HH^2(C^\bullet_{-p})=0.\]
\end{lemma}
\begin{proof}
For weight $-2p$ we have $\fso_{-2p}(V)=0,$ $\fso_{-2p}(W)=\{\beta\in\Hom(W_p,W_{-p})|\beta+\beta^*=0\}$ and $\Hom_{-2p+1}(W,V)\otimes K\cong\Hom(W_{p},IK^p).$
The map $\ad_\eta:\fso_{-2p}(W)\to \Hom(W_{p},K^p)$ is given by $\ad_\eta(\beta)=\eta_{-p}\beta$. 
The result now follows from Lemmas \ref{lemma vanishing H2 SO1n} and  \ref{Lemma: H1 of SO(1,n) fixed point}. 

If $p$ is odd, then by \eqref{eq fsoW2k+1} and \eqref{eq HomWV2k} we have 
        \[\xymatrix@=1em{\fso_{-p}(V)=0,&\fso_{-p}(W)\cong\Hom(W_p,W_0')&\text{and}&\Hom_{1-p}(W,V)\otimes K=\Hom(W_0',IK^p).}\]
The map $\ad_{\eta}:\Hom(W_p,W_0')\to \Hom (W_0',IK^p)$ is given by $\ad_\eta(\beta_p)=-\eta_{-p}\beta_p^*$. Again, the result now follows from Lemmas \ref{lemma vanishing H2 SO1n} and  \ref{Lemma: H1 of SO(1,n) fixed point}.

If $p$ is even, then
\begin{equation*}
\begin{split}
\fso_{-p}(V)&\cong\underbrace{K^p\oplus\cdots\oplus K^p}_{\floor{\frac{p}{4}} \text{ times}},\\
\fso_{-p}(W)&\cong\Hom(W_p,I)\oplus\Hom(W_p,W_0')\oplus \underbrace{K^p\oplus\cdots\oplus K^p}_{\floor{\frac{p-2}{4}}\text{ times}},\\
\Hom_{1-p}(W,V)\otimes K&\cong\Hom(W_p, I)\oplus \Hom(W_0',IK^p)\oplus \underbrace{K^p\oplus\cdots\oplus K^p}_{\frac{p}{2}\text{ times}}.
\end{split}
\end{equation*}
Setting $A=\Hom(W_p,I)\oplus\underbrace{K^p\oplus\cdots\oplus K^p}_{\frac{p}{2}-1\ \text{times}}$ we have $\fso_{-p}(V)\oplus\fso_{-p}(W)\cong \Hom(W_p,W_0')\oplus A$ and $\Hom_{-p}(W,V)\otimes K\cong  K^{p}\oplus \Hom(W_0',IK^p)\oplus A.$ 
The map $\ad_\eta$ is analogous to the one in the proof of Lemma \ref{lemma 2k<0 H1 vanishing H2 not -p -2p} 
 except that \eqref{eq p even case of map} is given by 
\[\xymatrix@=1em{&V_1\otimes K\ar[dr]^{\alpha_1\otimes\Id_K}&\\ I\oplus W_0'\ar[ur]^{\smtrx{1& 0}}\ar[dr]_-{-\beta_p^*}&&V_{-p+1}\otimes K\\&W_{-p}\ar[ur]_-{\eta_{-p}}&}\ \ \ \ \ \ \ \ \ \ \  \xymatrix@=1em{\\\text{and}\\}\ \ \ \ \ \ \ \ \ \ \ \ \xymatrix@=1em{&&\\W_p\ar[dr]_-{\beta_p}&&V_1\otimes K.\\&I\oplus W_0'\ar[ur]_{\smtrx{1&0}}&~}\]
Thus, $\ad_\eta$ restricted to $\Hom(W_p,W_0')$ is given by $\beta_p'\mapsto -\eta_{-p}\beta_p^{\prime\ast}.$ Hence, 
\[\ad_\eta=\smtrx{0&b\\\eta_{-p}&0\\0&a}:\Hom(W_p,W_0')\oplus A\longrightarrow K^{p}\oplus \Hom(W_0',K^p)\oplus A\] where $a:A\to A$ is an isomorphism.

Since $H^1(K^p)=0$, we have $\HH^2(C^\bullet_{-p})=0$. As in the odd case, we also find that $\HH^0(C^\bullet_{-p})=0$. Moreover, $\HH^1(C^\bullet_{-p})\cong H^0(K^p)\oplus \HH^1_{-p}$ where $\HH^1_{-p}$ is given by \eqref{eq H1-p}.
\end{proof}

The final case concerns the weight zero subcomplex.

\begin{lemma}
    \label{lemma C0 HH10 and HH20} There is a bundle $A$ so that
    \[\xymatrix@=1.5em{\fso_0(V)\oplus \fso_0(W)\cong\fso(W_0')\oplus \End(W_{-p})\oplus A&\text{and}&\Hom_1(W,V)\otimes K\cong \Hom(W_{-p},IK^p)\oplus A},\]
    where $\fso(W_0')$ is the bundle of skew-symmetric endomorphisms of $W_0'$ (with respect to to $Q_{W_0'}$). With respect to this splitting,
    \[\ad_\eta=\smtrx{0&\eta_{-p}&0\\0&b&a}:\xymatrix{\fso(W_0')\oplus \End(W_{-p})\oplus A\ar[r]&\Hom(W_{-p},IK^p)\oplus A},\]
    where $a:A\to A$ is an isomorphism. 
    In particular, 
    \[\xymatrix{\HH^2(C^\bullet_0)=0,&\HH^0(C^\bullet_0)=H^0(\fso(W_0'))& \text{and} &\HH^1(C^\bullet_0)=H^1(\fso(W_0'))\oplus \HH^1_{0,p}~,}\] where
    \[\xymatrix@=.8em{0\ar[r]&H^0(\End(W_{-p}))\ar[r]^-{\eta_{-p}}&H^0(\Hom(W_{-p},IK^p))\ar[r]&\HH^1_{0,p}\ar[r]&H^1(\End(W_{-p}))\ar[r]&H^1(\Hom(W_{-p},IK^p)\ar[r]&0.}\]
\end{lemma}
\begin{proof}
    By $\eqref{eq HomWV2k+1}$ we have
    \[\Hom_1(W,V)\otimes K\cong\begin{dcases}
        \Hom(W_{-p},IK^p)\oplus \underbrace{\cO\oplus\cdots\oplus\cO}_{p-1\ \text{times}} &\text{$p$ odd}\\
        \Hom(W_{-p},IK^p)\oplus \underbrace{\cO\oplus\cdots\oplus\cO}_{p-1\ \text{times}}\oplus \Hom(W_0',I)&\text{$p$ even}
    \end{dcases} \]
and by \eqref{eq fsoV2k} and \eqref{eq fsoW2k}, 
    \[\fso_0(V)\oplus \fso_0(W)\cong \begin{dcases}
        \End(W_{-p}) \oplus \underbrace{\cO\oplus\cdots\oplus\cO}_{p-1\ \text{times}}\oplus\,\fso(W_0')&\text{$p$ odd} \\
        \End(W_{-p})\oplus \underbrace{\cO\oplus\cdots\oplus\cO}_{p-1\ \text{times}}\oplus \Hom(W_0',I)\oplus \fso(W_0')&\text{$p$ even.}
    \end{dcases}\]
Hence, setting $A$ to be 
\[A=\begin{dcases}
    \underbrace{\cO\oplus\cdots\oplus\cO}_{p-1\ \text{times}}&\text{$p$ odd}\\ 
    \underbrace{\cO\oplus\cdots\oplus\cO}_{p-1\ \text{times}}\oplus \Hom(W_0',I)&\text{$p$ even}
\end{dcases}\]
yields $\fso_0(V)\oplus \fso_0(W)=\fso(W_0')\oplus \End(W_{-p})\oplus A$ and $\Hom(W,V)_1\otimes K= \Hom(W_{-p},IK^p)\oplus A.$

Since, $W_0'$ is an invariant bundle, the restriction of the map $\ad_\eta:\fso_0(W)\oplus \fso_0(V)\to\Hom_1(W,V)\otimes K$ to $\fso(W_0')$ is identically zero. The restriction of the map $\ad_\eta$ to $\End(W_{-p})\oplus A$ is similar to \eqref{eq ad eta 2kgeq0} with the exception that the term $W_{-p}\to V_{1-p}\otimes K$ is given by
\[\xymatrix@=.8em{&V_{1-p}\otimes K\ar[dr]^-{\alpha_{1-p}\otimes\Id_K}&\\W_{-p}\ar[ur]^-{\eta_{-p}}\ar[dr]_-{\beta_{-p}}&&V_{1-p}\otimes K.\\&W_{-p}\ar[ur]_-{\eta_{-p}}&}\] 
In particular, it is given by $\smtrx{\eta_{-p}&0\\b&a}:\End(W_{-p})\oplus A\to \Hom(W_{-p},IK^p)\oplus A$ where $a$ is an isomorphism. 

The hypercohomology complex for $C^\bullet$ splits as a direct sum of the following two complexes 
\[\xymatrix{0\ar[r]&\HH^0_{0,'}\ar[r]&H^0(\fso(W_0'))\ar[r]&0\ar[r]&\HH^1_{0,'}\ar[r]&H^1(\fso(W_0'))\ar[r]&0,}\]
and
\[\xymatrix@R=1em{0\ar[r]&\HH^0_{0,p}\ar[r]&H^0(\End(W_{-p}))\ar[r]&H^0(\Hom(W_{-p}, IK^p))\ar[r]&\HH^1_{0,p}\ar@{->}`r/3pt [d] `/10pt[l] `^d[llll] `^r/3pt[d][dlll]&\\&H^1(\End(W_{-p}))\ar[r]&H^1(\Hom(W_{-p},IK^p)) \ar[r]&\HH^2_{0,p}\ar[r]&0.}\]
By Lemma \ref{lemma: fixed point SO(1,n)}, $\left(W_p\oplus I\oplus  W_{-p},\smtrx{0&0&0\\\eta_{-p}&0&0\\0&\eta_{-p}^*&0}\right)$ is a stable $K^p$-twisted $\rO(2\rk(W_p)+1,\C)$-Higgs bundle, so the hypercohomology groups $\HH^0_{0,p}$ and $\HH^2_{0,p}$ both vanish and $\HH^1(C^\bullet_0)=H^1(\fso(W_0'))\oplus \HH^1_{0,p}.$
\end{proof}

 \subsection{Proof of Theorem \ref{Thm Psi open and closed}}

 We are now set up to prove Theorem \ref{Thm Psi open and closed}. We start by describing a neighborhood of the image of the map $\Psi$ which is open in $\cM(\SO(p,q)).$
 \begin{proposition}
    \label{prop: open nbhd of imPsi}
    For each $(I,\widehat W,\hat\eta,q_2,\ldots,q_{2p-2})$ in $\cM_{K^p}(\SO(1,q-p+1))\times \bigoplus\limits_{j=1}^{p-1}H^0(K^{2j})$, the second hypercohomology group for the associated $\SO(p,q)$-Higgs bundle vanishes
    \[\HH^2(C^\bullet(\Psi(I,\widehat W,\hat\eta,q_2,\ldots,q_{2p-2})))=0.\]
     In particular, an open neighborhood of $\Psi(I,\widehat W,\hat\eta,q_2,\ldots,q_{2p-2})$ in $\cM(\SO(p,q))$ is isomorphic to an open neighborhood of zero in
    \[\HH^1(C^\bullet(\Psi(I,\widehat W,\hat\eta,q_2,\ldots,q_{2p-2}))\sslash \Aut(\Psi(I,\widehat W,\hat\eta,q_2,\ldots,q_{2p-2})).\]
\end{proposition}
\begin{proof}
    By Lemma \ref{Lemma: reduction to fixed points H2=0}, it suffices to prove the above proposition at the fixed points of the $\C^*$-action in the image of $\Psi.$ These are the Higgs bundles given in Lemma \ref{Lem: Fixed points in the image of Psi}. In Lemmas \ref{lemma k>0 H1 vanishing H2 not p 2p}, \ref{lemma p, 2p H1 vanishing H2}, \ref{lemma 2k+1<0 H1 vanishing H2 not -p -2p}, \ref{lemma 2k<0 H1 vanishing H2 not -p -2p}, \ref{lemma -p  H1 vanishing H2} and \ref{lemma C0 HH10 and HH20} it is shown that if $(W,V,\eta)$ is a fixed point of the $\C^*$-action in the image of $\Psi$, then each of the graded pieces of $\HH^2(C^\bullet(W,V,\eta))$ vanish.
\end{proof}
\begin{proposition}
    \label{proposition H1 iso at fixed points}
    For all $\Psi((I,\widehat W,\hat\eta),0,\ldots,0)$ which are fixed points of the $\C^*$-action we have an isomorphism induced by $\Psi$.
    \[\HH^1(C^\bullet(\Psi(I,\widehat W,\hat\eta,0,\ldots,0))\sslash\Aut(\Psi(I,\widehat W,\hat\eta,0,\ldots,0))\cong \Big(\HH^1(C^\bullet(I,\widehat W,\hat\eta))\sslash\Aut(\widehat W)\Big)\times \bigoplus\limits_{j=1}^{p-1}H^0(K^{2j}).\] 
\end{proposition}

\begin{proof}
   Let $\Psi(I, \widehat W,\hat\eta,0,\ldots,0)$ be a fixed point of the $\C^*$-action. For the $\SO(1,q-p+1)$-Higgs bundle $(I,\widehat W,\hat\eta),$ the first hypercohomology group $\HH^1(C^\bullet(I,\widehat W,\hat\eta))$ of the deformation complex was computed in Lemma \ref{Lemma: H1 of SO(1,n) fixed point}. 
  In Lemmas \ref{lemma k>0 H1 vanishing H2 not p 2p}, \ref{lemma p, 2p H1 vanishing H2}, \ref{lemma 2k+1<0 H1 vanishing H2 not -p -2p}, \ref{lemma 2k<0 H1 vanishing H2 not -p -2p}, \ref{lemma -p  H1 vanishing H2} and \ref{lemma C0 HH10 and HH20} it was shown that the first hypercohomology group of the deformation complex of the $\SO(p,q)$-Higgs bundle is given by
 \[\HH^1(C^\bullet(\Psi(I,\widehat W,\hat\eta,0,\ldots,0)))\cong \HH^1(C^\bullet(I,\widehat W,\hat\eta))\times \bigoplus\limits_{j=1}^{p-1}H^0(K^{2j}).\]
 It is clear from our constructions that the isomorphism is induced by $\Psi$.

 By Lemma \ref{Lemma Gauge group fixing im Psi}, every $\rS(\rO(1,\C)\times \rO(q-p+1,\C))$ automorphism $(\det(g_{\widehat W}),g_{\widehat W})$ of $(I,\widehat W,\hat \eta)$ determines a unique automorphism of $\Psi(I,\widehat W,\hat\eta,0\dots, 0)
 $
 \[(g_V,g_W)=\left(\det(g_{\widehat W})\Id_{\cK_{p-1}},\smtrx{g_{\widehat W}&0\\0&\det(g_{\widehat W})\Id_{\cK_{p-2}}}\right)~.\] 
 Moreover, the action of such an automorphism on the holomorphic differentials in the above description of $\HH^1(C^\bullet(\Psi(I,\widehat W,\hat\eta,0,\ldots,0))$ is trivial. Thus,
\[\HH^1(C^\bullet(\Psi(I,\widehat W,\hat\eta,0,\ldots,0))\sslash\Aut(\Psi(I,\widehat W,\hat\eta,0,\ldots,0))\cong \Big(\HH^1(C^\bullet(I,\widehat W,\hat\eta))\sslash\Aut(\widehat W)\Big)\times \bigoplus\limits_{j=1}^{p-1}H^0(K^{2j})\]
as claimed. 
\end{proof}

\begin{theorem}
The image of the map $\Psi$ from \eqref{Eq Psi} is open and closed.  
\end{theorem}
\begin{proof}
    By Propositions \ref{prop: open nbhd of imPsi} and \ref{proposition H1 iso at fixed points}, the map $\Psi$ is open at all fixed points of the $\C^*$-action. For $(V,W,\eta)$ in the image of $\Psi$, there is $\lambda$ sufficiently close to zero such that $(V,W,\lambda\eta)$ is in a sufficiently small open neighborhood of a fixed point of the $\C^*$-action. Thus, $\Psi$ is open at all points.

    To show the image of $\Psi$ is closed, we use the properness of the Hitchin fibration. Namely, suppose $(I,\widehat W_i,\hat\eta_i,q_2^i,\dots,q_{2p-2}^i)$ is a sequence of points in $\cM_{K^p}(\SO(1,q-p+1))\times \bigoplus\limits_{j=1}^{p-1} H^0(K^{2j})$ which diverges. 
    Denote the associated Hitchin fibrations by 
    \[\xymatrix{h_p:\cM_{K^p}(\SO(1,q-p+1))\to H^0(K^{2p})&\text{and}&h:\cM(\SO(p,q))\to \displaystyle\bigoplus\limits_{j=1}^p H^0(K^{2j})}~.\]
    By the properness of $h_p$, $(q_2^i,\ldots,q^i_{2p-2},h_p(I,\widehat W_i,\hat\eta_i))$ diverges in $\bigoplus\limits_{j=1}^pH^0(K^{2j})$. Moreover, by the definition of the map $\Psi,$ applying the $\SO(p,q)$-Hitchin fibration to the image sequence  yields
    \[h(\Psi(I,\widehat W_i,\hat \eta_i,q_2^i,\ldots,q^i_{2p-2}))=(q_2^i,\ldots,q_{2p-2}^i,h_p(I,\widehat W_i,\hat \eta_i))~.\]
Since $h$ is proper, we conclude that $\Psi(I,\widehat W_i,\hat \eta_i,q_2^i,\ldots,q^i_{2p-2})$ also diverges in $\cM(\SO(p,q)).$
    \end{proof} 

The following direct consequence of the construction of the map $\Psi$ will be used in Section \ref{section positive}. 
\begin{corollary}\label{Cor: Higgs irr.}
    Consider the subgroup $\GL(n,\R)\times \SO(p-n,q-n)\subset\SO(p,q)$ defined by the embedding 
    \[(A,B)\mapsto \smtrx{A&\\&B\\&&A^{-1}}.\]
    Then no Higgs bundle in the image of $\Psi$ reduces to such a subgroup.
\end{corollary}

\section{Classification of local minima of the Hitchin function for $\cM(\SO(p,q))$}

In this section we will prove Theorem \ref{thm:MINIMA CLASSIFICATION} which classifies all local minima of the Hitchin function \eqref{EQ Hitchin Function} on $\cM(\SO(p,q))$. 
The strategy of proof is to divide the objects into the following three families:
\begin{enumerate}
\item stable $\SO(p,q)$-Higgs bundles with $\HH^2(C^\bullet(V,W,\eta))=0$,
\item stable $\SO(p,q)$-Higgs bundles whose corresponding $\SO(p+q,\C)$-Higgs bundle is strictly polystable,
\item strictly polystable $\SO(p,q)$-Higgs bundles.
\end{enumerate}
The first family consists of points which are either smooth or orbifold points of $\cM(\SO(p,q))$. For these points we can use Proposition \ref{prop: minima criteria} to classify such local minimum. The local minima in the other two families will be described by a direct study of their deformations.

Recall from \eqref{Eq fixedpoint deformation sheaf splitting} that the deformation complex of an $\SO(p,q)$-Higgs bundle $(V,W,\eta)$ which is a $\C^*$-fixed point decomposes as 
\begin{equation}\label{eq def complex splitting minima section}
C^\bullet_k:\fso_k(V)\oplus\fso_k(W)\xrightarrow{\ \ \ad_\eta\ \ }\Hom_{k+1}(W,V)\otimes K.
\end{equation}
Each graded piece gives rise to the long exact sequence \eqref{EQ deformation complex splitting} in hypercohomology.

\subsection{Stable minima with vanishing $\HH^2(C^\bullet)$}
\label{sec:minimal-subvarieties}
By Proposition \ref{prop: stable fixed points}, stable $\C^*$-fixed points are given by \eqref{eq: holomorphic chain int}. 
We start by studying the constraints on these chains imposed by the local minima condition for stable $\SO(p,q)$-Higgs bundles with vanishing $\HH^2(C^\bullet)$. This will be done by first proving two lemmas.

\begin{lemma}\label{lemma:structureoftwochains}
Let $(V,W,\eta)$ be a stable $\SO(p,q)$-Higgs bundle with $\eta\neq0$ and $\HH^2(C^\bullet(V,W,\eta))=0$. If $(V,W,\eta)$ is a local minimum of $f$, then the chain given by \eqref{eq: holomorphic chain int} must have one of the following forms (with $\eta_i\neq 0$ for all $i$):
\begin{equation}\label{eq: s<r minimal hol chain 1 chain}
    \xymatrix@C=1.8em@R=0em{V_{-s}\ar[r]^-{\eta_{s-1}^*}&W_{1-s}\ar[r]^-{\eta_{1-s}}&\cdots\ar[r]^-{\eta_{-2}}&V_{-1}\ar[r]^-{\eta_{0}^*}&W_0\ar[r]^-{\eta_{0}}&V_1\ar[r]^-{\eta_{-2}^*}&\cdots\ar[r]^-{\eta^*_{1-s}}&W_{s-1}\ar[r]^-{\eta_{s-1}}&V_s}
\end{equation} 
\begin{equation}\label{eq: r<s minimal hol chain 1 chain}
\xymatrix@C=1.8em@R=0em{W_{-r}\ar[r]^-{\eta_{-r}}&V_{1-r}\ar[r]^-{\eta_{r-2}^*}&\cdots\ar[r]^-{\eta_1^*}&W_{-1}\ar[r]^-{\eta_{-1}}&V_0\ar[r]^-{\eta^*_{-1}}&W_1\ar[r]^-{\eta_1}&\cdots\ar[r]^-{\eta_{r-2}}&V_{r-1}\ar[r]^-{\eta_{-r}^*}&W_r} 
\end{equation} 
\begin{equation}\label{eq: s<r minimal hol chain 2 chains}
\xymatrix@C=1.8em@R=0em{V_{-r}\ar[r]^-{\eta_{r-1}^*}&W_{1-r}\ar[r]^-{\eta_{1-r}}&\cdots\ar[r]^-{\eta_1^*}&W_{-1}\ar[r]^-{\eta_{-1}}&V_0\ar[r]^-{\eta_{-1}^*}&W_1\ar[r]^-{\eta_1}&\cdots\ar[r]^-{\eta^*_{1-r}}&W_{r-1}\ar[r]^-{\eta_{r-1}}&V_r,\\&&&&\oplus&&&&\\&&&&W_0&&&&}
\end{equation} 
\begin{equation}\label{eq: r<s minimal hol chain 2 chains}
\xymatrix@C=1.8em@R=0em{W_{-s}\ar[r]^-{\eta_{-s}}&V_{1-s}\ar[r]^-{\eta_{s-2}^*}&\cdots\ar[r]^-{\eta_{-2}}&V_{-1}\ar[r]^-{\eta_{0}^*}&W_0\ar[r]^-{\eta_{0}}&V_1\ar[r]^-{\eta_{-2}^*}&\cdots\ar[r]^-{\eta_{s-2}}&V_{s-1}\ar[r]^-{\eta_{-s}^*}&W_s\\&&&&\oplus&&&&\\&&&&V_0&&&&} 
\end{equation} 
\end{lemma}
\begin{proof}
If one of the chains in \eqref{eq: holomorphic chain int} vanishes, we are done. Assume both chains are non-zero chains. Let $r\geq 0$ be the maximal weight of the first chain and $s\geq 0$ be the maximal weight of the second chain. We have $r>0$ or $s>0$ since $\eta\neq0$. 
Since $(V,W,\eta)$ is a stable local minimum of the Hitchin function with $\HH^2(C^\bullet)=0$, the subcomplexes from \eqref{eq def complex splitting minima section} are isomorphisms for $k\geq 1$ by Proposition \ref{prop: minima criteria}.

 If $r$ and $s$ have different parity, then both of the chains start and end with a summand of $W$ if $r$ is odd and start and end with a summand of $V$ if $r$ is even. In either case, $\Hom_{r+s+1}(W,V)\otimes K=0$ but $\fso_{r+s}(W)\oplus \fso_{r+s}(V)$ is non-zero. Hence, the subcomplex $C^\bullet_{r+s}$ from \eqref{eq def complex splitting minima section} is not an isomorphism for $k=r+s$, contradicting $(V,W,\eta)$ being a stable minima with $\HH^2(C^\bullet)=0.$

Now assume $r$ and $s$ have the same parity, so the first chain starts and ends with a summand of $W$ if and only if $r$ is odd and the second chain starts and ends with a summand of $W$ if only only if $s$ is even. 
If $r\geq s,$ then $\Hom_{2r+1}(W,V)\otimes K=0$ and $\fso_{2r}(V)\oplus\fso_{2r}(W)=\Lambda^2V_r\oplus \Lambda^2W_r$. So the isomorphism of $C^\bullet_{2r}:\Lambda^2V_r\oplus \Lambda^2W_r\to 0$ implies that, whenever $V_r$ and $W_r$ are non-zero, they must be line bundles; more precisely we must have: (i) if $r$ is odd, $\rk(W_r)=1$ and, if $s=r$, $\rk(V_r)=1$ (if $s<r$, $V_r=0$), or (ii) if $r$ is even, $\rk(V_r)=1$ and, if $s=r$, $\rk(W_r)=1$ (if $s<r$, $W_r=0$). 
Since $r+s-1$ is odd, we have:
 \[\fso_{r+s-1}(V)=\begin{dcases}\{(\alpha,-\alpha^*)\in\Hom(V_{-s},V_{r-1})\oplus\Hom(V_{1-r},V_{s})\}
    &\text{if $r$ is odd}\\
    \{(\alpha,-\alpha^*)\in\Hom(V_{-r},V_{s-1})\oplus\Hom(V_{1-s},V_{r})\}&\text{if $r$ is even}
 \end{dcases}
\]
 \[\fso_{r+s-1}(W)=\begin{dcases}
    \{(\beta,-\beta^*)\in\Hom(W_{-r},W_{s-1})\oplus\Hom(W_{1-s},W_{r})\}&\text{if $r$ is odd}\\
    \{(\beta,-\beta^*)\in\Hom(W_{-s},W_{r-1})\oplus\Hom(W_{1-r},W_{s})\}&\text{if $r$ is even}
 \end{dcases}
\]
\[\Hom_{r+s}(W,V)\otimes K\cong \begin{dcases}
    \Hom(W_{-r},V_s)\otimes K&\text{if $r$ is odd}\\
    \Hom(W_{-s},V_r)\otimes K&\text{if $r$ is even~.}
 \end{dcases}\]
If $s>0$, then $r+s-1\geq 1$ so the isomorphism $C^\bullet_{r+s-1}:\fso_{r+s-1}(V)\oplus\fso_{r+s-1}(W)\to \Hom_{r+s}(W,V)\otimes K$ gives
\[\begin{dcases}
    \rk(V_{s})\rk(V_{r-1})+\rk(W_{s-1})=\rk(V_s)&\text{if $r$ is odd}\\
    \rk(W_{s})\rk(W_{r-1})+\rk(V_{s-1})=\rk(W_s)&\text{if $r$ is even~.}
\end{dcases}\]
This implies either $\rk(W_{s-1})=0$ or $\rk(V_{r-1})=0$ if $r$ is odd, and that either $\rk(V_{s-1})=0$ or $\rk(W_{r-1})=0$ if $r$ is even. Any of these conclusions contradicts Proposition \ref{prop: stable fixed points}. Thus, we conclude that $s=0$ and thus $r$ is even, so the holomorphic chain is given by \eqref{eq: s<r minimal hol chain 2 chains}. A similar argument shows that the holomorphic chain is of the form \eqref{eq: r<s minimal hol chain 2 chains} for $s> r$.
\end{proof}

\begin{lemma}\label{lemma:linebundles and powers of K}
Let $(V,W,\eta)$ be a stable $\SO(p,q)$-Higgs bundle which is a local minimum of the Hitchin function with $\eta\neq0$ and $\HH^2(C^\bullet(V,W,\eta))=0$; the associated holomorphic chain is given by \eqref{eq: s<r minimal hol chain 1 chain}, \eqref{eq: r<s minimal hol chain 1 chain}, \eqref{eq: s<r minimal hol chain 2 chains} or \eqref{eq: r<s minimal hol chain 2 chains}. For all $j\neq0,$ we have $\rk(W_j)=1$ and $\rk(V_j)=1.$ Moreover:
\begin{itemize}
\item In case \eqref{eq: s<r minimal hol chain 1 chain}, $V_j\cong V_{-1} K^{-j-1}$ and $W_j\cong V_{-1}K^{-j-1}$ for $0<|j|<s$.
\item In case \eqref{eq: r<s minimal hol chain 1 chain}, $V_j\cong W_{-1}K^{-j-1}$ and $W_j\cong W_{-1}K^{-j-1}$ for $0<|j|<r.$
\item In case \eqref{eq: s<r minimal hol chain 2 chains}, $\rk(V_0)=1$, and $V_j\cong V_0K^{-j}$ and $W_j\cong V_0K^{-j}$ for $0<|j|<r$.
\item In case \eqref{eq: r<s minimal hol chain 2 chains}, $\rk(W_0)=1$, and $V_j\cong V_0K^{-j}$ and $W_j\cong V_0K^{-j}$ for $0<|j|<s$.
\end{itemize}
\end{lemma}

\begin{proof}
The proof involves an inductive argument on the weights. We first consider the case where $(V,W,\eta)$ is the holomorphic chain \eqref{eq: s<r minimal hol chain 2 chains}. We have the following decompositions 
\[\xymatrix@=.5em{\End(V)=\displaystyle\bigoplus\limits_{j=-2r}^{2r}\End_k(V)~,&\End(W)=\displaystyle\bigoplus\limits_{k=2-2r}^{2r-2}\End_k(W)&\text{and}&\Hom(W,V)=\displaystyle\bigoplus\limits_{k=1-2r}^{2r-1}\Hom_k(W,V).}\]
For $2k>0$ we have $\Hom_{2k+1}(W,V)=\displaystyle\bigoplus\limits_{j=0}^{r-k-1}\Hom(W_{1-r+2j},V_{2-r+2j+2k})$, 
\begin{equation}
    \label{Eq End2k splittings section minima}\xymatrix@=.5em{\End_{2k}(V)=\displaystyle\bigoplus\limits_{j=0}^{r-k}\Hom(V_{2j-r},V_{2j+2k-r})&\text{and} &\End_{2k}(W)=\displaystyle\bigoplus\limits_{j=0}^{r-k-1}\Hom(W_{1-r+02j},W_{1-r+2j+2k}).}
\end{equation}
With respect to these splittings, $\fso(V)=\displaystyle\bigoplus\fso_k(V)$ and $\fso(W)=\displaystyle\bigoplus\fso_k(W)$ where, for $k>0$
\begin{equation}\label{Eq so2k splittings section minima}
    \xymatrix@=0em{\fso_{2k}(V)=\{(\alpha_0,\ldots,\alpha_{r-k})\in\End_{2k}(V)\ |\ \alpha_i+\alpha_{r-k-i}^*=0 \},\\
\fso_{2k}(W)=\{(\beta_0,\ldots,\beta_{r-k-1})\in\End_{2k}(V)\ |\ \beta_i+\beta_{r-k-1-i}^*=0 \}~.}
\end{equation}

Since $(V,W,\eta)$ is a stable minima of the Hitchin function with $\HH^2(C^\bullet)=0,$ for all $k>0$ we have 
$\fso_{2k}(V)\oplus \fso_{2k}(W)\cong\Hom_{2k+1}(W,V)\otimes K$.
Note that $r$ is even and non-zero. 
The isomorphism for $k=2r$ implies $\Lambda^2V_r\cong 0,$ hence $\rk(V_r)=1$. 

The isomorphism for $k=2r-2$ implies $\Hom(V_{-r},V_{r-2})\oplus \Lambda^2W_{r-1}\cong\Hom(W_{1-r},V_r)\otimes K.$ Thus,
\[\rk(V_{r-2})+\rk(\Lambda^2W_{1-r})=\rk(W_{1-r}),\]
which implies $\rk(W_{1-r})$ is either one or two. If $\rk(W_{1-r})=2$, taking the determinant of the isomorphism $C^\bullet_{2r-2}$ implies $V_rK^2=V_{r-2}$.  
Also, the kernels of the maps $\eta_{r-1}:W_{r-1}\to V_r\otimes K$ and $\eta_{1-r}:W_{1-r}\to V_{2-r}\otimes K$ have negative degree by stability. 
Using $V_j^*\cong V_{-j}$ and $W_j^*\cong W_{-j}$, we have 
\[\deg(V_{r-2})-2g+2<\deg(W_{r-1})<\deg(V_{r})+2g-2,\]
which contradicts $V_rK^2=V_{r-2}$. So rank $W_{r-1}=1$ and the isomorphism for $C^\bullet_{2r-2}$ gives the base case of our induction:
\[\xymatrix{1=\rk(V_{-r})=\rk(W_{1-r})=\rk(V_{2-r})&\text{and}&W_{1-r}\cong V_{2-r}K}~.\] 

If $r=2$ we are done, so assume $r\geq 4$ and that for an integer $k\in[1,\frac{r}{2}-1]$ we have
\begin{equation}
    \label{EQ induction hypothesis}W_{1-r}\cong V_{2-r}K\cong W_{3-r}K^2\cong \cdots\cong W_{2k-1-r}K^{2k-2}\cong V_{2k-r}K^{2k-1}~.
\end{equation}
We will prove that $V_{2k-r}\cong W_{2k+1-r}K\cong V_{2k+2-r}K^2$. 

The isomorphism $C^\bullet_{2r-2-2k}$ gives 
\begin{equation}\label{eq:2r-2k-2 iso}
    \xymatrix@=0em{\displaystyle\bigoplus\limits_{j=0}^{\left\lfloor\frac{k}{2}\right\rfloor}\Hom(V_{2j-r},V_{r+2j-2-2k})\oplus\displaystyle\bigoplus\limits_{j=0}^{\left\lfloor \frac{k-1}{2}\right\rfloor}\Hom(W_{2j+1-r},W_{r+2j-1-2k})\\\cong\ \ \ \displaystyle\bigoplus\limits_{j=0}^{k}\Hom(W_{2j+1-r},V_{r+2j-2k})\otimes K.}
\end{equation}
since $\Lambda^2V_{r-k-1}=0$ for $k$ odd and $\Lambda^2W_{r-k-1}=0$ for $k$ even by \eqref{EQ induction hypothesis}. Using \eqref{EQ induction hypothesis}, computing the ranks of both sides gives $\rk(V_{2k+2-r})+\left\lfloor\frac{k}{2}\right\rfloor+\rk(W_{2k+1-r})+\left\lfloor\frac{k-1}{2}\right\rfloor=k+\rk(W_{2k+1-r})$. Thus,
\[\rk(V_{2k+2-r})=1.\]

The isomorphism $C^\bullet_{2r-2-4k}$ implies
\[\displaystyle\bigoplus\limits_{j=0}^k\Hom(V_{2j-r},V_{r+2j-2-4k})\oplus\displaystyle\bigoplus\limits_{j=0}^{k-1}\Hom(W_{2j+1-r},W_{r+2j-1-4k})\oplus\Lambda^2W_{r-1-2k}\]\[\ \cong\ \ \ \displaystyle\bigoplus_{j=0}^{2k}\Hom(W_{2j+1-r},V_{r+2j-4k})\otimes K. \]
Using \eqref{EQ induction hypothesis}, this gives the following equality on ranks
\[\sum\limits_{j=0}^k \rk(V_{r+2j-2-4k})+\sum\limits_{j=0}^{k-1}\rk(W_{r+2j-1-4k})+\rk(\Lambda^2 W_{r-1-2k})= \sum\limits_{j=0}^{k-1}\rk(V_{r+2j-4k})+\sum\limits_{j=k}^{2k}\rk(W_{2j+1-r}).\]
Simplifying, yields $\rk(V_{4k+2-r})+\rk(\Lambda^2W_{2k+1-r})=\rk(W_{2k+1-r})$. Thus, $\rk(W_{2k+1-r})$ is one or two.  

If $\rk(W_{2k+1-r})=2,$ then the determinant of the isomorphism in \eqref{eq:2r-2k-2 iso} gives
\begin{equation}\label{eq:detC2r-2-2k}
\xymatrix@=0em{
\displaystyle\bigotimes\limits_{j=0}^{\left\lfloor\frac{k}{2}\right\rfloor}V_{r-2j}V_{r+2j-2-2k} \otimes W_{r-1}^2\Lambda^2W_{r-1-2k}\otimes\displaystyle\bigotimes\limits_{j=1}^{\left\lfloor\frac{k-1}{2}\right\rfloor}W_{r-2j-1}W_{r+2j-1-2k}\\
\cong\ \ \ \ \  \displaystyle\bigotimes\limits_{j=0}^{k-1}W_{r-2j-1}V_{r+2j-2k}K\otimes V_r^2K^2\otimes \Lambda^2W_{r-1-2k}~.}
\end{equation}
By \eqref{EQ induction hypothesis}, the above terms satisfy 
\begin{equation}\label{eq:determinantC2r-2-2k}
V^2_{r-2k}K^{2-2k}\cong \begin{dcases}
    V_{r-2j}V_{r+2j-2-2k},&\text{ for }j=1,\ldots,\lfloor \frac{k}{2}\rfloor \\
    W_{r-2j-1}W_{r+2j-1-2k},&\text{ for }j=1,\ldots,\lfloor \frac{k-1}{2}\rfloor\\
    W_{r-2j-1}V_{r+2j-2k}K,&\text{ for }j=0,\ldots,k-1.
\end{dcases}
\end{equation}
Hence, simplifying \eqref{eq:detC2r-2-2k} yields $V_{r-2k-2}\cong V_rK^{2+2k}.$ 
The Higgs field gives rise to non-zero maps $V_{r-2k-2}\rightarrow V_{r-2k}K^2$ and $V_{r-2k}\rightarrow V_{r}K^{2k}$ by Proposition \ref{prop: stable fixed points}. Thus, $\deg(V_{r-2k-2})-\deg(V_{r-2k})=4g-4$. As in the base case, this leads to a contradiction of stability. 
Namely, stability implies that the kernels of $\eta_{2k+1-r}:W_{2k+1-r}\rightarrow V_{2k+2-r} K$ and of $\eta_{r-1-2k}:W_{r-1-2k}\rightarrow V_{r-2k} K$ have negative degree, so that
$\deg(V_{2k-r})-2g+2<\deg(W_{2k+1-r})<\deg(V_{2k+2-r})+2g-2$. So $\rk(W_{2k+1-r})=1.$

 Using $\rk(W_{2k+1-r})=1,$ \eqref{EQ induction hypothesis} and \eqref{eq:determinantC2r-2-2k}, the determinant of  \eqref{eq:2r-2k-2 iso} gives
\begin{equation*}
\begin{split}
V_rV_{r-2k-2}\otimes\displaystyle\bigotimes\limits_{j=1}^{\left\lfloor\frac{k}{2}\right\rfloor}(V_{r-2k}^2K^{2-2k}) \otimes V_{r-2k}K^{1-2k}W_{r-1-2k}\otimes\displaystyle\bigotimes\limits_{j=1}^{\left\lfloor\frac{k-1}{2}\right\rfloor}(V_{r-2k}^2K^{2-2k})\\
\cong\ \ \ \ \ \displaystyle\bigotimes\limits_{j=0}^{k-1} (V_{r-2k}^2 K^{2-2k}) \otimes W_{r-1-2k}V_rK~,\ \ \ \ \ \ \ \ \ \ \ \ \ \ \ 
 \ \ \ \ \ \ \ 
\end{split}
\end{equation*}
which simplifies to $V_{2k-r}\cong V_{2k+2-r}K^2.$ 
The Higgs field defines a non-zero map $V_{2k-r}\to W_{2k+1-r}K\to V_{2k+2-r}K^2$. Thus,
\begin{equation}\label{eq:finalstep}
V_{2k-r}\cong W_{2k+1-r}K\cong V_{2k+2-r}K^2~.
\end{equation}
Recall that $k$ was an integer between $1$ and $\frac{r-2}{2}$. Since $r$ is even, we can take $k=(r-2)/2$, and hence \eqref{eq:finalstep} gives $V_{-2}\cong W_{-1}K\cong V_0K^2$. This completes the proof for the chain \eqref{eq: s<r minimal hol chain 2 chains}.

The difference for the chain \eqref{eq: r<s minimal hol chain 1 chain} is that $r$ is odd and instead of \eqref{EQ induction hypothesis} we must assume 
\[V_{1-r}\cong W_{2-r}K\cong V_{3-r}K^2\cong \cdots\cong V_{2k-1-r}K^{2k-2}\cong W_{2k-r}K^{2k-1},\]
where $k$ is an integer satisfying $1\leq k\leq (r-3)/2$. The same proof as above shows that $W_{2k-r}\cong V_{2k+1-r}K\cong W_{2k+2-r}K^2$. By taking $k=(r-3)/2$ we have $W_{-3}\cong V_{-2}K\cong W_{-1}K^2$, and no condition on $V_0$ is imposed. 
Switching the roles of $V$ and $W$ gives the proof for the chains \eqref{eq: s<r minimal hol chain 1 chain} and \eqref{eq: r<s minimal hol chain 2 chains}.
\end{proof}
We can now complete the classification of the stable minima with  vanishing $\HH^2(C^\bullet)$.
\begin{theorem}\label{thm:stablenon-zerominH2=0}
A stable $\SO(p,q)$-Higgs bundle $(V,W,\eta)$ with $p\leq q$, $\eta\neq0$ and $\HH^2(C^\bullet(V,W,\eta))=0$ defines a local minimum of the Hitchin function if and only if it is a holomorphic chain of the form \eqref{eq: s<r minimal hol chain 1 chain}, \eqref{eq: r<s minimal hol chain 1 chain}, \eqref{eq: s<r minimal hol chain 2 chains} or \eqref{eq: r<s minimal hol chain 2 chains} which satisfies one of the following:
\begin{enumerate}
    \item The chain is given by \eqref{eq: s<r minimal hol chain 1 chain} with $p=2$ and $0<\deg(V_{-1})<2g-2$.
    \item The chain is given by \eqref{eq: s<r minimal hol chain 1 chain} with $p\geq 2$, $s=p-1$ and the bundle $W_0$ decomposes as $W_0=I\oplus W_0'$, where $W_0'$ is a stable $\rO(q-p+1,\C)$-bundle with $\det(W_0')=I$. Moreover, $V_j=IK^{-j}$ and $W_j=IK^{-j}$ for all $j\neq 0,$ and with respect to the splitting of $W_0$, the chain is given by 
    \begin{equation}
        \label{eq holomorphic chain stable min thm}\xymatrix@C=2.3em@R=-.4em{&&&&I&\\V_{-s}\ar[r]^-{\eta_{s-1}^*}&W_{1-s}\ar[r]^-{\eta_{1-s}}&\cdots\ar[r]^-{\eta_{-2}}&V_{-1}\ar[r]^-{\smtrx{\eta_{0}^*\\0}}&\oplus\ar[r]^-{\smtrx{\eta_{0}&0}}&V_1\ar[r]^-{\eta_{-2}^*}&\cdots\ar[r]^-{\eta^*_{1-s}}&W_{s-1}\ar[r]^-{\eta_{s-1}}&V_s~,\\&&&&W_0'&}
    \end{equation}
    \item The chain is of the form \eqref{eq: r<s minimal hol chain 1
        chain} with $q=p$, and for some $2$-torsion line bundle $I$, $V_0=I\oplus I$, $V_j=IK^{-j}$ and $W_{j}=IK^{-j}$ for all $j\neq0,$ and the chain is given by 
    \begin{equation}
        \label{eq holomorphic chain 2 stable min thm}\xymatrix@C=2.3em@R=-.4em{&&&&I&\\W_{-r}\ar[r]^-{\eta_{-r}}&V_{1-r}\ar[r]^-{\eta_{r-2}^*}&\cdots\ar[r]^-{\eta_{1}^*}&W_{-1}\ar[r]^-{\smtrx{\eta_{-1}\\0}}&\oplus\ar[r]^-{\smtrx{\eta_{-1}^*&0}}&W_1\ar[r]^-{\eta_{1}}&\cdots\ar[r]^-{\eta_{r-2}}&V_{r-1}\ar[r]^-{\eta_{-r^*}}&W_r~,\\&&&&I&}
    \end{equation}

    \item The chain is of the form \eqref{eq: r<s minimal hol chain 1 chain} with $q=p+1$, $V_j=K^{-j}$ and $W_j=K^{-j}$ for all $|j|<p$ and $W_{-p}$ is a line bundle satisfying $\deg(W_{-p})\in(0,p(2g-2)].$

    \item The chain is of the form \eqref{eq: s<r minimal hol chain 2 chains} where $W_0$ is a stable $\rO(q-p+1,\C)$-bundle with $\det(W_0)=I$, and $V_j=IK^{-j}$ and $W_j=IK^{-j}$ for all $j\neq0.$

    \item The chain is of the form \eqref{eq: r<s minimal hol chain 2 chains} with $q=p+1$, $V_0=0$, $W_0\cong\cO$, $V_j=K^{-j}$ and $W_{j}=K^{-j}$ for $0<|j|<p$ and $W_{-p}$ is a line bundle satisfying $\deg(W_{-p})\in(0,p(2g-2)].$

    \item The chain is of the form \eqref{eq: r<s minimal hol chain 2 chains} with $q=p$, and for some 2-torsion line bundle $I$, $V_j=IK^{-j}$ and $W_{j}=IK^{-j}$ for all $j.$
 \end{enumerate} 
\end{theorem}
\begin{remark}\label{remark replacing invariant stable with polystable}
    Cases (2)-(7) are special cases of the fixed points considered in Lemma \ref{Lem: Fixed points in the image of Psi}. 
   In case (2), the Higgs bundle is still a local minimum of the Hitchin function if the invariant bundle $W_0'$ is strictly polystable.
   Similarly, replacing the stable orthogonal bundle $W_0$ in case (5) with a strictly polystable orthogonal bundle still defines a local minimum. We will prove that these are the only local minima apart from $\eta=0.$ 
    Note also that none of the above cases have $p=1$ and $q>2$. 
 \end{remark}

\begin{proof}
We first show that cases (1) and (2) are sufficient for the chain \eqref{eq: s<r minimal hol chain 1 chain} to be a stable minima with $\HH^2(C^\bullet)=0$ by invoking Proposition \ref{prop: minima criteria}. For case (1), $C^\bullet_{2}$ is the only isomorphism to consider. We have $\fso_2(V)\oplus\fso_2(W)=\Lambda^2 V_1$ and $\Hom_3(W,V)\otimes K=0,$ which is an isomorphism since $\rk(V_{-1})=1.$ For case (2), the holomorphic chain \eqref{eq holomorphic chain stable min thm} is a fixed point considered in Lemma \ref{Lem: Fixed points in the image of Psi} with $W_p=0.$ By Lemma \ref{lemma k>0 H1 vanishing H2 not p 2p}, $C^\bullet_k:\fso_k(V)\oplus\fso_k(W)\xrightarrow{\ad_\eta}\Hom_{k+1}(W,V)\otimes K$ is an isomorphism for all $k>0.$ 
 
We now show that cases (1) and (2) are necessary for chains of the form \eqref{eq: s<r minimal hol chain 1 chain}. We have a chain 
    \[\xymatrix{V_{-s}\ar[r]^-{\eta_{s-1}^*}&W_{1-s}\ar[r]^-{\eta_{1-s}}&\cdots\ar[r]^-{\eta_{-2}}&V_{-1}\ar[r]^-{\eta_0^*}&W_0\ar[r]^-{\eta_0}&V_1\ar[r]^-{\eta_{-2}^*}&\cdots\ar[r]^-{\eta_{1-s}^*}&W_{s-1}\ar[r]^-{\eta_{s-1}}&V_s},\] with $s\geq 1$ odd.
    By Lemma \ref{lemma:linebundles and powers of K} each of the bundles in the chain is a line bundle except $W_0$. So $p=s+1$ is even and $\rk(W_0)=q-p+2\geq 2$. Note that $\cO=\det(V)=\det(W)=\det(W_0).$

If $N=\ker(\eta_0)$, then $\eta_0^*$ maps $V_{-1}$ to  $N^\perp K\subset W_0\otimes K$. By Proposition \ref{prop: stable fixed points}, $\eta_0^*$ is non-zero, hence $\deg(N^\perp)-\deg(V_{-1})+2g-2\geq 0$. If $N$ is coisotropic then $N^\perp$ is isotropic, and  stability implies $\deg(V_{-1})+\deg(N^\perp)<0$, which implies $\deg(V_{-1})<g-1$.
If $N$ is not coisotropic, then $\eta_0\eta_0^*$ is a non-zero section of the line bundle $V_1^2K^2$. Thus, 
\begin{equation}\label{eq:degV_1} 
\deg(V_{-1})\leq 2g-2~.
 \end{equation} 

If $p=2$ and $\deg(V_{-1})< 2g-2$ we are done. If $\deg(V_{-1})=2g-2,$ then $\eta_0\eta_0^*$ is a nowhere vanishing section of the line bundle $V_1^2K^2,$ and hence the kernel of $\eta_0$ is a holomorphic orthogonal bundle $W_0'\subset W_0$ of rank $q-p+1.$ Furthermore, stability of $(V,W,\eta)$ forces $W_0'$ to be stable. Taking orthogonal complements gives a decomposition $W_0=W_0'\oplus I$ where $KV_{1}=I=\det(W_0')$ since $\cO=\det(W_0).$ By Lemma \ref{lemma:linebundles and powers of K}, the holomorphic chain is given by \eqref{eq holomorphic chain stable min thm}. Thus, for $p=2$ we are done. For $p>2$ we will show that stability forces $\deg(V_{-1})=2g-2$ and $V_{-s}=K^sI.$

For $p\geq 4$ and even, we have $s\geq 3$ and odd. Using decompositions analogous to \eqref{Eq End2k splittings section minima} and \eqref{Eq so2k splittings section minima} and $\rk(V_j)=\rk(W_j)=1$ for $j\neq0$, the isomorphism of $C^\bullet_{s-1}$ gives
\[\fso_{s-1}(V)\oplus\fso_{s-1}(W)\cong \displaystyle\bigoplus\limits_{j=0}^{\left\lfloor \frac{s-1}{4}\right\rfloor}\Hom(V_{2j-s},V_{2j-1})\oplus\displaystyle\bigoplus_{j=0}^{\left\lfloor \frac{s-3}{4}\right\rfloor}\Hom(W_{2j+1-s},W_{2j})\]\[\cong\ \Hom_{s}(W,V)\otimes K\ \cong\ \displaystyle\bigoplus\limits_{j=0}^{\frac{s-1}{2}} \Hom(W_{2j+1-s},V_{2j+1})\otimes K~.\]
Since $\det(W_0)=\cO$, the determinant of both sides of the isomorphism $C^\bullet_{s-1}$ is given by
    \begin{equation}\label{eq:proof1 determinant}
    V_sV_{-1}\otimes\bigotimes_{j=1}^{\left\lfloor \frac{s-1}{4}\right\rfloor}V_{s-2j}V_{2j-1}\otimes W_{s-1}^{\rk(W_0)}\otimes\bigotimes\limits_{j=1}^{\left\lfloor \frac{s-3}{4}\right\rfloor} W_{s-1-2j}W_{2j}\cong \bigotimes_{j=0}^\frac{s-3}{2}W_{s-1-2j}V_{2j+1}K\otimes (V_sK)^{\rk(W_0)}.
    \end{equation}
From Lemma \ref{lemma:linebundles and powers of K}, we have $W_{s-1}\cong V_1K^{2-s}$ and
\[V_1^2K^{3-s}\cong \begin{dcases}
    V_{s-2j}V_{2j-1}, &\text{ for }j=1,\ldots,\left\lfloor(s-1)/4\right\rfloor\\ 
    W_{s-1-2j}W_{2j}, &\text{ for }j=1,\ldots,\left\lfloor(s-3)/4\right\rfloor\\
    W_{s-1-2j}V_{2j+1}K, &\text{ for }j=0,\ldots,(s-3)/2.
\end{dcases}\]
This simplifies \eqref{eq:proof1 determinant} to $(V_sV_{-1}K^{s-1})^{p-q-1}\cong (V_1K)^2.$ As in the proof of Lemma \ref{lemma:linebundles and powers of K}, the Higgs field gives a non-zero map $V_1\rightarrow V_s K^{s-1}$. Therefore, 
\[0\geq (p-q-1)(\deg(V_s)-\deg(V_1)+(s-1)(2g-2))=2(\deg(V_1)+2g-2),\]
and hence $\deg(V_1)\leq 2-2g$. By \eqref{eq:degV_1}, we conclude that $\deg(V_{-1})=2g-2$ and $\deg(V_s)=-s(2g-2).$ As above, since $\deg(V_{-1})=2g-2$, the bundle $W_0$ decomposes as $W_0'\oplus I,$ where $W_0'$ is the kernel of $\eta_0$ and $\det(W_0)'=I=V_1K.$ Moreover, we have $V_s=IK^{-s}$ since, by Lemma \ref{lemma:linebundles and powers of K}, $W_{s-1}=IK^{1-s}$ and $\eta_{s-1}:W_{s-1}\to V_{s}\otimes K$ is non-zero. This completes the proof of (2).

Case (3) is similar to case (2), and cases (4) and (6) are almost identical. By Lemma \ref{lemma:linebundles and powers of K}, the holomorphic chain \eqref{eq: r<s minimal hol chain 1 chain} is given by 
\[\xymatrix@C=1.8em@R=0em{W_{-r}\ar[r]^-{\eta_{-r}}&V_{1-r}\ar[r]^-{\eta_{r-2}^*}&\cdots\ar[r]^-{\eta_1^*}&W_{-1}\ar[r]^-{\eta_{-1}}&V_0\ar[r]^-{\eta^*_{-1}}&W_1\ar[r]^-{\eta_1}&\cdots\ar[r]^-{\eta_{r-2}}&V_{r-1}\ar[r]^-{\eta_{-r}^*}&W_r~,}\]
where $\rk(W_j)=1$ for all $j.$ Thus, $r=q-1$ and either $\rk(V_0)=1$ and $q=p+1$ or $\rk(V_0)=2$ and $q=p$. 
If $q=p$, then, the holomorphic chain is identical to case (2) with the roles of $V$ and $W$ switched. By the same argument as case (2), we conclude that the holomorphic chain must be of the form \eqref{eq holomorphic chain 2 stable min thm}.

We now assume $\rk(V_0)=1$ and $q=p+1.$  Moreover, $V_0=\cO$ since $\cO\cong \det(V)\cong V_0$. Since the Higgs field defines a non-zero maps $W_{-1}\to\cO\otimes K$ and $W_{-1}\to W_1\otimes K^2,$ we conclude that $W_{-1}\cong K$. Thus, $W_j=K^{-j}$ and $V_j=K^{-j}$ for all $|j|<r$ by Lemma \ref{lemma:linebundles and powers of K}. Since $W_p$ is an invariant isotopic subbundle and the Higgs field $\eta_{-p}:W_{-p}\to V_{-p+1}K$ is non-zero, we conclude 
\[0<\deg(W_{-p})\leq p(2g-2).\] Thus, the conditions in case (4) are necessary.

The holomorphic chain from case (4) is a fixed point considered in Lemma \ref{Lem: Fixed points in the image of Psi} with $W_0'=0$ and $\rk(W_{-p})=1.$ By Lemmas \ref{lemma k>0 H1 vanishing H2 not p 2p} and \ref{lemma p, 2p H1 vanishing H2}, $C^\bullet_k:\fso_k(V)\oplus\fso_k(W)\xrightarrow{\ad_\eta}\Hom_{k+1}(W,V)\otimes K$ is an isomorphism for all $k>0$. Thus, the conditions in case (4) are also sufficient.

The holomorphic chain from case (5) is a fixed point considered in Lemma \ref{Lem: Fixed points in the image of Psi} with $W_{-p}=0$. By Lemma \ref{lemma k>0 H1 vanishing H2 not p 2p}, $C^\bullet_k:\fso_k(V)\oplus\fso_k(W)\xrightarrow{\ad_\eta}\Hom_{k+1}(W,V)\otimes K$ is an isomorphism for all $k>0$. Thus, the conditions in case (5) are sufficient. 

To show the conditions of (5) are necessary, note that the holomorphic chain \eqref{eq: s<r minimal hol chain 2 chains} is given by 
\[\xymatrix@C=1.8em@R=0em{V_{-r}\ar[r]^-{\eta_{-r}}&W_{1-r}\ar[r]^-{\eta_{r-2}^*}&\cdots\ar[r]^-{\eta_1^*}&W_{-1}\ar[r]^-{\eta_{-1}}&V_0\ar[r]^-{\eta^*_{-1}}&W_1\ar[r]^-{\eta_1}&\cdots\ar[r]^-{\eta_{r-2}}&W_{r-1}\ar[r]^-{\eta_{-r}^*}&V_r~.\\&&&&\oplus\\&&&&W_0}\]
By Lemma \ref{lemma:linebundles and powers of K}, $\rk(V_j)=1$ for all $j$, thus $r=p-1$ and $\rk(W_0)\geq1$. Also, if $V_0=I$, then $I=\det(V)=\det(W)=\det(W_0)$, and $V_j=IK^{-j}$ for all $|j|<p-1$ and $W_j=K^{-1}I$ for all $j\neq0.$ 
Since $W_0\neq0$, $\fso_{p-2}(V)\oplus\fso_{p-2}(W)\cong\Hom(W_{1-r},W_0)$ and $\Hom_{p-1}(W,V)\otimes K\cong \Hom(W_0,V_{p-1}K).$ Taking the determinant of this isomorphism and using $W_{2-p}=K^{p-2}I$ we conclude that $V_{1-p}=IK^{p-1}$, finishing the proof of case (5).

Finally, for case (7) the holomorphic chain \eqref{eq: r<s minimal hol chain 2 chains} is given by 
\[\xymatrix@C=1.8em@R=0em{W_{-s}\ar[r]^-{\eta_{-s}}&V_{1-s}\ar[r]^-{\eta_{s-2}^*}&\cdots\ar[r]^-{\eta_{-2}}&V_{-1}\ar[r]^-{\eta_{0}^*}&W_0\ar[r]^-{\eta_{0}}&V_1\ar[r]^-{\eta_{-2}^*}&\cdots\ar[r]^-{\eta_{s-2}}&V_{s-1}\ar[r]^-{\eta_{-s}^*}&W_s\\&&&&\oplus&&&&\\&&&&V_0&&&&}.\]
By Lemma \ref{lemma:linebundles and powers of K}, $\rk(W_j)=1$ for all $j.$ Thus $s=q-1=p-1$ and $V_0$ is a rank one orthogonal bundle $I$ with $I=\det(V)=\det(W)=W_0,$ $V_j=IK^{-j}$ for all $j$ and $W_j=IK^{-j}$ for all $|j|<s.$ Similar to case (4), we have $\fso_{p-2}(V)\oplus\fso_{p-2}(W)\cong\Hom(V_0,V_{p-2})$ and $\Hom_{p-1}(W,V)\otimes K\cong\Hom(W_{1-p},V_0K).$ Thus, the isomorphism $C^\bullet_{p-2}$ implies $W_{1-p}\cong IK^{p-1}.$ Thus, the conditions of (7) are necessary.  
As with the other cases, the conditions of case (7) are sufficient by Lemmas \ref{Lem: Fixed points in the image of Psi} and \ref{lemma k>0 H1 vanishing H2 not p 2p}.
\end{proof}

\subsection{Stable minima with non-vanishing $\HH^2(C^\bullet)$}
We now classify stable $\SO(p,q)$-Higgs bundles such that the associated $\SO(p+q,\C)$-Higgs bundle is strictly polystable. By Remark \ref{remark duality of H0 and H2 for complex groups}, these are exactly the stable $\SO(p,q)$-Higgs bundles which may have $\HH^2(C^\bullet)\neq0$. 
We will prove that such $\SO(p,q)$-Higgs bundles define minima of the Hitchin function if and only if the Higgs field $\eta$ is zero. 

The $\SO(p+q,\C)$-Higgs bundle associated to an $\SO(p,q)$-Higgs bundle $(V,Q_V,W,Q_W,\eta)$ is
\begin{equation}
    \label{eq:SO(p+q,C),quad,Higgs} (E,Q,\Phi)=\Big(V\oplus W,\smtrx{Q_V&0\\0&-Q_W},\smtrx{0&\eta\\\eta^*&0}\Big)~.
 \end{equation} 
Recall that a $\GL(p,\R)$-Higgs bundle is defined by a triple $(V,Q_V,\eta)$ where $(V,Q_V)$ is a rank $p$ orthogonal vector bundle and $\eta:V\to V\otimes K$ is a holomorphic map satisfying $\eta^*=Q_V^{-1}\eta^T Q_V=\eta$. 
Given such a $\GL(p,\R)$-Higgs bundle, we construct the $\SO(p,p)$-Higgs bundle $(V,Q_V,V,Q_V,\eta)$. Using the symmetry $\eta^*=\eta$, the corresponding $\SO(2p,\C)$-Higgs bundle is 
\[(E,Q,\Phi)=\Big(V\oplus V,\smtrx{Q_V&0\\0&-Q_V},\smtrx{0&\eta\\\eta&0}\Big)~.\]
Even if the $\SO(p,p)$-Higgs bundle $(V,V,\eta)$ is stable, the above $\SO(2p,\C)$-Higgs bundle is strictly polystable. Indeed, the following pair of disjoint degree zero isotropic subbundles are both $\Phi$-invariant
\[\xymatrix@R=.1em{V\ar[r]^{i_1\ \ }&V\oplus V&&\text{and}&&V\ar[r]^{i_2\ \ }&V\oplus V\\v\ar@{|->}[r]&(v,v)&&&&v\ar@{|->}[r]&(v,-v)}.\]
The following proposition shows that this example characterizes the stable $\SO(p,q)$-Higgs bundles which are not stable as $\SO(p+q,\C)$-Higgs bundles. 

\begin{proposition}\label{prop:comparison stability SO(p,q) SO(p+q,C)}
Let $(V,W,\eta)$ be a stable $\SO(p,q)$-Higgs bundle. The associated $\SO(p+q,\C)$-Higgs bundle \eqref{eq:SO(p+q,C),quad,Higgs} is strictly polystable if and only if 
\begin{equation}\label{eq:splittingstable-SOnotstable}
(V,Q_V,W,Q_W,\eta)\cong \Big(V_1\oplus V_2, \smtrx{Q_{V_1}&0\\0&Q_{V_2}}, V_1\oplus W_2, \smtrx{Q_{V_1}&0\\0&Q_{W_2}},\smtrx{\eta_1&0\\0&\eta_2}\Big)~,
\end{equation} 
where $(V_1,Q_{V_1},V_1,Q_{V_1},\eta_1)$ is a stable $\SO(p_1,p_1)$-Higgs bundle with $\eta_1^*=\eta_1$ and $(V_2,Q_{V_2},W_2,Q_{W_2},\eta_2)$ is a stable $\SO(p_2,q_2)$-Higgs bundle.
\end{proposition}
\begin{proof}
By the above discussion, the condition \eqref{eq:splittingstable-SOnotstable} is sufficient. We now show that it is necessary. 
Let $(V,W,\eta)$ be a stable $\SO(p,q)$-Higgs bundle and suppose the associated $\SO(p+q,\C)$-Higgs bundle $(E,Q,\Phi)$ given by \eqref{eq:SO(p+q,C),quad,Higgs} is strictly polystable, so that there is a degree zero proper subbundle $U\subset V\oplus W$, which is isotropic with respect to $Q$ and satisfies $\Phi(U)\subset U\otimes K$.
Let $V_1\subset V$ and $W_1\subset W$ be the respective image sheaves of the projection of $U$ onto each summand of $V\oplus W.$ The subsheaf $V_1\oplus W_1$ is preserved by $\Phi$, thus $\deg(V_1)+\deg(W_1)\leq 0$ by polystability of the associated $\SL(p+q,\C)$-Higgs bundle $(V\oplus W,\Phi).$ 

Consider the sequences
\[\xymatrix{0\ar[r]&U^w\ar[r]&U\ar[r]&V_1\ar[r]&0&\text{and}&0\ar[r]&U^v\ar[r]&U\ar[r]&W_1\ar[r]&0}~,\]
where the subsheaf $U^v\subset V$ is $Q_V$-isotropic, the subsheaf $U^w\subset W$ is $Q_W$-isotropic, $\eta(U^w)\subset U^v\otimes K$ and $\eta^*(U^v)\subset U^w\otimes K.$ Stability of the $\SO(p,q)$-Higgs bundle gives $\deg(U^v)+\deg(U^w)<0$, which implies $\deg(V_1)+\deg(W_1)>0.$ But, since $V_1\oplus W_1$ is preserved by $\Phi$, $\deg(V_1)+\deg(W_1)\leq 0$ by polystability of the Higgs bundle $(V\oplus W,\Phi).$ This contradiction implies
\[V_1\cong U\cong W_1.\] 

We claim that $V_1$ and $W_1$ are both orthogonal subbundles. Let $Q_{V_1}$ and $Q_{W_1}$ be the restrictions of $Q_V$ and $Q_W$ to $V_1$ and $W_1$ respectively. Consider the following sequences
\[\xymatrix@=1.5em{0\ar[r]&V_1^{\perp_{V_1}}\ar[r]&V_1\ar[r]&V_1/V_1^{\perp_{V_1}}\ar[r]&0&\text{and}&0\ar[r]&W_1^{\perp_{W_1}}\ar[r]&W_1\ar[r]&W_1/W_1^{\perp_{W_1}}\ar[r]&0},\]
where, by definition, $V_1^{\perp_{V_1}}=V_1^{\perp_V}\cap V_1$ and $W_1^{\perp_{W_1}}=W_1^{\perp_W}\cap W_1$. 
Since $V_1^{\perp_{V_1}}$ and $W_1^{\perp_{W_1}}$ are maximally isotropic subbundles of $V_1$ and $W_1$ respectively, both $V_1/V_1^{\perp_{V_1}}$ and $W_1/W_1^{\perp_{W_1}}$ are orthogonal bundles. 
In particular, $V_1^{\perp_{V_1}}$ and $W_1^{\perp_{W_1}}$ are degree zero isotropic subbundles of $V$ and $W$ respectively. Moreover, since $\eta^*=Q_W^{-1}\eta^TQ_V$, we have $Q_V(\eta(-),-)=Q_W(-,\eta^*(-))$. This, together with the fact that $\eta(W_1)\subset V_1\otimes K$ and $\eta^*(V_1)\subset W_1\otimes K$ (by $\Phi$-invariance of $U$), shows that 
\[\xymatrix{\eta(W_1^{\perp_{W_1}})\subset V_1^{\perp_{V_1}}\otimes K&\text{and}&\eta^*(V_1^{\perp_{V_1}})\subset W_1^{\perp_{W_1}}\otimes K.}\]
Again, stability of the $\SO(p,q)$-Higgs bundle $(V,W,\eta)$ implies both $V_1^{\perp_{V_1}}$ and $W_1^{\perp_{W_1}}$ are zero, which implies $V_1\subset V$ and $W_1\subset W$ are both orthogonal subbundles.

If $p_1=\rk(W_1)=\rk(V_1)$, then $(V_1,W_1,\eta|_{W_1})$ is an $\SO(p_1,p_1)$-Higgs bundle. Note that isomorphism between $V_1$ and $W_1$ is given by including $V_1$ into $V\oplus W$ and projecting onto $W.$ 
Denoting this isomorphism by $g:V_1\to W_1$, we have $\eta|_{W_1} g=(g^{-1}\otimes \Id_K)\eta|_{W_1}^*$. Moreover, $g$ is orthogonal since for any $x,y\in V_1$ we have $(x,g(x)),(y,g(y))\in U$, and   
\[0=Q((x,g(x)),(y,g(y)))=Q_{V_1}(x,y)-Q_{W_1}(g(x),g(y))\]
since $U$ is isotropic.
Therefore the pair $(\Id_V,g^{-1})$ defines an isomorphism between $(V_1,W_1,\eta|_{W_1})$ and $(V_1,V_1,\eta_1)$ with $\eta_{1}=\eta|_{W_1}g$. In particular, $\eta_1=\eta_1^*$.

Let $V_2$ and $W_2$ be the orthogonal complements of $V_1$ and $W_1$ respectively and let $\eta_2:W_2\to V_2\otimes K$ be the restriction of $\eta$ to $W_2$. By the above discussion, we obtain the desired decomposition \eqref{eq:splittingstable-SOnotstable} of the $\SO(p,q)$-Higgs bundle $(V,W,\eta)$.
\end{proof}

If a stable $\SO(p,q)$-Higgs bundle
\[(V,Q_V,W,Q_W,\eta)\cong \Big(V_1\oplus V_2, \smtrx{Q_{V_1}&0\\0&Q_{V_2}}, V_1\oplus W_2, \smtrx{Q_{V_1}&0\\0&Q_{W_2}},\smtrx{\eta_1&0\\0&\eta_2}\Big),\]
with $\eta_1^*=\eta_1$, is a local minimum of the Hitchin function, then $(V_1,Q_{V_1},\eta_1)$ is a local minimum of the Hitchin function on the $\GL(p_1,\R)$-Higgs bundle moduli space and $(V_2,Q_{V_2},W_2,Q_{W_2},\eta_2)$ is a local minimum of the Hitchin function on the $\SO(p_2,q_2)$-Higgs bundle moduli space.

For $p>1$, the local minima in the $\GL(p,\R)$-Higgs bundle moduli space with non-zero Higgs field are described in Example \ref{ex: minima for GL(n,R) and Upq}. When $p=2$, such local minima are of the form \eqref{eq:SL2}, and hence the holomorphic chain of the corresponding $\SO(2,2)$-Higgs bundle is of the form
\begin{equation}\label{eq:minS022}
\vcenter{   \xymatrix@R=0em{L\ar[r]^-{\Phi_1}&L^{-1}K\\
\hspace{2.5cm}\oplus&\\
L\ar[r]_-{\Phi_1}&L^{-1}K.}}
\end{equation}
When a stable $\SO(p_2,q_2)$-Higgs bundle is added to \eqref{eq:minS022}, the resulting $\SO(2+p_2,2+q_2)$-Higgs bundle is not stable since $L$ and $L^{-1}$ are a pair of proper isotropic subbundles exchanged by the Higgs field. 

For $p\geq 3$, the holomorphic chain associated to a $\GL(p,\R)$ local minimum is given by 
\[\xymatrix{V_{\frac{1-p}{2}}\ar[r]^-{1}&V_{\frac{3-p}{2}}\ar[r]^-{1}&\cdots\ar[r]^-{1}&V_{\frac{p-3}{2}}\ar[r]^-{1}&V_{\frac{p-1}{2}}},\]
where $V_j=IK^{-j}$ for all $j$ and some $2$-torsion line bundle $I$. The holomorphic chain of the associated $\SO(p,p)$-Higgs bundle is 
\begin{equation}
    \label{eq so(p,p) of an GL(p) minima}
\vcenter{   \xymatrix@R=0em{V_{\frac{1-p}{2}}\ar[r]^-{1}&V_{\frac{3-p}{2}}\ar[r]^-{1}&\cdots\ar[r]^-{1}&V_{\frac{p-3}{2}}\ar[r]^-{1}&V_{\frac{p-1}{2}}\\
&&\oplus&\\
V_{\frac{1-p}{2}}\ar[r]^-{1}&V_{\frac{3-p}{2}}\ar[r]^-{1}&\cdots\ar[r]^-{1}&V_{\frac{p-3}{2}}\ar[r]^-{1}&V_{\frac{p-1}{2}}.}}
\end{equation}
By Proposition \ref{prop: stable fixed points}, such an $\SO(p,p)$-Higgs bundle is not stable if $p\geq 4 $ is even.

By the above discussion, the potential stable
  $\SO(p,q)$ local minima which are not stable as $\SO(p+q,\C)$-Higgs bundles have the form \eqref{eq:splittingstable-SOnotstable}, where $(V_2,W_2,\eta_2)$ is a stable $\SO(p_2,q_2)$ local minimum and $(V_1,V_1,\eta_1)$ is either a stable $\SO(p_1,p_1)$-Higgs bundle with $\eta_1=0$ or $p_1>2$ is odd and $(V_1,V_1,\eta_1)$ is a holomorphic chain of the form \eqref{eq so(p,p) of an GL(p) minima}. The next two propositions address these cases.
\begin{proposition}\label{prop H^2 not zero not a min}
    For $p\geq 3$ odd, the stable $\SO(p,p)$-Higgs bundle given by \eqref{eq so(p,p) of an GL(p) minima} with $V_j=IK^{-j}$ for all $j$ and some $2$-torsion line bundle $I$ is not a minimum of the Hitchin function. 
\end{proposition}
\begin{proof}
    By assumption $r=\frac{p-1}{2}$ is a positive integer. Set $V=\bigoplus\limits_{j=0}^{2r}V_{j-r}$ and $W=\bigoplus\limits_{j=0}^{2r}W_{j-r}$ with $V_j=IK^{-j}$ and $W_j=IK^{-j}$ for all $j$ and some $2$-torsion line bundle $I$. The holomorphic chain \eqref{eq so(p,p) of an GL(p) minima} is given by
    \[  \xymatrix@R=0em{V_{-r}\ar[r]^-{1}&W_{1-r}\ar[r]^-{1}&\cdots\ar[r]^-{1}&V_{-1}\ar[r]^-{1}&W_0\ar[r]^-{1}&V_1\ar[r]^-{1}&\cdots\ar[r]^-{1}&V_{r-1}\ar[r]^-{1}&W_r\\
&&&&\oplus&\\
W_{-r}\ar[r]^-{1}&V_{1-r}\ar[r]^-{1}&\cdots\ar[r]^-{1}&W_{-1}\ar[r]^-{1}&V_0\ar[r]^-{1}&W_1\ar[r]^-{1}&\cdots\ar[r]^-{1}&W_{r-1}\ar[r]^-{1}&V_{r}}\]


Let $\beta\in\Omega^{0,1}(K^{1-2r})$ which is non-zero in cohomology and, with respect to the above splittings of $V$ and $W,$ consider the deformed orthogonal holomorphic structures:
\[\xymatrix{\dbar_V^\beta=\smtrx{\dbar_{K^{r}}&&&&\\0&\dbar_{K^{r-1}}&&&\\&&\ddots&\\-\beta&0&\cdots&\dbar_{K^{1-r}}&\\0&\beta^*&\cdots&0&\dbar_{K^{-r}}}
&\text{and}&\dbar_W^\beta=\smtrx{\dbar_{K^{r}}&&&&\\0&\dbar_{K^{r-1}}&&&\\&&\ddots&\\\beta&0&\cdots&\dbar_{K^{1-r}}&\\0&-\beta^*&\cdots&0&\dbar_{K^{-r}}}}~.\]
In the above splittings of $V$ and $W$, the Higgs field is given by 
\[\eta=\smtrx{0&&&\\1&0&&\\&\ddots&\ddots&\\&&1&0}:W\to V\otimes K,\]
and a calculation shows that $\eta$ is still holomorphic with respect $\dbar_W^\beta$ and $\dbar_V^\beta$. So $\bar\partial^\beta_V$, $\bar\partial^\beta_W$, together with the corresponding orthogonal structures, and $\eta$, define an $\SO(p,p)$-Higgs bundle $(\mathbb{V},\mathbb{W},\eta)$. 
Since \eqref{eq so(p,p) of an GL(p) minima} is stable, and stability is an open condition,  $(\mathbb{V},\mathbb{W},\eta)$ is also stable. Moreover, since $\beta$ is non-zero in cohomology, $(\mathbb{V},\mathbb{W},\eta)$ is not isomorphic to \eqref{eq so(p,p) of an GL(p) minima}, so it is not $\cS$-equivalent to it. 

Consider the following orthogonal gauge transformations of $V$ and $W$
\[g_t^V=g_t^W=\smtrx{t^r &&&\\ &t^{r-1}&&\\ &&\ddots&\\&&&t^{-r}}.\] 
For each $t\in\mathbb{C}^*$, we see that 
\[g^V_t\bar\partial^\beta_V(g^V_t)^{-1}=\smtrx{\dbar_{K^{r}}&&&&\\0&\dbar_{K^{r-1}}&&&\\&&\ddots&\\-t^{1-2r}\beta&0&\cdots&\dbar_{K^{1-r}}&\\0&t^{1-2r}\beta^*&\cdots&0&\dbar_{K^{-r}}}=\dbar^{t^{1-2r}\beta}_V,\]
$g^W_t\bar\partial^\beta_W(g^W_t)^{-1}=\dbar^{t^{1-2r}\beta}_W$ and $g^V_t(t\eta )(g_t^V)^{-1}=\eta$. Thus $\lim_{t\to\infty}(\mathbb{V},\mathbb{W},t\eta)$ is equal to the Higgs bundle given by \eqref{eq so(p,p) of an GL(p) minima}. By Proposition \ref{prop:notS-equiv-liminfty-notmin} we conclude that \eqref{eq so(p,p) of an GL(p) minima} is not a local minimum.
\end{proof}


\begin{proposition}\label{prop:stableH2neq0-zeroGL(p)}
Let $(V,W,\eta)$ be a stable $\SO(p,q)$-Higgs bundle of the form
\[(V,Q_V,W,Q_W,\eta)\cong \Big(U\oplus V', \smtrx{Q_{U}&0\\0&Q_{V'}}, U\oplus W', \smtrx{Q_{U}&0\\0&Q_{W'}},\smtrx{0&0\\0&\eta'}\Big)\] 
where $(U,Q_{U},U,Q_{U},0)$ is a stable $\mathrm{S}(\mathrm{O}(p_1)\times\mathrm{O}(p_1))$-Higgs bundle and $(V',Q_{V'},W',Q_{W'},\eta')$ is a stable $\SO(p_2,q_2)$  local minimum from Theorem \ref{thm:stablenon-zerominH2=0}. Then $(V,W,\eta)$ is not a local minimum.
\end{proposition}
\begin{proof}
Suppose that $(V',Q_{V'},W',Q_{W'},\eta')$ is a minimum of type (1) from Theorem \ref{thm:stablenon-zerominH2=0}. Then $(V,W,\eta)$ can be represented by 
\[\xymatrix@R=-.3em{&U&\\&\oplus&\\&U&\\&\oplus&\\V_{-1}\ar[r]^-{\eta_{0}^*}&W_0\ar[r]^-{\eta_0}&V_1.}\]
Since, $\deg(V_{-1})>0,$ we have $H^1(\Hom(U,V_1))\neq0$ by Riemann-Roch. Hence, $\alpha\in H^1(\Hom(U,V_1)) 0$ defines a rank $p$ holomorphic orthogonal bundle $\mathbb{V}$. In the $C^\infty$ splitting $V_{-1}\oplus U\oplus V_1$ the $\dbar$-operator is
\[\bar{\partial}_{\mathbb{V}}=\smtrx{\bar\partial_{V_{-1}}&&\\-\alpha^*&\bar\partial_U&\\&\alpha&\bar\partial_{V_{1}}}.\] 

Analogously to the previous proposition, $(\mathbb{V},W,\eta)$, with $\eta|_U=0$ and $\eta|_{W_0}:W_0\to V_1\otimes K\hookrightarrow\mathbb{V}\otimes K$ is a stable $\SO(p,q)$-Higgs bundle which is not isomorphic to $(V,W,\eta)$. Using the gauge transformations $g^V_t=\smtrx{t&&\\&\Id_U&\\&&t^{-1}}$ and $g^W_t=\Id_W$, one computes $\lim_{t\to\infty}(\mathbb{V},W,t\eta)=(V,W,\eta)$. By Proposition \ref{prop:notS-equiv-liminfty-notmin} $(V,W,\eta)$ is not a local minimum. For the other types of local minima from Theorem \ref{thm:stablenon-zerominH2=0}, the argument is similar. Namely, one can take the summand of $U\oplus U$ in $V$ or in $W$ according to where the highest weight summand of the minimum $(V',Q_{V'},W',Q_{W'},\eta')$ lies.
\end{proof}
We conclude that the only stable $\SO(p,q)$-local minima with $\HH^2(C^\bullet)\neq 0$ have vanishing Higgs field, and that the stable local minima with non-zero Higgs field are classified by Theorem \ref{thm:stablenon-zerominH2=0}.

\subsection{Strictly polystable minima}
Recall from Proposition \ref{Prop: strictly polystable SOpq} that a strictly polystable $\SO(p,q)$-Higgs bundle is isomorphic to 
\[\Big(E\oplus E^*\oplus V,\smtrx{0&\Id&0\\ \Id&0&0\\0&0&Q_V},F\oplus F^*\oplus W,\smtrx{0&\Id&0\\ \Id&0&0\\0&0&Q_W},\smtrx{\beta&0&0\\0&\gamma^T&0\\0&0&\eta}\Big),\]
where $(E,F,\beta,\gamma)$ is a polystable $\U(p_1,q_1)$-Higgs bundle with $\deg(E)+\deg(F)=0$, and $(V,W,\eta)$ is a stable $\SO(p-2p_1,q-2q_1)$-Higgs bundle. Here $0\leq p_1\leq p/2$, $0\leq q_1\leq q/2$, and $(p_1,q_1)\neq (0,0)$.

\begin{proposition}\label{Prop Upq in SO2p2q a min iff pq=1 or pq=0}
    Let $(E,F,\beta,\gamma)$ be a polystable $\U(p,q)$-Higgs bundle with $\deg(E)+\deg(F)=0$ which is a local minimum in $\cM(\U(p,q))$. The associated strictly polystable $\SO(2p,2q)$-Higgs bundle 
    \begin{equation}
        \label{eq SO(2p,2q) of U(p,q)}
        \Big(E\oplus E^*,\smtrx{0&\Id\\ \Id&0},F\oplus F^*,\smtrx{0&\Id\\ \Id&0},\smtrx{\beta&0\\0&\gamma^T}\Big)
    \end{equation}
    is a local minimum of the Hitchin function if and only if $\beta=\gamma=0$ or $p\leq 1$ or $q\leq 1.$
\end{proposition}
\begin{proof}
If $\beta=\gamma=0$, the Higgs field is identically zero and we have a minimum. In particular, if $p=0$ or $q=0$ we have $\beta=\gamma=0.$
    Now suppose $p,q>0$ and that the $\SO(2p,2q)$-Higgs bundle \eqref{eq SO(2p,2q) of U(p,q)} is a local minimum with non-zero Higgs field. Then the $\U(p,q)$-Higgs bundle $(E,F,\beta,\gamma)$ is a local minimum in $\cM(\U(p,q))$. Thus, either $\beta=0$ or $\gamma=0$ (cf. Example \ref{ex: minima for GL(n,R) and Upq}). Up to switching the roles of $E$, $F$, $E^*$ and $F^*$, the relevant holomorphic chain for the $\SO(2p,2q)$-Higgs bundle is 
        \begin{equation}
    \label{eq chain for Upq min}
    \xymatrix@R=-.2em{&E&\\F\ar[r]^{\smtrx{\beta\\0}}&\oplus\ar[r]^{\smtrx{0&\beta^T}}&F^*.\\&E^*&}
\end{equation}
    Since the $\U(p,q)$-Higgs bundle $(E,F,\beta,0)$ is polystable and $\beta\neq 0$, we must have $\deg(E)<0<\deg(F)$. For $p=1$ or $q=1$, the associated $\SO(2p,2q)$-Higgs bundle is a local minimum by Proposition \ref{prop:SO2q-absolute-minimum}.

We now show that \eqref{eq SO(2p,2q) of U(p,q)} is not a local minimum if $p,q>1$ and $\beta,\gamma$ not both zero. First assume $(E,F,\beta,0)$ is a stable $\U(p,q)$-Higgs bundle. Consider the chain \eqref{eq chain for Upq min}. By stability, $\deg(F^*)$ and $\deg(E)$ are both negative.
When $p,q>1,$ Riemann-Roch implies there exist $\alpha\in H^1(\Lambda^2F^*)\setminus 0$ and $\sigma\in H^1(\Lambda^2E)\setminus 0$. These classes define holomorphic orthogonal bundles with $\dbar$-operators
\[\xymatrix{\bar\partial_V=\smtrx{\bar\partial_E & \sigma\\ 0 & \bar\partial_{E^*}}&\text{and}&\bar\partial_W=\smtrx{\bar\partial_{F^*} & \alpha\\ 0 & \bar\partial_F}}.\]
Define the Higgs field $\eta:W\to V\otimes K$ by the composition $W\to F\xrightarrow{\beta} E\otimes K\to V\otimes K$. Since semistability is an open condition, $(V,W,\eta)$ is a semistable $\SO(2p,2q)$-Higgs bundle. Furthermore, since $(E,F,\beta,0)$ is stable, the pairs of isotropic subbundles $(E,F)$ and $(E^*,F^*)$ are the only destabilizing pairs of \eqref{eq chain for Upq min}. However, these are not destabilizing pairs of $(V,W,\eta)$ because $F$ and $E^*$ are not subbundles, hence $(V,W,\eta)$ and \eqref{eq chain for Upq min} are not $\cS$-equivalent. For each $t\in\C^*$, the gauge transformations 
\[\xymatrix{g^V_t=\smtrx{t^{-\frac{1}{2}} \Id_E&\\&t^{\frac{1}{2}}\Id_{E^*}}&\text{and}&g_t^W=\smtrx{t^{-\frac{1}{2}}\Id_{F^*} &\\&t^{\frac{1}{2}}\Id_F}}\] act as 
\[(\bar\partial_V,\bar\partial_W,t\eta)\mapsto\left(\smtrx{\bar\partial_E& t^{-1}\sigma\\ 0 &\bar\partial_{E^*}},\smtrx{\bar\partial_{F^*}& t^{-1}\alpha\\ 0 &\bar\partial_F},\eta\right).\] Thus the limit $\lim_{t\to\infty}(V,W,t\eta)$ is isomorphic to \eqref{eq chain for Upq min}. By Proposition \ref{prop:notS-equiv-liminfty-notmin}, \eqref{eq chain for Upq min} is not a minimum when $(E,F,\beta,0)$ is a stable $\U(p,q)$-Higgs bundle.


Finally, let $(E,F,\beta,0$) be a strictly polystable $\U(p,q)$-Higgs bundle. The integer $\tau=|\rk(F)\deg(E)-\rk(E)\deg(F)|$ satisfies $\tau\leq p(2g-2)$ \cite[Theorem A]{UpqHiggs}. For $p\leq q$ and $\tau<p(2g-2)$ or $p=q$ and $\tau=p(2g-2)$, it follows from \cite[Theorem 5.1]{chains-2018} that there is a path $\varepsilon:[0,1]\to \cM(\U(p,q))$ such that $\varepsilon(0)=(E,F,\beta,0)$ and $\varepsilon(t)$ is a stable $\U(p,p)$-Higgs bundle which is a minimum in $\cM(\U(p,q))$ for all $t>0$. By the previous argument, the $\SO(2p,2p)$-Higgs bundle corresponding to $\varepsilon(t)$ is not a minimum for $t> 0$. Hence, it is also not a minimum when $t=0$. 
For $p<q$ and $\tau=p(2g-2)$, \cite[Theorem B]{UpqHiggs} implies that $(E,F,\beta,0)$ is isomorphic to $(E,F',\beta',0)\oplus F''$, where $(E,F',\beta',0)$ is a polystable $\U(p,p)$-local minimum. By the previous arguments, we conclude that \eqref{eq SO(2p,2q) of U(p,q)} is not a local minimum in $\cM(\SO(2p,2q))$ when $p,q>1$.
    \end{proof}
    

The next proposition shows that adding a stable
  $\SO(p,q)$ local minimum from Theorem
  \ref{thm:stablenon-zerominH2=0} to a certain local minimum from Proposition \ref{Prop Upq in SO2p2q a min iff pq=1 or pq=0} is not a local minimum.

\begin{proposition}\label{Prop U11 not min in SOpq unless p=q=2}
    Let $(E,F,\beta,\gamma)$ be a polystable $\U(m,n)$-Higgs bundle with $\deg(E)+\deg(F)=0$. Suppose that either $m=1$, $\beta=0$ and $\gamma\neq 0$, or $n=1$, $\gamma=0$ and $\beta\neq 0$. If $(V',W',\eta')$ is a stable $\SO(p,q)$-local minimum with $\eta'\neq0$, then the $\SO(p+2m,q+2n)$-Higgs bundle
    \[(V,Q_V,W,Q_W,\eta)=\Big(E\oplus E^*\oplus V', \smtrx{0&\Id&0\\ \Id&0&0\\0&0&Q_{V'}},F\oplus F^*\oplus W', \smtrx{0&\Id&0\\ \Id&0&0\\0&0&Q_{W'}}, \smtrx{\beta&0&0\\0&\gamma^T&0\\0&0&\eta'}\Big)\]
    is not a local minimum. 
  \end{proposition}
\begin{proof}
Up to switching the roles of $E,$ $V'$, $F,$ and $W'$, it suffices to consider holomorphic chains of one of the following six types:
    \begin{equation}
        \label{eq hol chains case 1} \xymatrix@R=-.2em{&F&\\E\ar[r]^-{\smtrx{\gamma\\0}}&\oplus\ar[r]^-{\smtrx{0&\gamma^T}}&E^*\\&F^*&\\&\oplus&\\V_{-1}\ar[r]^-{\eta_{0}^*}&W_0\ar[r]^-{\eta_0}&V_1}\ \ \ \ \ \ \ \ \ \ \ \ \xymatrix@R=-.2em{\\ \\ \\\text{or}\\\\}\ \ \ \ \ \ \ \ \ \ \ \ \ \xymatrix@R=-.2em{&E&\\F\ar[r]^-{\smtrx{\beta\\0}}&\oplus\ar[r]^-{\smtrx{0&\beta^T}}&F^*\\&E^*\\&\oplus&\\V_{-1}\ar[r]^-{\eta_{0}^*}&W_0\ar[r]^-{\eta_0}&V_1}
    \end{equation}
    where $\rk(V_{-1})=1$ and $0<\deg(V_{-1})\leq 2g-2$; 
    \begin{equation}
            \label{eq hol chains case 2}
            \xymatrix@C=1.5em@R=-.2em{&&F&\\&E\ar[r]^-{\smtrx{\gamma\\0}}&\oplus\ar[r]^-{\smtrx{0&\gamma^T}}&E^*\\&&F^*&\\&&\oplus&\\V_{1-p}\ar[r]^-{1}&W_{2-p}\ar[r]^-{1}&\cdots\ar[r]^-{1}&W_{p-2}\ar[r]^-{1}&V_{p-1}}\ \  \xymatrix@R=-.2em{\\ \\ \\\text{or}\\\\}\ \ \xymatrix@C=1.5em@R=-.2em{&&E&\\&F\ar[r]^-{\smtrx{\beta\\0}}&\oplus\ar[r]^-{\smtrx{0&\beta^T}}&F^*\\&&E^*\\&&\oplus&\\V_{1-p}\ar[r]^-{1}&W_{2-p}\ar[r]^-{1}&\cdots\ar[r]^-{1}&W_{p-2}\ar[r]^-{1}&V_{p-1}}
    \end{equation}
    where $V_j=IK^{-j}$ and $W_j=IK^{-j}$ for all $j$ and some $I$ with $I^2\cong\cO$;
    \begin{equation}
            \label{eq hol chains case 3}
            \xymatrix@C=1.7em@R=-.2em{&&F&\\&E\ar[r]^-{\smtrx{\gamma\\0}}&\oplus\ar[r]^-{\smtrx{0&\gamma^T}}&E^*\\&&F^*&\\&&\oplus&\\W_{-p}\ar[r]^-{\eta_{-p}}&V_{1-p}\ar[r]^-{1}&\cdots\ar[r]^-{1}&V_{1-p}\ar[r]^-{\eta_{-p}^*}&W_{p}}\ \  \xymatrix@R=-.2em{\\ \\ \\\text{or}\\\\}\ \ \  \xymatrix@C=1.7em@R=-.2em{&&E&\\&F\ar[r]^-{\smtrx{\beta\\0}}&\oplus\ar[r]^-{\smtrx{0&\beta^T}}&F^*\\&&E^*\\&&\oplus&\\W_{-p}\ar[r]^-{\eta_{-p}}&V_{1-p}\ar[r]^-{1}&\cdots\ar[r]^-{1}&V_{1-p}\ar[r]^-{\eta_{-p}^*}&W_{p}}
    \end{equation}
    where $V_j=K^{-j}$ and $W_j=K^{-j}$ for all $|j|<p$, $\rk(W_{-p})=1$, $0<\deg(W_{-p})\leq p(2g-2)$ and $\eta_{-p}\neq0.$

Furthermore, in \eqref{eq hol chains case 1}, \eqref{eq hol chains case 2} and \eqref{eq hol chains case 3}, the first chain has $m=1$, $n>0$, $\deg(F)\leq0\leq\deg(E)$ and $\gamma\neq 0$, while the second chain has $n=1$, $m>0$, $\deg(E)\leq0\leq\deg(F)$ and $\beta\neq 0$.
 We will show that each of the above holomorphic chains is not a minimum.  As in the proof of Proposition \ref{Prop Upq in SO2p2q a min iff pq=1 or pq=0},  we may assume the $\U(m,n)$-Higgs bundle is stable by the results of \cite{chains-2018,UpqHiggs}.

 Since $\Hom(E,V_1)$ is in the kernel of $\ad_\eta:\fso(V)\oplus \fso(W)\to \Hom(W,V)\otimes K$, we may use $\alpha\in H^1(\Hom(E,V_1))\setminus 0$ to deform the holomorphic structure on $V$ by considering non-zero extension 
\[\xymatrix{0\to V_1\to \widetilde V\to E\to 0&\text{and}&0\to E^*\to \widetilde V^*\to V_{-1}\to 0}.\]
Namely, $\tilde V\oplus \tilde V^*$ is a rank $p$ holomorphic orthogonal bundle. Defining $\tilde\eta:F\oplus F^*\oplus W_0\to \tilde V\oplus \tilde V^*$ by the compositions $W_0\xrightarrow{\eta_0} V_{1}\otimes K\to \tilde V\otimes K$ and $F^*\xrightarrow{\gamma^T} E^*\otimes K\to\tilde V^*\otimes K$ gives a semistable $\SO(p,q)$-Higgs bundle $(\tilde V\oplus \tilde V^*,F\oplus F^*\oplus W_0,\tilde\eta)$. 
However, this semistable Higgs bundle is $\cS$-equivalent to the original Higgs bundle. To fix this, we also deform $F\oplus F^*\oplus W_0$.

First assume $\rk(F)=1$. Since $\gamma^T$ is non-zero, we have a short exact sequence 
\[0\to F^*\xrightarrow{\gamma^T}E^*\otimes K\to T\to 0,\]
 where $T$ is a torsion sheaf. 
 Since $V_1\otimes K$ is locally free, this yields the exact sequence 
 \[0\to \Hom(E^*,V_1)\to\Hom(F^*,V_1\otimes K)\to \Hom(T,V_1\otimes K)\to 0,\] which implies that the map 
$H^1(\Hom(E^*,V_1))\to H^1(\Hom(F^*,V_1\otimes K))$, $\sigma\mapsto(\sigma\otimes\Id_K)\gamma^T$ is surjective. 
For any $\delta\in H^1(\Hom(F^*,W_0))\setminus 0$, we have $\eta_0\delta\in H^1(\Hom(F^*,V_1\otimes K))$, and there exists $\sigma\in H^1(\Hom(E^*,V_1))$ such that 
\begin{equation}\label{eq:choice ext classes}
\eta_0\delta-(\sigma\otimes\Id_K)\gamma^T=0
\end{equation} in cohomology.
Let $\mathbb{V}$ and $\mathbb{W}$ be the holomorphic orthogonal bundles, defined respectively by the $C^\infty$ bundles $V_1\oplus E\oplus E^*\oplus V_{-1}$ and $F\oplus W_0\oplus F^*$, together with the $\bar\partial$-operators
\begin{equation}\label{eq:deformationsVandW}
\xymatrix{\bar\partial_{\mathbb{V}}=\smtrx{\bar\partial_{V_1} & \alpha & \sigma &  \\ & \bar\partial_E &  & -\sigma^* \\ &  & \bar\partial_{E^*} & -\alpha^* \\ &  &  & \bar\partial_{V_{-1}}}&\text{and}& \bar\partial_{\mathbb{W}}=\smtrx{\bar\partial_F & -\delta^* &   \\ & \bar\partial_{W_0} & \delta \\ &  & \bar\partial_{F^*}}},
\end{equation}
where $\alpha\in\Omega^{0,1}(\Hom(E,V_1))$, $\sigma\in\Omega^{0,1}(\Hom(E^*,V_1))$ and $\delta\in\Omega^{0,1}(\Hom(F^*,W_0))$ are $(0,1)$-forms representing the cohomology classes $\alpha$, $\sigma$ and $\delta$ respectively. Notice that \eqref{eq:choice ext classes} implies that there is $\epsilon\in\Omega^0(\Hom(F^*,V_1\otimes K))$ so that the representatives $\sigma$ and $\delta$ satisfy 
\begin{equation}\label{eq:choice reps ext classes}
\eta_0\delta-(\sigma\otimes\Id_K)\gamma^T=\epsilon\bar\partial_{F^*}-\bar\partial_{V_1}\epsilon.
\end{equation}
Finally, let $\tilde\eta:\mathbb{W}\to\mathbb{V}\otimes K$ be given, according to the above $C^\infty$ decompositions, by
\[\tilde\eta=\smtrx{0& \eta_0 & -\epsilon  \\0 &0  &0  \\ 0& 0 & \gamma^T\\ 0&0 &0 }.\]
The Higgs field $\tilde\eta$ is holomorphic by \eqref{eq:choice reps ext classes}.
As in the previous propositions, $(\mathbb{V},\mathbb{W},\tilde\eta)$ is a semistable $\SO(p,q)$-Higgs bundle which is not $\cS$-equivalent to the first chain of \eqref{eq hol chains case 1}.  Using the gauge transformations $g^{\mathbb{V}}_t=\smtrx{t^{-1}&&&\\&t^{-\frac{1}{2}}&&\\&&t^{\frac{1}{2}}&\\&&&t}$ and $g^{\mathbb{W}}_t=\smtrx{t^{-3/2}&&\\&\Id_{W_0}&\\&&t^{3/2}}$, one shows that $\lim_{t\to\infty}(\mathbb{V},\mathbb{W},t\tilde\eta)$ is the first chain of \eqref{eq hol chains case 1}. By Proposition \ref{prop:notS-equiv-liminfty-notmin}, the first chain of \eqref{eq hol chains case 1} is not a local minimum.

Now suppose $n>1$. Since $\deg(F)\leq0$, we have $\theta\in H^1(\Lambda^2F)\setminus 0$. Using $\theta$, we can the deformed Higgs bundle
\[\xymatrix{\bar\partial_{\mathbb{V}}=\smtrx{\bar\partial_{V_1}&\alpha&&\\&\bar\partial_E&&\\&&\bar\partial_{E^*}&-\alpha^*\\&&&\bar\partial_{V_{-1}}},& \bar\partial_{\mathbb{W}}=\smtrx{\bar\partial_F&&\theta\\&\bar\partial_{W_0}&\\&&\bar\partial_{F^*}}&\text{ and } &\tilde\eta=\smtrx{0&\eta_0&0\\0&0&0\\0&0&\gamma^T\\0&0&0}},\]  in the $C^\infty$-decompositions $V_1\oplus E\oplus E^*\oplus V_{-1}$ and $F\oplus W_0\oplus F^*$. As above one uses suitably chosen gauge transformations and Proposition \ref{prop:notS-equiv-liminfty-notmin} to conclude that the first chain of \eqref{eq hol chains case 1} is not a local minimum.

An analogous argument, using $W_{p-2}$ and $V_{p-1}$ instead of $W_0$ and $V_1$, can be used to prove that a strictly polystable $\SO(p,q)$-Higgs bundle represented by the first chain of \eqref{eq hol chains case 2} is not a local minimum. The second chain in  \eqref{eq hol chains case 3} is also dealt in a similar manner.

Consider the second chain of \eqref{eq hol chains case 1}. 
Since $\rk(F)=1$ and $\beta\neq 0$, we have a short exact sequence $0\to F\xrightarrow{\beta}E\otimes K\to Q\to 0,$ where $Q$ is the quotient sheaf. 
One sees that the map $H^1(\Hom(E,V_1))\to H^1(\Hom(F,V_1\otimes K))$, $a\mapsto a\beta$ is surjective. 
So, as in the previous case, by picking a non-zero element $c\in H^1(\Hom(F,W_0))$, there exists $a\in H^1(\Hom(E,V_1))$ such that $\eta_0c-a\beta=0$ in cohomology. Given this choice and given a non-zero element $b\in H^1(\Hom(E^*,V_1))$, we construct a non-trivial deformation of the second chain of \eqref{eq hol chains case 1} in a similar manner to the case $\rk(F)=1$ in the first chain of \eqref{eq hol chains case 1}.

An analogous argument can be used to prove that the second chain of \eqref{eq hol chains case 2} and the first chain in \eqref{eq hol chains case 3} are not a local minimum.
\end{proof}

\subsection{Summary of classification of minima of Hitchin function on $\cM(\SO(p,q))$}

Putting everything together, the following theorem classifies all polystable minima of the Hitchin function in the moduli space of $\SO(p,q)$-Higgs bundles for $p\leq q$.

\begin{theorem}\label{thm:MINIMA CLASSIFICATION}
For $1\leq p\leq q$, let $f:\cM(\SO(p,q))\to \R$ be the Hitchin function on the moduli space of polystable $\SO(p,q)$-Higgs bundles given by \eqref{EQ Hitchin Function}.  A polystable $\SO(p,q)$-Higgs bundle $(V,W,\eta)$ is a local minimum of $f$ if and only if $\eta=0$ or $(V,W,\eta)$ is isomorphic to a holomorphic chain of one of the following 
mutually exclusive 
types, where we have suppressed the twisting by $K$ in the Higgs field from the notation:
\begin{enumerate}
    \item $p=2$ and $(V,W,\eta)$ is of the form  
\[\xymatrix{V_{-1}\ar[r]^-{\eta_0^*}&W\ar[r]^-{\eta_0}&V_{1}}~,\]
where $V=V_{-1}\oplus V_1$ with $\rk(V_{-1})=1$ and $0<\deg(V_{-1})< 2g-2$, $V_1=V_{-1}^*$ and $\eta_0$ is non-zero.
\item $p\geq 2$ and $(V,W,\eta)$ is of the form
    \[\xymatrix@C=1.8em@R=0em{V_{1-p}\ar[r]^-{\eta_{p-2}^*}&W_{2-p}\ar[r]^-{\eta_{2-p}}&V_{3-p}\ar[r]^{\eta_{p-4}^*}&\cdots\ar[r]^-{\eta_{p-4}}&V_{p-3}\ar[r]^{\eta_{2-p}^*}&W_{p-2}\ar[r]^{\eta_{p-2}}&V_{p-1}~,\\&&&\oplus&\\&&&W_0'&}\]
    where $W_0'$ is a polystable $\rO(q-p+1,\C)$-bundle with $\det(W_0')=I$,  $W=W_0'\oplus\bigoplus\limits_{i=1}^{p-1}W_{-p+2i}$ with $W_j=IK^{-j}$ for all $j$, $V=\bigoplus\limits_{i=0}^{p-1}V_{1-p+2i}$ with $V_j=IK^{-j}$ for all $j$, 
 and each $\eta_j$ is non-zero. 
 \item $p=q$ and $(V,W,\eta)$ is of the form
 \[\xymatrix@C=1.8em@R=0em{W_{1-p}\ar[r]^-{\eta_{1-p}}&V_{2-p}\ar[r]^-{\eta_{p-3}^*}&W_{3-p}\ar[r]^{\eta_{3-p}}&\cdots\ar[r]^-{\eta_{p-3}^*}&W_{p-3}\ar[r]^{\eta_{p-3}}&V_{p-2}\ar[r]^{\eta_{1-p}^*}&W_{p-1}~,\\&&&\oplus&\\&&&I&}\]
    where $I$ is a $2$-torsion line bundle, $W=\bigoplus\limits_{i=0}^{p-1}W_{1-p+2i}$, $V=I\oplus \bigoplus\limits_{i=1}^{p-1}V_{-p+2i}$ with $W_j=IK^{-j}$ and $V_j=IK^{-j}$ for all $j$, 
 and each $\eta_j$ is non-zero. 
\item $q=p+1$ and $(V,W,\eta)$ is of the form
\[\xymatrix@C=1.8em@R=0em{W_{-p}\ar[r]^-{\eta_{-p}}&V_{1-p}\ar[r]^-{\eta_{p-2}^*}&W_{2-p}\ar[r]^{\eta_{2-p}}&V_{3-p}\ar[r]^{\eta_{p-4}^*}&\cdots\ar[r]^-{\eta_{p-4}}&V_{p-3}\ar[r]^{\eta_{2-p}^*}&W_{p-2}\ar[r]^{\eta_{p-2}}&V_{p-1}\ar[r]^{\eta_{-p}^*}&W_p}~,\]
    where $V=\bigoplus\limits_{i=0}^{p-1}V_{1-p+2i}$ with $V_j=K^{-j}$ for all $j$, 
$W=\bigoplus\limits_{i=0}^{p-1}W_{-p+2i}$ with $W_j=K^{-j}$ for all $|j|<p$, $\rk(W_{-p})=1$ with $0<\deg(W_{-p})\leq p(2g-2)$ and
each $\eta_j$ is non-zero. 
\end{enumerate}
\end{theorem}
\begin{remark}
In cases (2) and (3), $\det(V)= I^{p}=\det(W)$. Thus, such a Higgs bundle always reduces to $\SO_0(p,q)$ when $p$ is even, and reduces to $\SO_0(p,q)$ only when $I=\cO$ for $p$ odd.
\end{remark}

\begin{proof}
    If $\eta =0$, then we are done, 
    so suppose $\eta\neq0.$ By Theorem \ref{thm:stablenon-zerominH2=0} and Propositions \ref{prop H^2 not zero not a min} and \ref{prop:stableH2neq0-zeroGL(p)}, the result holds if $(V,W,\eta)$ is a stable $\SO(p,q)$-Higgs bundle, so suppose it is a strictly polystable $\SO(p,q)$-Higgs bundle. By Proposition \ref{Prop: strictly polystable SOpq}, 
    \[(V,W,\eta)\cong\Big(E\oplus E^*\oplus V',F\oplus F^*\oplus W',\smtrx{\gamma&&\\&\beta^*\\&&\eta'}\Big)~,\]
    where $(E,F,\beta,\gamma)$ is a polystable $\U(p_1,q_1)$-Higgs bundle and $(V',W',\eta')$ is a stable $\SO(p_2,q_2)$-Higgs bundle which does not necessarily have $p_2\leq q_2$. By Proposition \ref{Prop Upq in SO2p2q a min iff pq=1 or pq=0} and Proposition \ref{Prop U11 not min in SOpq unless p=q=2}, if such a Higgs bundle is a minimum of the Hitchin function, then one of the following hold
    \begin{enumerate}[(a)]
        \item $\beta=\gamma=0$ and $(V',W',\eta')$ is a minimum from Theorem \ref{thm:stablenon-zerominH2=0},
        \item $p_1=1$, $\beta=0$ or $\gamma=0$ and $\eta'=0$,
        \item $q_1=1$, $\beta=0$ or $\gamma=0$ and $\eta'=0$.
    \end{enumerate}

For case (a), note that if $p_2=0$ or $q_2=0$ then the Higgs field is zero and we are at a minimum. 
Consider a holomorphic chain of the form 
\[\xymatrix@R=0em@C=1.5em{V_{-r}'\ar[r]&W_{1-r}'\ar[r]&\ \ \cdots\ \ \ar[r]&W_{r-1}'\ar[r]&V_r'\\&&\oplus&&\\&&\hspace{.15cm}E\oplus E^*}\ \ \ \ \ \text{or}\ \ \ \  \ \xymatrix@C=1.5em@R=0em{W_{-r}'\ar[r]&V_{1-r}'\ar[r]&\ \ \cdots\ \ \ar[r]&V_{r-1}'\ar[r]&W_r'\\&&\oplus&&\\&&\hspace{.15cm}F\oplus F^*}\]
where $V_{-r}'$ and $W_{-r}'$ are  holomorphic line bundles of positive degree.
Since $\deg(E)=0$ and $\deg(V_r')<0$,  $H^1(\Hom(E,V_r'))$ and $H^1(\Hom(E^*,V_r'))$ are both non-zero. For $\alpha\in H^1(\Hom(E,V_r'))\setminus 0$ and $\sigma\in H^1(\Hom(E^*,V_r'))\setminus 0$, take a deformation of $V$ by fixing all the summands $V'_{2-r},\ldots,V'_{r-2}$, and deforming $V'_{-r}\oplus E\oplus E^*\oplus V'_r$ to $\mathbb{V}$ as in \eqref{eq:deformationsVandW}. Keep $W$ fixed. Keep also the Higgs field fixed, except that its restriction to $W'_{r-1}$ is composed with the inclusion of $V'_r\otimes K\to\mathbb{V}\otimes K$. As in the proofs of the previous propositions, this yields a polystable $\SO(p,q)$-Higgs bundle deforming the first chain above and decreasing $f$ by Proposition \ref{prop:notS-equiv-liminfty-notmin}. Similarly, the second chain does not define a minimum. 


Since $q\geq p,$ the only way we can have a holomorphic chain 
\[\xymatrix@C=1.5em@R=0em{W_{-r}'\ar[r]&V_{1-r}'\ar[r]&\ \ \cdots\ \ \ar[r]&V_{r-1}'\ar[r]&W_r'\\&&\oplus&&\\&&\hspace{.15cm}E\oplus E^*}\]
with $\rk(W_j')=\rk(V_j')=1$ for all $j$ is if $E=0$. Such a holomorphic chain is not strictly polystable.
To finish case $(a)$, consider holomorphic chains of the form 
\[\xymatrix@R=0em@C=1.5em{V_{-r}'\ar[r]&W_{1-r}'\ar[r]&\ \ \cdots\ \ \ar[r]&W_{r-1}'\ar[r]&V_r'\\&&\oplus&&\\&&\hspace{.15cm}F\oplus F^*}~.\]
By Theorem \ref{thm:stablenon-zerominH2=0} and Remark \ref{remark replacing invariant stable with polystable}, such a Higgs bundle is a polystable minimum if and only if it satisfies the conditions of case (1) or case (2) in the statement of the theorem.

For case (b), we have $\rk(E)=1$ and up to switching $E$ and $E^*$ the holomorphic chains are given by 
\begin{equation}\label{case b}
    \xymatrix@R=0em{&F&\\E\ar[r]^{\smtrx{\gamma\\0}}&\oplus\ar[r]^{\smtrx{0&\gamma^*}}&E^*\\&F^*&\\&\oplus&\\&\hspace{.1cm}V'\oplus W'&}
\end{equation}
where $0<\deg(E).$ As above, (with the roles of $E$ and $V'$ switched) this does not define a local minimum if $V'\neq0.$ When $V'=0,$ we have a local minimum satisfying case (1) of the statement of theorem. 

For case (c), we have $\rk(F)=1$ and the holomorphic chain is given by \eqref{case b} with $E$ and $F$ switched. As above, this is not a minimum if $W'=0.$ Since $p\leq q$ and $\rk(V)=\rk(V')+2\rk(E)\leq 2,$
we have $V'=0$, giving a local minimum satisfying case (1) of the statement of theorem.
 \end{proof}


  \section{The connected components of $\cM(\SO(p,q))$}\label{section pi0 count}
In this section we use the results from the previous sections to count the number of connected components of the moduli space $\cM(\SO(p,q))$, with $1\leq p\leq q$.  If $p\ne 2$ or if $(p,q)=(2,2)$ or $(p,q)=(2,3)$ then we have enough information to give a precise count.  In the remaining cases, namely $p=2$, $q\geq 4$, we give a lower bound on the number of connected components of $\cM(\SO(2,q))$ and conjecture that this bound is sharp.

\subsection{Connected components of $\cM(\SO(p,q))$ for $2< p\leq q$}
Recall from \eqref{eq sopq moduli top inv} that  the moduli space of $\SO(p,q)$-Higgs bundles decomposes as 
\begin{equation}\label{eq Msopq top decomp}
\cM(\SO(p,q))=\coprod_{a,b,c}\cM^{a,b,c}(\SO(p,q)),
\end{equation}
where the indices $(a,b,c)$ are classes in  $H^1(X,\Z_2)\times H^2(X,\Z_2)\times  H^2(X,\Z_2)$ and a polystable $\SO(p,q)$-Higgs bundle $(V,Q_V,W,Q_W,\eta)$ is in $\cM^{a,b,c}(\SO(p,q))$ if $a$ is the first Stiefel-Whitney class of $(V,Q_V)$ and $(W,Q_W),$ $b$ is the second Stiefel-Whitney class of $(V,Q_V)$ and $c$ is the second Stiefel-Whitney class of $(W,Q_W)$. Notice that each $\cM^{a,b,c}(\SO(p,q))$ is not necessarily connected.

When $2< p\leq q$, the maximal compact subgroup $\rS(\rO(p)\times \rO(q))\subset\SO(p,q)$ is semisimple. Thus by Proposition \ref{prop pi0 H semisimple} each of the spaces $\cM^{a,b,c}(\SO(p,q))$ is nonempty and has a unique connected component in which every Higgs bundle $(V,Q_V,W,Q_W,\eta)$ can be deformed to one with vanishing Higgs field.  Such components account for $2^{2g+2}$ connected components of $\cM(\SO(p,q))$. These are the `mundane' components mentioned in the Introduction. Taking into account the `exotic' components, we obtain the following precise count of the connected components of $\cM(\SO(p,q))$ for $2< p\leq q$.
\begin{theorem}\label{Theorem: component count p>2}
    Let $X$ be a compact Riemann surface of genus $g\geq 2$ and denote the moduli space of $\SO(p,q)$-Higgs bundles on $X$ by $\cM(\SO(p,q)).$ For $2< p\leq q$, we have 
    \[|\pi_0(\cM(\SO(p,q)))|=2^{2g+2}+\begin{dcases}
        2^{2g+1}-1+2p(g-1)&\text{if\ $q=p+1$}\\
        2^{2g+1}&\text{otherwise}~.
    \end{dcases}\]
\end{theorem}
\begin{remark}\label{rem: orientation doesn't change component count}
We have often ignored the orientation of an $\SO(p,q)$-Higgs bundle. This is justified because the choice of orientation does not effect the component count of Theorem \ref{Theorem: component count p>2}. Namely, every Higgs bundle can be deformed to a local minimum of the Hitchin function, and, for $2< p\leq q,$ such local minima either have zero Higgs field or are given by cases (2)-(4) of Theorem \ref{thm:MINIMA CLASSIFICATION}. The components corresponding to zero Higgs field are labeled by the topological invariants of $\rS(\rO(p)\times \rO(q))$-bundles. For minimum of cases (2)-(4) of Theorem \ref{thm:MINIMA CLASSIFICATION}, there is a holomorphic orthogonal summand of either $V$ or $W$ with odd rank. Taking the isomorphism which is $-\Id$ on this summand and $\Id$ on the other summands reverses the orientation and acts on the Higgs field by $\eta\mapsto -\eta$. However, since the minimum is a $\C^*$-fixed point, there is a orientation preserving gauge transformation which sends $-\eta\mapsto \eta.$
\end{remark}
\begin{proof}
By the above discussion we only need to determine the number of connected components of $\cM(\SO(p,q))$ with the property that the Higgs field never vanishes. Recall that if $\mathrm{Min}(\cM(\SO(p,q)))$ is the subspace of $\cM(\SO(p,q))$ where the Hitchin function \eqref{EQ Hitchin Function} attains a local minimum, then 
\[|\pi_0(\cM(\SO(p,q)))|\leq |\pi_0(\mathrm{Min}(\cM(\SO(p,q))))|.\]
From Theorem \ref{thm:MINIMA CLASSIFICATION}, an $\SO(p,q)$-Higgs bundle $(V,W,\eta)$, with $2< p<q-1$, is a minimum of the Hitchin function with non-zero Higgs field if and only if the holomorphic chain is given by: 

\begin{equation}\label{eq min type 2}
\xymatrix@C=1.8em@R=0em{V_{1-p}\ar[r]^-{\eta_{p-2}^*}&W_{2-p}\ar[r]^{\eta_{2-p}}&V_{3-p}\ar[r]^{\eta_{p-4}^*}&\cdots\ar[r]^-{\eta_{p-4}}&V_{p-3}\ar[r]^{\eta_{2-p}^*}&W_{p-2}\ar[r]^{\eta_{p-2}}&V_{p-1}~,\\&&&\oplus&\\&&&W_0'&}
\end{equation}
where the bundle $W_0'$ is a polystable $\rO(q-p+1,\C)$-bundle with $\det(W_0')=I$, $V_j=IK^{-j}$ and $W_j=IK^{-j}$ for all $j\neq 0,$ and each $\eta_j$ is non-zero.
Such chains also define minimum when $q=p$. The other minimum when $q=p$ are given by holomorphic chains 
\begin{equation}\label{eq min type 2.5}
        \xymatrix@C=2.3em@R=-.4em{&&&&I&\\W_{1-p}\ar[r]^-{\eta_{1-p}}&V_{2-p}\ar[r]^-{\eta_{3-p}^*}&\cdots\ar[r]^-{\eta_{1}^*}&W_{-1}\ar[r]^-{\smtrx{\eta_{-1}\\0}}&\oplus\ar[r]^-{\smtrx{\eta_{-1}^*&0}}&W_1\ar[r]^-{\eta_{1}}&\cdots\ar[r]^-{\eta_{p-3}}&V_{p-2}\ar[r]^-{\eta_{1-p}^*}&W_{p-1}~,\\&&&&I&},
    \end{equation}
    where $I^2=\cO,$ $V_j=IK^{-j}$ and $W_j=IK^{-j}$ for all $j\neq 0,$ and each $\eta_j$ is non-zero.
When $q=p+1$, in addition to minimum of the form \eqref{eq min type 2} with $\rk(W_0')=2$,  there are also minima of the form
\begin{equation}\label{eq min type 3}
    \xymatrix@C=1.8em@R=0em{W_{-p}\ar[r]^-{\eta_{-p}}&V_{1-p}\ar[r]^-{\eta_{p-2}^*}&W_{2-p}\ar[r]^{\eta_{2-p}}&V_{3-p}\ar[r]^{\eta_{p-4}^*}&\cdots\ar[r]^-{\eta_{p-4}}&V_{p-3}\ar[r]^{\eta_{2-p}^*}&W_{p-2}\ar[r]^{\eta_{p-2}}&V_{p-1}\ar[r]^{\eta_{-p}^*}&W_p}~,
 \end{equation} 
where $V_j=K^{-j}$ and $W_j=K^{-j}$ for all $|j|<p,$ $W_{-p}$ is a holomorphic line bundle with $0<\deg(W_{-p})\leq p(2g-2)$ and each $\eta_j$ is non-zero.

For $2< p=q$, each type of minimum is labeled by the choice of the $2$-torsion line bundle $I,$ yielding $2^{2g+1}$ connected components. For $2<p<q$, the connected components of the minima subvarieties of the form \eqref{eq min type 2} are labeled by the first and second Stiefel-Whitney class of the bundle $W_0'$ by Proposition \ref{prop pi0 H semisimple}. 
Thus, the number of connected components of these minima subvarieties is given by $|\Bun_X(\rO(q-p+1))|=2^{2g+1}$ for $2< p<q-1$. 
For $2< p=q-1$, when the first Stiefel-Whitney class of $W_0'$ vanishes the second Stiefel-Whitney class also vanishes since $sw_1(W_0')=0$ implies $W_0'=L\oplus L^{-1}$ for some degree zero line bundle $L.$ This gives $2^{2g+1}-1$ connected components of the minima subvarieties whose Higgs bundles are of the form \eqref{eq min type 2}.  There are $p(2g-2)$ connected components of minima subvarieties of type \eqref{eq min type 3} since its connected components are labeled by $\deg(W_{-p})\in(0,p(2g-2)]$.

Finally, by Theorem \ref{Thm Psi open and closed}, each of the above minima are in a different connected component of the image the map $\Psi:\cM_{K^p}(\SO(1,q-p+1))\times \bigoplus\limits_{j=1}^{p-1}H^0(K^{2j})\to \cM(\SO(p,q)).$ Thus, each such minima subvariety defines a connected component. 
\end{proof}

The following is a direct corollary of the above proof. This formulation will be useful in Section \ref{section positive}. Recall notation \eqref{eq:KnQn}.
\begin{corollary}\label{cor: Higgs bundle dichotomy}
Suppose $2<p<q-1$. For polystable Higgs bundles $(V,W,\eta)\in\cM(\SO(p,q))$ we have the following dichotomy: 
\begin{itemize}
    \item Either $(V,W,\eta)$ can be  deformed to a polystable $(V',W',0)$,
    \item or $(V,W,\eta)$ can be  deformed to $(\cK_{p-1}\otimes I,W_0'\oplus \cK_{p-2}\otimes I,(0\ \ \eta_0))$, where $W_0'$ is a polystable rank $q-p+1$ orthogonal bundle with $\Lambda^{q-p+1} W_0'=I$ and $(\cK_{p-1},\cK_{p-2},\eta_0)$ is the unique minimum in the $\SO(p-1,p)$-Hitchin component. 
 \end{itemize} 
\end{corollary}

For minima of the form \eqref{eq min type 2}, \eqref{eq min type 2.5} or  \eqref{eq min type 3}, the first and second Stiefel-Whitney classes of $V$ and $W$ are readily computed. The results are shown in the table. 

\begin{center}
    \begin{tabular}{|c|c|c|c|}
\hline
Type of min.& $a=sw_1(W)$& $b=sw_2(V)$ & $c=sw_2(W)$\\
\hline
\eqref{eq min type 2} &$\begin{matrix} 0&$\text{if $p$ is even}$\\
sw_1(W_0')&$\text{if $p$ is odd}$\end{matrix}$
& 0 & $sw_2(W'_0)$\\
\hline \eqref{eq min type 2.5}&$\begin{matrix} 0&$\text{if $p$ is even}$\\
sw_1(I)&$\text{if $p$ is odd}$\end{matrix}$&0&0\\
\hline
\eqref{eq min type 3} & 0&0& $\deg(W_{-p})\pmod 2$\\
\hline
\end{tabular}
\end{center}

\noi The following corollaries are immediate. Recall the notation of \eqref{eq Msopq top decomp}.

\begin{corollary}\label{corollary component count Mabc}
For $2<p< q-1,$ we have
\[|\pi_0(\cM^{a,b,c}(\SO(p,q)))|=\begin{dcases}
    2&\text{if $p$ is odd and $b=0$}\\
    2^{2g}+1&\text{if $p$ is even, $a=0$ and $b=0$}\\
    1&\text{otherwise~.}
\end{dcases}\]
 \end{corollary}

\begin{corollary}
\label{corollary connected components of SO(p,p+1)}
    For $2<p$ and $p=q-1$, we have
    \[|\pi_0(\cM^{a,b,c}(\SO(p,p+1)))|=\begin{dcases}
    2&\text{if $p$ is odd, $b=0$ and $a\neq0$}\\
    2+p(g-1)&\text{if $p$ is odd and $a=b=c=0$}\\
     1+p(g-1)&\text{if $p$ is odd and $a=b=0$ and $c\neq 0$}\\
    2+2^{2g}+p(g-1)&\text{if $p$ is even and $a=b=c=0$}\\
    1+2^{2g}+p(g-1)&\text{if $p$ is even and $a=b=0$ and $c\neq0$}\\
    1&\text{otherwise~.}
\end{dcases} \]
\end{corollary}

\begin{corollary}
\label{corollary connected components of SO(p,p)}
    For $2<p$ and $p=q$,  we have
    \[|\pi_0(\cM^{a,b,c}(\SO(p,p)))|=\begin{dcases}
    3&\text{if $p$ is odd and $b=c=0$}\\
    2^{2g+1}+1&\text{if $p$ is even and $a=b=c=0$}\\
    1&\text{otherwise~.}
\end{dcases} \]
\end{corollary}

\noi We observe finally that the following corollary is immediate since the map $\Psi$ is injective.
\begin{corollary}\label{corollary connected components of Kp SO(1,n)}
    For $p\geq 1$, the number of connected components of $\cM_{K^p}(\SO(1,q))$ are given by 
    \[|\pi_0(\cM_{K^p}(\SO(1,q)))|=\begin{dcases}
        2^{2g}& q=1\\
        2^{2g+1}-1+p(2g-2)& q=2\\
        2^{2g+1}&q>2~.
    \end{dcases}\] 
    In particular, if $q>2$ then every polystable $K^p$-twisted $\SO(1,q)$-Higgs bundle can be  deformed to one with zero Higgs field.
\end{corollary}

\subsection{Connected components of $\cM(\SO(2,q))$}

\label{Section SO2q}
In the previous section a complete component count of $\cM(\SO(p,q))$ when $p\leq q$ and $p\neq 2$ was given. We now discuss the case $p=2$. In this special case the group $\SO(p,q)$ is a group of Hermitian type. Furthermore in this case the minima of type (1) from Theorem \ref{thm:MINIMA CLASSIFICATION} appear. These are given by holomorphic chains of the form 
\begin{equation}
    \label{eq fix 2}\xymatrix{V_{-1}\ar[r]^{\eta_0^*}&W\ar[r]^{\eta_0}&V_1}~,
\end{equation}
where $0<\deg(V_{-1})<2g-2$ and $\eta_0$ is non-zero.

Let $(V,W,\eta)$ be an $\SO(2,q)$-Higgs bundle. As in the general case, the first and second Stiefel-Whitney classes of the orthogonal bundles provide primary topological invariants which help distinguish the connected components of the  moduli space. However, when the first Stiefel-Whitney class vanishes, we have $(V,Q_V)\cong( L\oplus L^{-1}, \smtrx{0&1\\1&0})$ for some line bundle $L.$ The natural number $|\deg(L)|$ satisfies $|\deg(L)|=sw_2(V)\pmod 2$ and provides a refinement of the second Stiefel-Whitney class invariant. This natural number is the absolute value of the so-called Toledo invariant of the $\SO(2,q)$-Higgs bundle. Moreover, if such an $\SO(2,q)$-Higgs bundle $(V,W,\eta)$ is polystable then 
\[|\deg(L)|\leq 2g-2.\]
This inequality is usually referred to as the Milnor-Wood inequality and was derived in the proof of Theorem \ref{thm:stablenon-zerominH2=0} (see \eqref{eq:degV_1}). The special maximal case $|\deg(L)|=2g-2$ will be discussed in Section \ref{section Cayley partner}. 

Examining the minima classification of Theorem \ref{thm:MINIMA CLASSIFICATION} and using Theorem \ref{Thm Psi open and closed}, in the case $2=p\leq q$ we see that the only obstruction to obtaining a full connected component count of $\cM(\SO(2,q))$ is the connectedness of the fixed point set \eqref{eq fix 2}.
In particular, for $2=p<q$, we get bounds, rather than precise values, namely
\[|\pi_0(\cM(\SO(2,q))|\geq \begin{dcases} 2^{2g+2}-4+4(g-1)+2^{2g+1}+4g-5 &\text{if\ $q=3$}\\
2^{2g+2}-4+4(g-1)+2^{2g+1}& \text{if $q\geq 4$}
\end{dcases}\]
It follows from \cite{AndreQuadraticPairs}, that the above inequality was shown to be an equality for $q=3$: 
\begin{equation}
|\pi_0(\cM(\SO(2,3))|=3\times 2^{2g+1}+8g-13.
\end{equation}
We conjecture that equality also holds above for $q\geq 4$.

The complete count of components for $\cM(\SO(2,2))$ has been deduced by different methods in \cite[Corollary 7.1]{baragliaschaposnikmonodromyrank2}. We obtain the same count, as we now briefly explain, leaving the details for the reader. By Proposition \ref{prop SO(2,2) fixed points are minima} and \eqref{EQ SO(2,2) fixed points}, any non-zero local minima reduces to $\SO_0(2,2)$. The allowed topological types of a polystable $\SO_0(2,2)$-Higgs bundle are given by a pair of integers $(l,m)$ such that $l\geq 0$ and $l-2g+2\leq m\leq 2g-2-l$, and if $l=0$, then only $|m|$ is an invariant. All the minima are connected subvarieties, except when $(l,m)$ equals $(0,2g-2)$ or $(2g-2,0)$ each corresponding to $2^{2g}$ Hitchin components. Adding the zero minima which do not reduce to $\SO_0(2,2)$, yields the following.

\begin{proposition}
$|\pi_0(\cM(\SO(2,2))|=3(2^{2g+1}-1)+2g(2g-3)$.
\end{proposition}

\section{Positive surface group representations and Cayley partners} \label{section positive}

In this section, we recall the Non-Abelian Hodge correspondence between the Higgs bundle moduli space and the moduli space of surface group representations. After proving some immediate consequences of Theorem \ref{Theorem: component count p>2}, we discuss how the exotic components of Theorem \ref{Thm Psi open and closed} are related to recent work of Guichard and Wienhard on positive Anosov representations \cite{PosRepsGWPROCEEDINGS}. Finally, we show this relation with positive Anosov representations can be seen as a generalization of the phenomenon which produces the so-called Cayley partner of a $\G$-Higgs bundle with maximal Toledo invariant  for $\G$ a Hermitian group of tube type.

\subsection{Surface group representations}
Let $\pi_1(S)$ be the fundamental group of a closed oriented surface $S$ of genus $g\geq 2$ and let $\G$ be a real semisimple Lie group. A representation $\rho:\pi_1S\to\G$ is called {\em reductive} if the composition of $\rho$ with the adjoint representation of $\G$ is a completely reducible representation. 

 Denote the set of reductive representations by $\Hom^{red}(\pi_1S,\G).$
 The conjugation action of $\G$ on $\Hom(\pi_1S,\G)$ does not in general have a Hausdorff quotient. However, if we restrict to the set of reductive representations, the quotient will be Hausdorff.
\begin{definition}
  The {\em $\G$-representation variety} $\cR(S,\G)$ of a surface group $\pi_1S$ is the space of conjugacy classes of reductive representations of $\pi_1S$ in $\G$:
  \[\cR(S,\G)=\Hom^{red}(\pi_1S,\G)/\G~.\]
\end{definition}
\begin{example}\label{EX: Fuchsian reps}
  The set of {\em Fuchsian representations} $\Fuch(S)\subset\cR(S,\SO(2,1))$ is defined to be the subset of conjugacy classes of {\em faithful} representations with {\em discrete image}.  
The space $\Fuch(S)$ defines one connected components of $\cR(S,\SO(2,1))$ \cite{TopologicalComponents} and is in one to one correspondence with the Teichm\"uller space of isotopy classes of marked Riemann surface structures on the surface $S.$ Since the surface $S$ is assumed to be orientable, every Fuchsian representation reduces to $\SO_0(2,1).$

For $\G$ a split real form, there is a preferred class of embeddings 
\begin{equation}
    \label{eq principal embedding}\iota:\SO_0(2,1)\to \G
\end{equation} called a principal embedding. 
When $\G$ is an adjoint group, the principal embedding is unique up conjugation. 
For the split real form $\G=\SO_0(p,p-1)$, the principal embedding is given by taking the $(p-1)^{st}$-symmetric product of the standard action of $\SO_0(2,1)$ on $\R^3.$ 
The principal embedding defines a map $\iota:\cR(S,\SO_0(2,1))\to\cR(S,\G)$, and the Hitchin component $\Hit(S,\G)\subset\cR(S,\G)$ is defined to be the connected component containing $\iota(\Fuch(S)).$
\end{example}

Each representation $\rho\in\cR(S,\G)$ defines a flat $\G$-bundle $E_\rho=(\widetilde S\times\G)/\pi_1S~.$
This gives a decomposition of the $\G$ representation variety:
\[\cR(S,\G)=\bigsqcup\limits_{a\in\Bun_S(\G)}\cR^a(\G)~,\]
where $a\in\Bun_S(\G)$ is the topological type of the flat $\G$-bundle of the representations in $\cR^a(\G).$ When $\G$ is a Hermitian Lie group $\Bun_S(\G)$ is infinite. Such $\G$-Higgs bundles and surface group representations acquire a discrete invariant called the Toledo invariant. While the Toledo invariant has several different descriptions, they all yield a finite set of allowed rational values, and hence give a notion of maximality (see for example \cite{MilnorWoodIneqDomicToledo,BIWmaximalToledoAnnals,BGRmaximalToledo}). In particular, $\cR^a(\G)$ is nonempty for only finitely many values of $a\in\Bun_S(\G).$

The following theorem links the $\G$-representation variety and the $\G$-Higgs bundle moduli space. It was proven by Hitchin \cite{selfduality}, Donaldson \cite{harmoicmetric}, Corlette \cite{canonicalmetrics} and Simpson \cite{SimpsonVHS} in various generalities. For the general statement below see \cite{HiggsPairsSTABILITY}.

\begin{theorem}\label{Nonabelian Hodge Correspondence}
  Let $S$ be a closed oriented surface of genus $g\geq2$ and $\G$ be a real semisimple Lie group. 
  For each Riemann surface structure $X$ on $S$ there is a homeomorphism between the moduli space $\cM_K(\G)$ of $\G$-Higgs bundles on $X$ and the $\G$-representation variety $\cR(S,\G)$. 
  Moreover, for each $a\in\Bun_S(\G)$, this homeomorphism identifies the spaces $\cM_K^a(\G)$ and $\cR^a(\G).$
\end{theorem}

As in \eqref{eq Msopq top decomp}, for $(a,b,c)\in H^1(S,\Z_2)\times H^2(S,\Z_2)\times H^2(S,\Z_2),$ we have 
\[\cR(S,\SO(p,q))=\coprod\cR^{a,b,c}(\SO(p,q)).\]
Using Theorem \ref{Theorem: component count p>2} and the above correspondence we have a connected component count of $\cR(S,\SO(p,q)).$
\begin{theorem}
    \label{thm pioRSO(p,q)} Let $S$ be a closed surface of genus $g\geq 2$.  For $2< p\leq q$, the number of connected components of the representation variety $\cR(S,\SO(p,q))$ is given by 
    \[|\pi_0(\cR(S,\SO(p,q)))|=2^{2g+2}+\begin{dcases}
        2^{2g+1}-1+2p(g-1)&\text{if\ $q=p+1$}\\
        2^{2g+1}&\text{otherwise}~.
    \end{dcases}\]
\end{theorem}
\begin{remark}
    The connected components of $\cR^{a,b,c}(\SO(p,q))$ are given by corollaries \ref{corollary component count Mabc}, \ref{corollary connected components of SO(p,p+1)}, and \ref{corollary connected components of SO(p,p)}.
\end{remark}

Corollary \ref{cor: Higgs bundle dichotomy} can now be interpreted as a dichotomy in terms of the $\SO(p,q)$ representation variety.
\begin{theorem}\label{THM: rep var dichotomy}
Let $S$ be a closed surface of genus $g\geq 2$. For $2<p<q-1$, the representation variety $\cR(S,\SO(p,q))$ is disjoint union of two sets
\begin{equation}
    \label{eq rep variety dichotomy c and ex}\cR(S,\SO(p,q))=\cR^{cpt}(S,\SO(p,q))\sqcup \cR^{ex}(S,\SO(p,q))~,
\end{equation}
where
\begin{itemize}
    \item  $[\rho]\in\cR^{cpt}(S,\SO(p,q))$ if and only if $\rho$ can be  deformed to a compact representation,
    \item  $[\rho]\in\cR^{ex}(S,\SO(p,q))$ if and only if $\rho$ can be  deformed to a representation 
    \begin{equation}\label{eq: model rep}
    \rho'=\alpha\oplus (\iota\circ\rho_{\mathrm{Fuch}})\otimes \det(\alpha)~,
    \end{equation}
where $\alpha$ is a representation of $\pi_1S$ into the compact group $\rO(q-p+1)$, $\rho_{\mathrm{Fuch}}$ is a Fuchsian representation of $\pi_1S$ into $\SO_0(2,1)$, and $\iota$ is the principal embedding from \eqref{eq principal embedding}.
 \end{itemize}
\end{theorem}

\begin{proof}
For the first part, note that a representation $\rho:\pi_1S\to\SO(p,q)$ can be  deformed to a compact representation if and only if the corresponding Higgs bundle can be  deformed to one with vanishing Higgs field. 

If $\rho$ cannot be  deformed to a compact representation, then by Corollary \ref{cor: Higgs bundle dichotomy}, the associated $\SO(p,q)$-Higgs bundle $(V,W,\eta)$ can be  deformed to (cf.\ \eqref{eq:KnQn})
\[(\cK_{p-1}\otimes I,\widehat W\oplus \cK_{p-2}\otimes I,(0\ \ \eta_0)),\] where $\widehat W$ is a polystable rank $q-p+1$ orthogonal bundle with $\Lambda^{q-p+1}\widehat W=I$ and $(\cK_{p-1},\cK_{p-2},\eta_0)$ is the unique minimum in the $\SO(p-1,p)$-Hitchin component. 
Through Theorem \ref{Nonabelian Hodge Correspondence}, the Higgs bundle description of the Hitchin component from \eqref{THM Hitchin component} is identified with the representation variety from Example \ref{EX: Fuchsian reps}. In particular, if $s_H$ is the Hitchin section from \eqref{EQ Hitchin section O(p,C)}, the representation associated to $s_H(0)$ is $\iota\circ \rho_{\Fuch}$ for a Fuchsian representation $\rho_{\Fuch}$ \cite{liegroupsteichmuller}. In particular, the representation associated to the unique minimum in the $\SO_0(p,p-1)$-Hitchin component $(\cK_{p-1},\cK_{p-2},\eta_0)$ is given by $\iota\circ\rho_{\Fuch}$ for a Fuchsian representation $\rho_{\Fuch}.$

If $A\in\SO_0(p,p-1)$ and $B\in\rO(q-p+1),$ then $(A,B)\mapsto \smtrx{\det(B)\cdot A&0\\0&B}$ defines an embedding
\[\SO_0(p,p-1)\times \rO(q-p+1)\hookrightarrow\SO(p,q).\] 
If $\alpha:\pi_1S\to\rO(q-p+1)$ is the representation associated to the polystable $\rO(q-p+1,\C)$-bundle $\widehat W,$ then the representation associated to the $\SO(p,q)$-Higgs bundle $(\cK_{p-1}\otimes I,\widehat W\oplus \cK_{p-2}\otimes I,(0\ \ \eta_0))$ is given by $\alpha\oplus (\iota\circ\rho_\Fuch)\otimes \det(\alpha)$.
\end{proof}

\subsection{Positive Anosov representations}
Anosov representations were introduced by Labourie \cite{AnosovFlowsLabourie} and have many interesting geometric and dynamic properties which generalize convex cocompact representations into rank one Lie groups. 
Important examples of Anosov representations include Fuchsian representations, quasi-Fuchsian representations, Hitchin representations into split real groups and maximal representations into Lie groups of Hermitian type. We will describe the necessary properties of Anosov representations and refer the reader to \cite{AnosovFlowsLabourie,guichard_wienhard_2012,AnosovAndProperGGKW,KLPDynamicsProperCocompact} for more details. 

Let $\G$ be a semisimple Lie group and $\rP\subset\G$ be a parabolic subgroup. Let $\rL\subset\rP$ be the Levi factor (the maximal reductive subgroup) of $\rP$, it is given by $\rL=\rP\cap\rP^{opp}$, where $\rP^{opp}$ is the opposite parabolic of $\rP.$ 
The homogeneous space $\G/\rL$ realized as the unique open $\G$ orbit in $\G/\rP\times\G/\rP$, and points $(x,y)\in\G/\rP\times\G/\rP$ in this open orbit are called {\em transverse}.

\begin{definition}\label{DEF: Anosov rep}
    Let $S$ be a closed orientable surface of genus $g\geq 2$. Let $\partial_\infty\pi_1S$ be the Gromov boundary of the fundamentatl group $\pi_1S$, topologically $\partial_\infty\pi_1S\cong\R\mathbb{P}^1$. A representation $\rho:\pi_1S\to\G$ \emph{is $\rP$-Anosov} if there exists a unique continuous boundary map $\xi_\rho:\partial_\infty\pi_1S\to\G/\rP$
which satisfies 
\begin{itemize}
    \item Equivariance: $\xi(\gamma\cdot x)=\rho(\gamma)\cdot\xi(x)$ for all $\gamma\in\pi_1S$ and all $x\in\partial_\infty\pi_1S$.
    \item Transversality: for all distinct $x,y\in\partial_\infty\pi_1S$ the generalized flags $\xi(x)$ and $\xi(y)$ are transverse.
    \item Dynamics preserving: see \cite{AnosovFlowsLabourie,guichard_wienhard_2012,AnosovAndProperGGKW,KLPDynamicsProperCocompact} for the precise notion. 
\end{itemize}
The map $\xi_\rho$ will be called the {\em $\rP$-Anosov boundary curve}.
\end{definition}

One important property of Anosov representations is that they define an open subset of the representation variety $\cR(S,\G)$. The set of Anosov representations is however not closed. For example, for the group $\PSL(2,\C)$ the set of Anosov representations corresponds to the non-closed set quasi-Fuchsian representations of $\cR(S,\PSL(2,\C))$.
The special cases of Hitchin representations and maximal representations define connected components of Anosov representations. Both Hitchin representations and maximal representations satisfy an additional ``positivity'' property which is a closed condition. 
For Hitchin representations this was proven by Labourie \cite{AnosovFlowsLabourie} and Fock-Goncharov \cite{fock_goncharov_2006}, and for maximal representations by Burger-Iozzi-Wienhard \cite{MaxRepsAnosov}. 
These notions of positivity have recently been unified and generalized by Guichard and Wienhard \cite{PosRepsGWPROCEEDINGS}.

For a parabolic subgroup $\rP\subset\G$, denote the Levi factor of $\rP$ by $\rL$ and the unipotent subgroup by $\U\subset\rP$. 
The Lie algebra $\fp$ of $\rP$ admits an $Ad_{\rL}$-invariant decomposition $\fp=\fl\oplus\fu$ where $\fl$ and $\fu$ are the Lie algebras of $\rL$ and $\U$ respectively. 
Moreover, the unipotent Lie algebra $\fu$ decomposes into irreducible $\rL$-representation:
\[\fu=\bigoplus\fu_\beta~.\]
Recall that a parabolic subgroup $\rP$ is determined by fixing a simple restricted root system $\Delta$ of a maximal $\R$-split torus of $\G$, and choosing a subset $\Theta\subset\Delta$ of simple roots. 
To each simple root $\beta_j\in\Theta$ there is a corresponding irreducible $\rL$-representation space $\fu_{\beta_j}.$
\begin{definition}(\cite[Definition 4.2]{PosRepsGWPROCEEDINGS})\label{DEF: Positive Structure}
    A pair $(\G,\rP^\Theta)$ \emph{admits a positive structure} if for all $\beta_j\in\Theta,$ the $\rL^\Theta$-representation space $\fu_{\beta_j}$ has an $\rL^\Theta_0$-invariant acute convex cone $c_{\beta_j}^\Theta$, where $\rL^\Theta_0$ denotes the identity component of $\rL^\Theta$. 
\end{definition}
If $(\G,\rP^\Theta)$ admits a positive structure, then exponentiating certain combinations of elements in the $\rL^\Theta_0$-invariant acute convex cones give rise to a semigroup $\U^\Theta_{>0}\subset\U^\Theta$ \cite[Theorem 4.5]{PosRepsGWPROCEEDINGS}. The existence of the semigroup $U_{>0}$ gives a well defined notion of positively oriented triples of pairwise transverse points in $\G/\rP^\Theta.$ This notion allows one to define a {\em positive Anosov representation}. 
\begin{definition}
(\cite[Definition 5.3]{PosRepsGWPROCEEDINGS})\label{DEF: Positive rep}
    If the pair $(\G,\rP^\Theta)$ admits a positive structure, then a $\rP^\Theta$-Anosov representation $\rho:\pi_1S\to\G$ is called \emph{positive} if the Anosov boundary curve $\xi:\partial_\infty\pi_1S\to \G/\rP^\Theta$ sends positively ordered triples in $\partial_\infty\pi_1S$ to positive triples in $\G/\rP^\Theta.$
\end{definition}
\begin{conjecture}\label{Conj GW}
    (\cite{PosRepsGLW,PosRepsGWPROCEEDINGS}) If $(\G,\rP^\Theta)$ admits a notion of positivity, then the set $\rP^\Theta$-positive Anosov representations is an open and closed subset of $\cR(S,\G).$
\end{conjecture}
In particular, the aim of this conjecture is to characterize the connected components of $\cR(S,\G)$ which are not labeled by primary topological invariants as being connected components of positive Anosov representations, such connected components are referred as higher Teichm\"uller components. 
\begin{remark}
    When $\G$ is a split real form and $\Theta=\Delta$, the corresponding parabolic is a Borel subgroup of $\G$. In this case, the connected component of the identity of the Levi factor is $\rL^\Delta_0\cong(\R^+)^{rk(\G)}$ and each simple root space $\fu_{\beta_i}$ is one dimensional. The $\rL^\Delta_0$-invariant acute convex cone in each simple root space $\fu_{\beta_i}$ is isomorphic to $\R^+.$ The set of $\rP^\Delta$-positive Anosov representations into a split group are exactly Hitchin representations. 
    When $\G$ is a Hermitian Lie group of tube type and $\rP$ is the maximal parabolic associated to the Shilov boundary of the Riemannian symmetric space of $\G$, the pair $(\G,\rP)$ also admits a notion of positivity \cite{BIWmaximalToledoAnnals}. In this case, the space of maximal representations into $\G$ are exactly the $\rP$-positive Anosov representations. In particular, the above conjecture holds in these two cases.  
\end{remark}

In general, the group $\SO(p,q)$ is not a split group and not a group of Hermitian type. 
Nevertheless, if $p\neq q$, then $\SO(p,q)$ has a parabolic subgroup $\rP^\Theta$ which admits a positive structure. Here $\rP^\Theta$ is the stabilizer of the partial flag $V_1\subset V_2\subset\cdots \subset V_{p-1},$
where $V_j\subset\R^{p+q}$ is a $j$-plane which is isotropic with respect to a signature $(p,q)$ inner product with $p<q.$ 
Here the subgroup $\rL^\Theta_{pos}\subset\rL^\Theta\subset\SO(p,q)$ which preserves the cones $c^\Theta_{\beta_j}$ is isomorphic to $\rL^\Theta_{pos}\cong(\R^+)^{p-1}\times \SO(1,q-p+1)~.$ We refer the reader to \cite{PosRepsGWPROCEEDINGS} and \cite[Section 7]{CollierSOnn+1components} for more details. 

To construct examples of $\SO(p,q)$ positive Anosov representations we have the following proposition.
\begin{proposition}\label{prop: example of positive reps}
    Let $p<q$. Consider the signature $(p,q)$-inner product $\langle x,x\rangle= \sum\limits_{j=1}^px_j^2-\sum\limits_{j=p+1}^{p+q}x_j^2$. If $A\in\SO_0(p,p-1)$ and $B\in\rO(q-p+1),$ then the set matrices $\smtrx{\det(B)\cdot A&0\\0&B}$ defines an embedding
\[\SO_0(p,p-1)\times \rO(q-p+1)\hookrightarrow\SO(p,q).\]
 If $\rho_{Hit}:\pi_1S\to\SO_0(p,p-1)$ is a Hitchin representation and $\alpha:\pi_1S\to\rO(q-p+1)$ is any representation, then 
    \[\rho=\rho_{Hit}\otimes \det(\alpha)\oplus \alpha:\pi_1S\to\SO(p,q)\]
    is a $\rP^{\Theta}$-positive Anosov representation.
\end{proposition}
This is proven for $q=p+1$ in \cite[Section 7]{CollierSOnn+1components}, and the proof for general $q$ is the same. 
For the proof of the first part of the above proposition it suffices to show that the map $\SO(p,p-1)\to\SO(p,q)$ described above sends the  positive semigroup $U^\Delta_{>0}\subset\SO(p,p-1)$ into the positive semigroup $U^\Theta_{>0}.$ The second part follows from the fact that a representation $\rho$ is a $\rP$-Anosov representation if and only if the restriction of $\rho$ to any finite index subgroup is $\rP$-Anosov, and the fact that the centralizer of an Anosov representation acts trivially on the Anosov boundary curve. 

Using Proposition \ref{prop: example of positive reps} and Theorem \ref{THM: rep var dichotomy},  we conclude that for $q>p+1$ the connected components of $\cR(S,\SO(p,q))$ from Theorem \ref{Thm Psi open and closed} contain $\rP^\Theta$-positive Anosov representations. 

\begin{proposition}\label{Prop existence of positive reps}
Let $\rP^\Theta\subset\SO(p,q)$ be the stabilizer of the partial flag $V_1\subset V_2\subset\cdots \subset V_{p-1},$
where $V_j\subset\R^{p+q}$ is a $j$-plane which is isotropic with respect to a signature $(p,q)$ inner product with $p<q$.  If $q>p+1$, then each connected component of $\cR^{ex}(S,\SO(p,q))$ from \eqref{eq rep variety dichotomy c and ex} contains $\rP^\Theta$-positive Anosov representations.
\end{proposition}
\begin{remark}
When $q=p+1$, this was shown in \cite{CollierSOnn+1components} for the analogous connected components which contain minima of the form \eqref{eq min type 2}. The components which contain minima of the form \eqref{eq min type 3} are smooth, and one cannot use Proposition \ref{prop: example of positive reps} to obtain positive representations in these components. 
    However, we note that if Conjecture \ref{Conj GW} holds, then each of the these smooth connected components of $\cR(S,\SO(p,p+1))$ consists of positive representations since each component would be contained in a component of positive representations into $\SO(p,p+2).$ 
\end{remark}

Another special feature of Hitchin representations and   maximal representations is that they satisfy a certain irreducibility condition. Namely, if $\rho:\pi_1S\to \G$ is such a representation, then there is \emph{no} proper parabolic subgroup $\rP$ so that $\rho$ factors as $\rho:\pi_1S\to \rP\hookrightarrow \G.$ For the Hitchin case, this follows from smoothness, and for the maximal case it follows from the from \cite[Theorem 5]{BIWmaximalToledoAnnals}. For the components in $\cR^{ex}(S,\SO(p,q))$, with $2<p<q-1$ (cf.\ \eqref{eq rep variety dichotomy c and ex}), it follows from Corollary \ref{Cor: Higgs irr.}. Let $\cR^{cpt}(S,\SO(p,q))$ be the union of the connected components of $\cR(S,\SO(p,q))$ containing compact representations.
\begin{proposition}\label{Prop: irreducibility of exotic reps}
Let $2<p\leq q$ and $\rho\in \cR(S,\SO(p,q))\setminus\cR^{cpt}(S,\SO(p,q))$. Then $\rho$ does not factor through any proper parabolic subgroup of $\SO(p,q)$.
\end{proposition}
\begin{proof}
    Suppose $\rho\in\cR^{ex}(S,\SO(p,q))$ factors through a proper parabolic subgroup $\rP$. Since points of $\cR(S,\SO(p,q))$ consist of completely reducible representations, $\rho$ must factor through the Levi factor $\rL$ of $\rP.$ Consequently, the $\SO(p,q)$-Higgs bundle associated to $\rho$ must reduce to an $\rL$-Higgs bundle. The Levi factors of parabolics of $\SO(p,q)$ are isomorphic to $\GL(n,\R)\times\SO(p-n,q-n)$, for some $n$, embedded as
    \[(A,B)\mapsto\smtrx{A\\&B\\&&A^{-1}}~.\]
    But by Corollary \ref{Cor: Higgs irr.}, the Higgs bundles in the components associated to $\cR^{ex}(S,\SO(p,q))$ do not reduce to such groups, leading to a contradiction. 
\end{proof}

Propositions \ref{Prop existence of positive reps} and \ref{Prop: irreducibility of exotic reps} give further evidence for Conjecture \ref{Conj GW}, and it is thus natural to expect that all representations in the connected components from Theorem \ref{Thm Psi open and closed} are positive Anosov representations. Indeed, this would follow from Conjecture \ref{Conj GW} and Proposition \ref{Prop existence of positive reps}. Moreover, if Conjecture \ref{Conj GW} is true, then the connected components of Theorem \ref{Thm Psi open and closed} correspond exactly to those connected components of $\cR(S,\SO(p,q))$ which contain positive Anosov representations. 

\subsection{Positivity and a generalized Cayley correspondence}
\label{section Cayley partner}
We conclude the paper by interpreting the parameterization of the `exotic' connected components of the $\SO(p,q)$-Higgs bundle moduli space from Theorem \ref{Thm Psi open and closed} as a generalized Cayley correspondence.  

Let $\G$ be a simple adjoint Hermitian Lie group of tube type and let $\G/\rP$ be the Shilov boundary of the symmetric space of $\G.$ In \cite{BGRmaximalToledo}, it is proven that if $\rL$ is the Levi factor of $\rP$, then the space of Higgs bundles with maximal Toledo invariant is isomorphic to $\cM_{K^2}(\rL).$ More generally, an analogous statement holds when $\G'\to \G$ is a finite cover such that a $\G$-Higgs bundle with maximal Toledo invariant lifts to a $\G'$-Higgs bundle. This correspondence between maximal $\G$-Higgs bundles and $K^2$-twisted $\rL$-Higgs bundles is called the Cayley correspondence. 

\begin{remark}
    In \cite{BGRmaximalToledo}, the above statement is stated differently. We use the above interpretation because it relates directly with the notions of positivity discussed in the previous section.  
\end{remark}

Note that the above parabolic and Levi factor are exactly the objects which appear in the notion of positivity when $\G$ is Hermitian Lie group of tube type. When $\G$ is a split real form the Hitchin components of $\mathcal{M}(G)$ admit an analogous interpretation. 
Namely, if $\G$ is such a split group, then $(\G,\rP)$ admits a positive structure when $\rP$ is a minimal parabolic subgroup. In this case, $\rL\subset\rP$ is $(\R^*)^{\rk(\G)}$ and the identity component $\rL_0$ is given by $(\R^+)^{\rk(\G)}.$ 
Moreover, the moduli space of $K^j$-twisted $\R^+$-Higgs bundles is isomorphic to $H^0(K^j):$
\[\cM_{K^j}(\R^+)\cong H^0(K^j).\]
Thus, when the Hitchin base is $\bigoplus\limits_{j=1}^{\rk(G)}H^0(K^{m_j+1}),$ the Hitchin components are given by 
\[\cM_{K^{m_1+1}}(\R^+)\times \cdots \times \cM_{K^{m_{\rk(G)}+1}}(\R^+).\] 
In particular, the Higgs bundles associated surface group to representations into split real groups which are positive with respect the minimal parabolic subgroup also satisfy a `Cayley correspondence'.

For the group $\SO(p,q),$ the Levi factor of the parabolic $\rP^\Theta$ so that $(\SO(p,q),\rP^\Theta)$ has a positive structure is $\rL^\Theta=\SO(1,q-p+1)\times (\R^*)^{p-1}$. Moreover, the subgroup $\rL^\Theta_{pos}$ which preserves the positive cones is 
\[\rL^\Theta_{pos}\cong\SO(1,q-p+1)\times\underbrace{\R^+\times \cdots\times\R^+}_{(p-1)\text{-times}}.\]
Recall that the `exotic' connected components in the image of $\Psi$ Theorem \ref{Thm Psi open and closed} are isomorphic to 
\[\cM_{K^p}(\SO(1,q-p+1))\times \prod\limits_{j=1}^{p-1}H^0(K^{2j}).\]
Using $\cM_{K^{2j}}(\R^+)=H^0(K^{2j})$, this is equivalent to 
\[\cM_{K^p}(\SO(1,q-p+1))\times \prod\limits_{j=1}^{p-1}\cM_{K^{2j}}(\R^+).\]

 When $2=p\leq q$, we recover the Cayley correspondence for groups of Hermitian type \cite{HermitianTypeHiggsBGG,BGRmaximalToledo}. 
Hence, for $2<p\leq q$ we have established that the Higgs bundles associated to representations into $\SO(p,q)$ which cannot be  deformed to compact representations satisfy a generalized Cayley correspondence. 
    Moreover, when $p<q-1$ each such component of the representation variety contains positive representations by Proposition \ref{Prop existence of positive reps}. 
    This suggests a general theorem for positive representations which relates the connected components of the subgroup of $L^\Theta$ which preserves the cones with the product of moduli spaces of appropriately twisted $L_j$-Higgs bundles.
Indeed this is consistent with results in \cite{TopInvariantsAnosov} where topological invariants for $\theta$-positive representations are defined in terms of principal bundles with structure group given by the Levi subgroup we have identified as $L^{\Theta}$. It would be interesting to understand in more detail the relation between these two points of view.

\appendix
\section{Review of gauge theory and the Hitchin--Kobayashi correspondence}
\label{sec:review-gauge}

Details on points treated sketchily in the following can be found in
\cite{selfduality} and \cite{HiggsPairsSTABILITY}. For simplicity we
consider $K$-twisted Higgs bundles but analagous
statements can be made for $L$-twisted Higgs bundles.

Let $\G$ be a real semisimple Lie group and $\rH\subset\G$ a maximal compact subgroup. Let $P$ be a $C^\infty$ principal $\rH^\C$-bundle and fix a reduction
to a principal $\rH$-bundle $P_{\rH}$.  \emph{Hitchin's self-duality equations} are
\begin{equation}
  \label{eq:hitchins-equations}
\begin{aligned}
  F(A)-[\varphi,\tau(\varphi)] &=0,\\
  \dbar_A\varphi &=0.
\end{aligned}  
\end{equation}
Here $A$ is a $\rH$-connection on $P_{\rH}$, $\dbar_A$ its associated $\dbar$-operator
and $\varphi\in\Omega^{1,0}(P_{\rH}[\fm^\C])$. The map $\tau\colon \Omega^{1,0}(P_{\rH}[\fm^\C])\to\Omega^{0,1}(P_{\rH}[\fm^\C])$ is obtained
by combining the compact real structure on $\fg^\C$ with conjugation
on the form component.

A pair $(A,\varphi)$ gives a corresponding $\G$-Higgs
bundle structure $(\dbar_A,\varphi)$ on $P$; we denote the
corresponding $\G$-Higgs bundle by $(\cE_A,\varphi)$. Conversely,
given a $\G$-Higgs bundle $(\cE_A,\varphi)$, where the holomorphic
bundle $\cE_A$ is defined by $\dbar_A$, one obtains a pair
$(A,\varphi)$ by taking $A$ to be the Chern connection associated to
$\dbar_A$ via the fixed reduction $P_{\rH}\subset P$.
The Hitchin--Kobayashi correspondence for $\G$-Higgs bundles
\cite{HiggsPairsSTABILITY} says that the $\G$-Higgs bundle is
polystable if and only if there is a structure $(\dbar_A,\varphi)$ in
its $\cG_{\rH^\C}$-orbit such that the corresponding pair
$(A,\varphi)$ solves Hitchin's equations. Moreover, this pair is
unique up to $\cG_{\rH}$-gauge transformations, where
$\mathcal{G}_{\rH}$ denotes the gauge group of $\rH$-gauge
transformations of $P_{\rH}$.

We recall the following  alternative point of view. Instead of
fixing a reduction of the principal $\rH^\C$-bundle $P$, we can
consider a fixed structure of $\G$-Higgs bundle $(\dbar_A,\varphi)$
and consider \eqref{eq:hitchins-equations} as equations for a
reduction of structure group to $\rH\subset\rH^\C$, usually known as a
\emph{a harmonic metric} . The
Hitchin--Kobayashi correspondence then says that such a reduction
exists if and only if $(\dbar_A,\varphi)$ defines a polystable
$\G$-Higgs bundle.

The space $\mathcal{A}$ of $\rH$-connections on $P_{\rH}$ is an affine
space modeled on $\Omega^1(P_{\rH}[\fh])$.  Let
$\mathcal{C}\subset \mathcal{A}\times\Omega^{1,0}(P_{\rH}[\fm^\C])$ denote
the configuration space of solutions to Hitchin's equations
(\ref{eq:hitchins-equations}). As a set, the moduli
space of solutions to Hitchin's self-duality equations is
\begin{displaymath}
  \mathcal{M}^a_{\rH}(\G) = \mathcal{C}/\mathcal{G}_{\rH},
\end{displaymath}
where $a$ is the topological type. We shall denote by $\cM_{\rH}(\G)$ the
union of the moduli spaces $\mathcal{M}^a_{\rH}(\G)$ over all
topological types $a$.
In order to give the moduli space a topology,
suitable Sobolev completions must be used in standard fashion; see
\cite{atiyah-bott:1982}, and also \cite[Sec.~8]{hausel-thaddeus:2004}
where the straightforward adaptation to Higgs bundles is discussed in
the case $\G=\GL(n,\C)$. The moduli space $\cM_{\rH}(\G)$ then
becomes a Hausdorff topological space.

The Hitchin--Kobayashi correspondence can now be stated as saying that
the map
\begin{equation}
\label{eq:MH-cong-M}
\begin{aligned}
  \cM_{\rH}(\G) &\xrightarrow{\cong} \mathcal{M}(\G),\\
  (A,\varphi) &\mapsto (\dbar_A,\varphi)
\end{aligned}
\end{equation}
is a bijection. It follows from the constructions that it is in fact
a homeomorphism. Here and below, in analogy with
Notation~\ref{not:equivalence}, we do no distinguish notationally
between a pair $(A,\varphi)$ and its gauge equivalence class.

The moduli space $\cM_{\rH}(\G)$ can be given additional
structure by considering the deformation complex
\begin{equation}
  \label{eq:def-complex-equations}
  \Omega^0(P_{\rH}(\fh))
  \xrightarrow{d_0}
  \Omega^1(P_{\rH}(\fh))\times\Omega^{1,0}(P_{\rH}[\fm^\C])
  \xrightarrow{d_1}
  \Omega^2(P_{\rH}(\fh))\times\Omega^{1,1}(P_{\rH}[\fm^\C]).
\end{equation}
The operator $d_0$ is given by the infinitesimal action of the gauge
group and the operator $d_1$ is obtained by linearizing Hitchin's
equations; the fact that $d_1\circ d_0=0$ follows because
$(A,\varphi)$ is a solution. Denote the $i$th
cohomology group of this complex by $H^i_{(A,\varphi)}$.

\begin{proposition}
  \label{prop-def-complex-isomorphism}
  Let $(A,\varphi)$ be a solution to Hitchin's equations and let
  $(\mathcal{E},\varphi)$ be the corresponding Higgs bundle. Then
  there are isomorphisms
  \begin{align*}
    H^0_{(A,\varphi)}\otimes\C
      &\cong \HH^0(C^\bullet(\mathcal{E},\varphi)),\\
    H^1_{(A,\varphi)}
      &\cong \HH^1(C^\bullet(\mathcal{E},\varphi)),\\
    H^2_{(A,\varphi)}
      &\cong \HH^2(C^\bullet(\mathcal{E},\varphi))
        \oplus H^0_{(A,\varphi)},
  \end{align*}
where $C^\bullet(\mathcal{E},\varphi)$ is the deformation
complex~\eqref{eq complex of sheaves}.
\end{proposition}

\begin{proof}
  The hypercohomology groups of the complex
  $C^\bullet(\mathcal{E},\varphi)$ can be calculated, using a Dolbeault
  resolution, as the cohomology groups of the complex 
  \begin{equation}
    \label{eq:C-deformation-complex}
  \Omega^0(P[\fh^\C])
    \xrightarrow{\delta_0}
  \Omega^{0,1}(P[\fh^\C])\times\Omega^{1,0}(P[\fm^\C])
    \xrightarrow{\delta_1}
  \Omega^{1,1}(P[\fm^\C]),
\end{equation}
where the differentials are constructed combining the adjoint action
of $\varphi$ with $\dbar_A$.
  The proposition now follows essentially as in
  \cite[Sec.~6.4.2]{donaldson-kronheimer:1990} (which gives the
  analogous comparison between the deformation complexes for solutions
  to the anti-self duality equations and holomorphic vector bundles on
  a complex surface) using the K\"ahler identities and the bundle
  isomorphisms
  \begin{align*}
    \Omega^{0,1}(P[\fh^\C]) &\cong \Omega^{1}(P_{\rH}[\fh])\\
    \Omega^{0}(P[\fh^\C]) &\cong
      \Omega^{0}(P_{\rH}[\fh])\oplus\Omega^{2}(P_{\rH}[\fh]).
  \end{align*}
\end{proof}

\begin{proposition}
  Let $(A,\varphi)\in \cM_{\rH}(\G)$ and let
  $(\mathcal{E}_A,\varphi)$ be the corresponding polystable $\G$-Higgs
  bundle. Then the following statements are equivalent:
  \begin{enumerate}
  \item $H^0_{(A,\varphi)} =0$ and $H^2_{(A,\varphi)}=0$.
  \item $\HH^0(C^\bullet(\mathcal{E}_A,\varphi))=0$ and
    $\HH^2(C^\bullet(\mathcal{E}_A,\varphi))=0$. 
  \item $(\mathcal{E}_A,\varphi)$ is stable as a $\G^\C$-Higgs bundle.
  \end{enumerate}
\end{proposition}

\begin{proof}
  The equivalence of the first two statements is immediate from
  Proposition~\ref{prop-def-complex-isomorphism}. The equivalence of
  the last two statements is also immediate in view of
  Remarks~\ref{remark stable open in polystable} and
  \ref{remark duality of H0 and H2 for complex groups}.
\end{proof}

\begin{definition}
Let $\mathcal{C}^s\subset\mathcal{C}$ denote the subspace
of pairs $(A,\varphi)$ such that $(\mathcal{E}_A,\varphi)$ is stable
as a $\G^\C$-Higgs bundle. Similarly, let $\mathcal{C}_{\C}^s \subset
\mathcal{A}\times\Omega^{1,0}(P_{\rH}[\fm^\C])$ denote the subspace
of pairs $(A,\varphi)$ such that $\dbar_A\varphi=0$ and
$(\mathcal{E}_A,\varphi)$ is stable as a $\G^\C$-Higgs bundle. 
Define $\mathcal{M}_{\rH}^s(\G)\subset \cM_{\rH}(\G)$  and
$\mathcal{M}^s(\G)\subset \mathcal{M}(\G)$ analogously.
\end{definition}

We note that $\mathcal{C}_{\C}^s $ is
an infinite dimensional K\"ahler manifold whose K\"ahler structure is induced
from the ambient space $\mathcal{A}\times\Omega^{1,0}(P_{\rH}[\fm^\C])$.

Let $\Gamma_{(A,\varphi)}\subset\mathcal{G}_{\rH}$ denote the
stabilizer of a solution $(A,\varphi)$ to Hitchin's equations. This is
a compact Lie group with Lie algebra $H^0_{(A,\varphi)}$
\cite{HiggsPairsSTABILITY}. The standard gauge theoretic
construction of the moduli space can now be summarized as follows.

\begin{proposition}
  \label{prop:local-model-gauge}
The subspace of $\mathcal{C}$ where
$H^2_{(A,\varphi)}=0$ is a smooth infinite
dimensional manifold. Moreover, for $(A,\varphi)$
with $H^2_{(A,\varphi)}=0$ a neighbourhood of the corresponding point in the 
moduli space is modeled on a neighbourhood of zero in
$H^1_{(A,\varphi)}$ modulo the action of $\Gamma_{(A,\varphi)}$. If
additionally $H^0_{(A,\varphi)} =0$, then $\Gamma_{(A,\varphi)}$ is
finite. Thus $\mathcal{M}^s_{\rH}(\G)$ is a
K\"ahler orbifold with K\"ahler form induced from $\mathcal{C}_{\C}^s$.
\end{proposition}

\begin{remark}
  \label{rem:symplectic-quotient}
  The action of $\mathcal{G}_{\rH}$ on $\mathcal{C}_{\C}^s$ is
  Hamiltonian with moment map
  $\mu(A,\varphi)= F(A)-[\varphi,\tau(\varphi)]$.
  Hence
  the moduli space $\mathcal{M}^s(G)$ can be viewed as the
  infinite dimensional symplectic quotient
  \begin{displaymath}
    \mathcal{M}^s(G) = \mu^{-1}(0)/\mathcal{G}_{\rH} \cong
    \mathcal{C}_{\C}^s / \mathcal{G}_{\rH^\C}.
  \end{displaymath}
  The isomorphism comes from the Hitchin--Kobayashi correspondence,
  which can thus be viewed as an infinite dimensional Kempf--Ness
  correspondence. Note that the K\"ahler form on $\mathcal{C}_{\C}^s$
  restricts to a $2$-form on $\mathcal{C}^s$ which is non-degenerate
  in directions transverse to the $\mathcal{G}_{\rH}$-orbits ---
  indeed this is just the pullback of the K\"ahler form on
  $\mathcal{M}^s(\G)$.
\end{remark}

The following was proved in \cite{HiggsPairsSTABILITY}. It is
analogous to the decomposition of a polystable  vector bundle into a
direct sum of stable ones, and plays a central role in the proof of the
Hitchin--Kobayashi correspondence.

\begin{proposition}
  \label{prop:JH}
  Let $(\mathcal{E},\varphi)$ be a polystable $\G$-Higgs
  bundle. Then there is a real reductive subgroup $\G'\subset\G$ and a
  \emph{Jordan--H\"older reduction} of $(\mathcal{E},\varphi)$
  to a stable $\G'$-Higgs bundle $(\mathcal{E}',\varphi')$. The
  Jordan--H\"older reduction is unique up to isomorphism. Moreover,
  the solution to Hitchin's equations on $(\mathcal{E}',\varphi')$
  induces the solution on $(\mathcal{E},\varphi)$.
\end{proposition}


Next we recall Hitchin's method
\cite{selfduality,liegroupsteichmuller} for studying the topology of
$\mathcal{M}(\G)$ using gauge theoretic methods, and explain how to
translate it to the holomorphic point of view. Alternatively
one could work exclusively using the holomorphic point of view, using
Simpson's adaptation in \cite[Sec.~11]{SimpsonModuli2}.


Similarly to the holomorphic action of $\C^*$ on $\mathcal{M}(G)$
defined in Section~\ref{sec:basic-properties-action}, there is an
action of $S^1$ on $\mathcal{A}\times\Omega^{1,0}(P_{\rH}[\fm^\C])$ given by
\begin{displaymath}
  e^{i\theta}\cdot(A,\varphi) = (A,e^{i\theta}\varphi).
\end{displaymath}
This action clearly preserves the subspaces $\mathcal{C}^s$, $\mathcal{C}$ and
$\mathcal{C}^s_{\C}$, and it descends to $\cM_{\rH}(\G)$.

\begin{proposition}
  \label{prop:S1-equivariance}
  Let $S^1$ act on $\mathcal{M}(G)$ by restriction of the
  $\C^*$-action defined above. Then the following statements hold.
  \begin{enumerate}
  \item The bijection $\cM_{\rH}(\G) \to \mathcal{M}(\G)$
    defined in (\ref{eq:MH-cong-M}) is $S^1$-equivariant.
  \item The class of $(A,\varphi)$ in $\cM_{\rH}(\G)$ is fixed
    under the $S^1$-action if and only if the class of the
    corresponding Higgs bundle $(\mathcal{E}_A,\varphi)$ in
    $\mathcal{M}(\G)$ is fixed under the $\C^*$-action.
  \end{enumerate}
\end{proposition}

\begin{proof}
  Statement (1) is clear. Statement (2) is a consequence of the
  Hitchin--Kobayashi correspondence.
\end{proof}

Since the vector bundle $P[\fm^\C] \cong P_{\rH}[\fm^\C]$ has a Hermitian
metric coming from the reduction of structure group to $\rH$, one can
define the \emph{Hitchin function}:
\begin{equation}
  \label{EQ Hitchin Function gauge}
  f\colon\cM_{\rH}(\G)\to\R,\quad (A,\varphi)\mapsto\int_X||\varphi||^2.
\end{equation}
We shall abuse notation and denote by the same letter the map
$f\colon\cM(\G)\to\R$ induced via the identification
\eqref{eq:MH-cong-M}. 
Using Uhlenbeck's weak compactness theorem, Hitchin \cite{selfduality}
showed that the map $f$ is proper. Thus, as noted in
Section~\ref{sec:basic-properties-action}, the Hitchin function can be
used to study the connected components of the moduli space of
$\G$-Higgs bundles.


The following is central for identifying local minima of $f$.

\begin{lemma}
  \label{lem:minima-fixed-HH-vanishing}
  Let $(A,\varphi)\in \mathcal{M}^s_{\rH}(\G)$. If $(A,\varphi)$ is a
  local minimum of $f$, then it is a fixed point of the
  $S^1$-action. Equivalently, the corresponding Higgs bundle
  $(\cE_A,\varphi)\in\cM^s(\G)$ is a fixed point of the $\C^*$-action.
\end{lemma}

\begin{proof}
  On the smooth locus of $\cM_{\rH}(\G)$, the $S^1$-action is
  Hamiltonian with respect to the K\"ahler form and the function $f$
  (suitably normalized) is a moment map for this action (see
  \cite{selfduality,liegroupsteichmuller}). This means that, when
  multiplied by $\sqrt{-1}$, the vector field generating the
  $S^1$-action is the gradient of $f$ and, therefore, critical points
  of $f$ are exactly the fixed points of the $S^1$-action. This proves
  the proposition when $\Gamma_{(A,\varphi)}$ is trivial.

  For a general $(A,\varphi)\in \mathcal{M}^s_{\rH}(\G)$ we can argue
  on the smooth manifold $\mathcal{C}^s\subset
  \mathcal{C}^s_{\C}$. Indeed, by its very definition, the function
  $f$ lifts to the infinite dimensional K\"ahler manifold
  $\mathcal{C}^s_{\C}$ and it is a moment map for the $S^1$-action
  there. Thus, in view of Remark~\ref{rem:symplectic-quotient}, and in
  a similar way to the argument of the preceding paragraph, it follows
  that $(A,\varphi)$ is a critical point of $f$ restricted to
  $\mathcal{C}^s$ if and only if its $\mathcal{G}_{\rH}$-gauge
  equivalence class is fixed by the $S^1$-action.
%
%
\end{proof}


We have the following useful observation. Let $\G'\subset\G$ be a
reductive subgroup (we take this to include the choice of compatible
Cartan data). Then a solution $(A,\varphi)$ to Hithin's equations for
$\G'$ on a principal $\rH'$-bundle induces a solution for $\G$ on the
$\rH$-bundle obtained by extension of structure group. Hence we have a
well defined map
\begin{displaymath}
  \cM(\G')\longrightarrow\cM(\G)
\end{displaymath}
which is clearly compatible with the respective Hitchin
functions. This leads immediately to the following result.

\begin{lemma}
  \label{lem:reducing-minima}
  Let $\G'\subset\G$ be a reductive subgroup. Suppose
  $(\cE,\varphi)$ is a $\G$-Higgs bundle which reduces to
  a $\G'$-Higgs bundle. If $(\cE,\varphi)$ is a minimum of the Hitchin
  function on $\cM(\G)$ then it is a minimum of the Hitchin function
  on $\cM(\G')$.
\end{lemma}

A solution $(A,\varphi)$ to Hitchin's equations is called
\emph{simple} if its stabilizer $\Gamma_{(A,\varphi)}$ is trivial. The
following proposition is simple to check.

\begin{proposition}
  \label{prop:gauge-theta}
  Suppose that $(A,\varphi)\in \cM_{\rH}(\G)$ is a fixed point for the
  $S^1$-action. Then for each $e^{i\theta}$ there is a gauge
  transformation $g(\theta)\in\mathcal{G}_{\rH}$ such that
  \begin{displaymath}
    g(\theta)\cdot(A,\varphi) = (A,e^{i\theta}\varphi).
  \end{displaymath}
  The gauge transformation $g(\theta)$ is determined up to an element
  of the stabilizer $\Gamma_{(A,\varphi)}$. Moreover, if $(A,\varphi)$ is simple, then $e^{i\theta} \mapsto g(\theta)$ defines a group
homomorphism $S^1\to\mathcal{G}_{\rH}.$ 
\end{proposition}

\begin{proposition}
Suppose $(A,\varphi)\in \cM_{\rH}(\G)$ is a fixed point for the
$S^1$-action. If $(A,\varphi)$ is simple, then there is an induced action of $S^1$ on
$H^1_{(A,\varphi)}$. 
\end{proposition}

\begin{proof}
  For each $e^{i\theta}$, the derivative of its action on
  $\mathcal{C}$ defines a map
  $H^1_{(A,\varphi)}\to H^1_{(A,e^{i\theta}\varphi)}$. Composing with
  the inverse of the derivative of the unique gauge transformation
  $g(\theta)$ from Proposition~\ref{prop:gauge-theta} we get a well
  defined map
  $H^1_{(A,\varphi)}\to H^1_{(A,\varphi)}$. Using the fact that
  $\theta\to g(\theta)$ is a group homomorphism it is easy to see that
  this gives an action of $S^1$.
\end{proof}

If $(A,\varphi)$ has discrete stabilizer, then for each $\theta_0$ and
each choice of gauge transformation $g(\theta_0)$ as in
Proposition~\ref{prop:gauge-theta}, there is a unique smooth family
$g(\theta)$ defined in a neighborhood of $\theta_0$. Taking
$\theta_0=0$ and $g(0)$ to be the identity we get the following
result, by an argument similar to the proof of the preceding
proposition.

\begin{proposition}
\label{prop:infinitesimal-S1-action-stable}
Suppose $(A,\varphi)\in \cM_{\rH}^s(\G)$ is a fixed point for the
$S^1$-action. Then there is an induced local action of a neighborhood
of the identity in $S^1$ on
$H^1_{(A,\varphi)}$. In particular, there is an inifinitesimal
$S^1$-action on $H^1_{(A,\varphi)}$, and a well-defined
infinitesimal gauge transformation
$\psi=\frac{dg_\theta}{d\theta}\big\vert_{\theta=0}\in \Omega^0(P_{\rH}[\fh])$.
\end{proposition}

\begin{remark}
  \label{rem:psi-phi-iphi}
  Note that $[\psi,\varphi] = i\phi$ because $g(\theta)\cdot(A,\varphi) = (A,e^{i\theta}\varphi)$.
\end{remark}

Now fix a maximal torus $\ft\subset\fh$. Since any element of $\fh$ is
conjugate to an element in $\ft$, there is a point $p_0\in P_{\rH}$
with the property stated in the following proposition.

\begin{proposition}
\label{prop:fixed-hodge-stable}
Let $(A,\varphi)\in \cM_{\rH}^s(\G)$ be a fixed point and let $\psi\in \Omega^0(P_H[\fh])$ be the infinitesimal gauge transformation provided by Proposition~\ref{prop:infinitesimal-S1-action-stable}. Let $p_0\in P$ be such that the infinitesimal gauge transformation provided by Proposition~\ref{prop:infinitesimal-S1-action-stable} satisfies $\psi(p_0)\in\ft$. Define
\begin{displaymath}
  \rH_0 = Z_{\rH}(\psi(p_0))\subset\rH.
\end{displaymath}
Then there is a subbundle $P_{\rH_0}\subset P_{\rH}$ which gives a
reduction of structure group to $\rH_0$.
\end{proposition}

\begin{proof}
Define
\begin{displaymath}
  P_{\rH_0}=\{p\in P_{\rH}\suchthat \psi(p)\in\ft\} \subset
  P_{\rH}.
\end{displaymath}
Let $\psi(p)\in\ft$. A point $\Ad(h)(\psi(p))=\psi(ph^{-1})$ in the
adjoint orbit of $\psi(p)$ lies in $\ft$ if and only if $h\in Z_{\rH}(\psi(p))$. Moreover, this centralizer does not depend on the choice of $\psi(p)$ in the adjoint orbit, as long as $\psi(p)$ lies in $\ft$. We therefore have an identification
of the fiber $P_{H_0,x}$ of $P_{H_0}$ over $x=\pi(p_0)\in X$:
\begin{align*}
  \rH_0 &\xrightarrow{\cong} P_{H_0,x},\\
  c &\mapsto p_0\cdot c
\end{align*}
where the action comes from the right action of $H$ on $P_{H_0}$.

Now note that, since $d_A\psi=0$, the eigenvalues for the action of
$\psi$ on $P_{\rH}[\fh]$ are constant. Hence the orbit in $\fh$ of
$\psi(p)$ under the adjoint action of $\rH$ is independent of
$p\in P_{\rH}$. It follows that the centralizer used in the preceding
paragraph is the same for all fibers of $P_{\rH}$ and, therefore, the
construction globalizes to show that $P_{\rH_0}\subset P_{\rH}$
defines a reduction of structure group, as we wanted.
\end{proof}

\begin{remark}
Since the reduction $P_{\rH_0}\subset P_{\rH}$ just constructed only depends on the choice of the maximal torus $\ft\subset\fh$, it is unique up to conjugation by $\rH$.
\end{remark}

\begin{proposition}
  \label{prop:weight-decomposition-fixed-H}
  Suppose $(A,\varphi)\in \cM_{\rH}^s(\G)$ is a fixed point for the
  $S^1$-action. Then there is a weight decomposition into
  $ik$-eigenspaces for the adjoint action of $\psi$ on the Lie algebra
  bundles $P_H[\fh^\C]$ and $P_H[\fm^\C]$:
  \begin{displaymath}
    P_\rH[\fh^\C] = \bigoplus_k P_\rH[\fh^\C]_{k}
    \quad\text{and}\quad
    P_\rH[\fm^\C]  = \bigoplus_k P_\rH[\fm^\C]_{k},
  \end{displaymath}
  where $\varphi\in H^0(P_H[\fm^\C]_1\otimes K)$ and
  $P_\rH[\fh^\C]_{0}$ is identified with the adjoint bundle
  $P_{\rH_0}[\fh^\C_0]$. 
\end{proposition}

\begin{proof}
This is immediate from Proposition~\ref{prop:fixed-hodge-stable} --- indeed, taking the weight space decomposition $\fh^\C=\bigoplus \fh^\C_k$ for the adjoint action of $\psi(p_0)$ we have $P_\rH[\fh^\C]_{k} = P_{\rH_0}[\fh^\C_{k}]$, and similarly for $\fm^\C$. The fact that $\varphi$ has weight one follows from Remark~\ref{rem:psi-phi-iphi}.
\end{proof}


\begin{remark}
  For any fixed $(A,\varphi)$ in the moduli space we can use the
  Jordan--H\"older reduction to a stable $\G'$-Higgs bundle to get a
  reduction of structure group as in
  Proposition~\ref{prop:fixed-hodge-stable}. However, the weight
  decomposition of Proposition~\ref{prop:weight-decomposition-fixed-H}
  is, in general, no longer well defined. This is because the center
  of the maximal compact $\rH'\subset \G'$ may act non-trivially on
  the complement of ${\fg'}^\C$ in $\fg^\C$.
\end{remark}

For a $(\cE, \varphi)\in\cM^ s(\G)$ which is fixed under the
$\C^*$-action, the weight decomposition from Proposition~
\ref{prop:weight-decomposition-fixed-H} translates into
\begin{equation}
  \label{eq:weight-decomposition-adE}
  \cE[\fg^\C] = \cE[\fh^\C] \oplus \cE[\fm^\C]
  = \bigoplus \cE[\fh^\C]_k \oplus \bigoplus \cE[\fm^\C]_k
\end{equation}
with $\cE[\fh^\C]_{k} = P_\rH[\fh^\C]_{k}$ and $\cE[\fm^\C]_{k} =
P_\rH[\fm^\C]_{k}$, and where $\varphi\in H^0(\cE[\fm^\C]_1\otimes K)$.
This gives a decomposition $C^\bullet(\cE,\varphi) = \bigoplus  C^\bullet_k(\cE,\varphi)$ of the
deformation complex \eqref{eq complex of sheaves}, where
\begin{equation}
  \label{eq:Ck-complex}
  C^\bullet_k(\cE,\varphi):
  \cE[\fh^\C]_{k}\xrightarrow{\ad_\varphi} \cE[\fm^\C]_{k+1}\otimes K.
\end{equation}


\begin{proposition}
\label{prop:gauge-complex-Ck}
Suppose $(A,\varphi)\in \cM_{\rH}^s(\G)$ is a fixed point for the
$S^1$-action. 
Let 
\begin{math}
  H^1_{(A,\varphi)} = \bigoplus H^1_{(A,\varphi),k}
\end{math}
be the decomposition into
$ik$-eigenspaces for the infinitesimal $S^1$-action given by
Proposition~\ref{prop:infinitesimal-S1-action-stable}.
Then there are canonical isomorphisms
\begin{displaymath}
  H^1_{(A,\varphi),k} \cong \HH^1(C^\bullet_k(\mathcal{E}_A,\varphi)).
\end{displaymath}
\end{proposition}

\begin{proof}
In a similar way to Proposition~\ref{prop:infinitesimal-S1-action-stable},  there is an infinitesimal $S^1$-action on  $\HH^1(C^\bullet(\mathcal{E}_A,\varphi))$ and, clearly, the isomorphism $H^1_{(A,\varphi)} \cong \HH^1(C^\bullet(\mathcal{E}_A,\varphi))$ of Proposition~\ref{prop-def-complex-isomorphism} is $S^1$-equivariant. Thus there is a weight space decomposition $\HH^1(C^\bullet(\mathcal{E}_A,\varphi))=\bigoplus \HH^1(C^\bullet(\mathcal{E}_A,\varphi))_k$ with $H^1_{(A,\varphi),k} \cong \HH^1(C^\bullet(\mathcal{E}_A,\varphi))_k$. It remains to see 
that $\HH^1(C^\bullet(\mathcal{E}_A,\varphi))_k \cong
\HH^1(C^\bullet_k(\mathcal{E}_A,\varphi))$ and this is an easy check
using the induced weight decomposition of the Dolbeault resolution \eqref{eq:C-deformation-complex}.
\end{proof}



We shall use the subscript ``+'' for the direct sums of subspaces
with $k>0$.

\begin{lemma}
  \label{lem:local-minima-HH-plus-vanishes}
  Let $(A,\varphi)\in \mathcal{M}^s_{\rH}(\G)$. If $(A,\varphi)$ is a
  fixed point of the $S^1$-action, then it is a local minimum of the
  Hitchin function if and only if $H^1_{(A,\varphi),k}=0$ for all
  $k>0$. Equivalently, a fixed point $(\cE,\varphi)\in\cM^s(\G)$ for
  the $\C^*$-action is a local minimum of the Hitchin function if and
  only if $\HH^1(C^\bullet_k(\mathcal{E},\varphi))=0$ for all $k>0$.
\end{lemma}

\begin{proof}
  Hitchin \cite{selfduality,liegroupsteichmuller} showed that on the
  smooth locus of $\cM_{\rH}(\G)$, the subspace $H^1_{(A,\varphi),k}$
  can be identified with the $-k$-eigenspace for the Hessian of
  $f$. The extension to points of $\cM_{\rH}(\G)$ which are orbifold
  singularities follows as in the proof of
  Lemma~\ref{lem:minima-fixed-HH-vanishing}. The equivalence of the
  statement for $(\cE,\varphi)\in\cM^s(\G)$ follows from
  Proposition~\ref{prop:gauge-complex-Ck}.
\end{proof}



We shall also need to show that certain $\G$-Higgs bundles
which do not satisfy the hypothesis of Proposition~\ref{prop: minima
  criteria} are not local minima of $f$. To this end we have the
following result, analogous to a criterion of
Simpson \cite [Lemma~11.8]{SimpsonModuli2}.

\begin{lemma}
  \label{lem:non-minimum-infinity-limit}
  Let $(\cE_0,\varphi_0)\in \mathcal{M}(G)$ be a fixed point of the
  $\C^*$-action. Suppose there exists a semistable $\G$-Higgs bundle
  $(\cE,\varphi)$, which is not $\cS$-equivalent to $(\cE_0,\varphi_0)$,
  and such that $\lim_{t\to\infty}(\cE,t\varphi)=(\cE_0,\varphi_0)$ in
  $\mathcal{M}(G)$. Then $(\cE_0,\varphi_0)$ is not a local minimum of
  $f$.
\end{lemma}

\begin{proof}
  Replacing $(\cE,\varphi)$ with the polystable representative of its
  $\cS$-equivalence class, we may assume that it is polystable.
  Note also that $(\cE,\varphi)$ cannot be a fixed point of the
  $\C^*$-action.
  
  Consider first the case when $(\cE,\varphi)$ is stable. Then, as in
  the proof of Lemma~\ref{lem:minima-fixed-HH-vanishing}, we can use
  the moment map interpretation of $f$ to deduce that the function
  $\R_{>0}\to\R$ defined by $t\mapsto f(\cE,t\varphi)$ is strictly
  increasing as $t$ tends to infinity.
  For the general case, consider the Jordan--H\"older reduction of
  $(\cE,\varphi)$ given by
  Proposition~\ref{prop:JH}. This is a stable $\G'$-Higgs bundle for some
  $\G'\subset \G$ and cannot be fixed under the $\C^*$-action, since
  otherwise $(\cE,\varphi)$ would also be  fixed. Since the natural
  map $\mathcal{M}(\G')\to\mathcal{M}(\G)$ is $\C^*$-equivariant and
  compatible with the respective Hitchin functions, the result
  follows by the same argument as in the previous paragraph.
\end{proof}






\bibliographystyle{plain}
\bibliography{mybib.bib}

\end{document}